\documentclass[a4paper,reqno,11pt]{amsart}

\usepackage{fullpage}

\usepackage{amssymb}
\usepackage{dsfont}

\usepackage[utf8]{inputenc} 
\usepackage[T1]{fontenc} 

\usepackage{hyperref}
\usepackage[nobysame,alphabetic,initials,msc-links]{amsrefs}

\DefineSimpleKey{bib}{how}

\renewcommand{\eprint}[1]{#1}
\BibSpec{misc}{%
  +{}{\PrintAuthors}  {author}
  +{,}{ \textit}      {title}
  +{,}{ }             {how}
  +{}{ \parenthesize} {date}
  +{,} { available at \eprint}        {eprint}
  +{,}{ available at \url}{url}
  +{,}{ }             {note}
  +{.}{}              {transition}
}

\numberwithin{equation}{section}

\theoremstyle{plain}
\newtheorem{thm}{Theorem}[section]
\newtheorem{prop}[thm]{Proposition}
\newtheorem{lemma}[thm]{Lemma}
\newtheorem{cor}[thm]{Corollary}
\theoremstyle{definition}

\newtheorem{defn}[thm]{Definition}
\theoremstyle{remark}

\newtheorem{remark}[thm]{Remark}

\theoremstyle{plain}

\newcommand\bp{\begin{proof}}
\newcommand\ep{\end{proof}}

\newcommand{\un}{\mathds{1}}
\newcommand\dach{{\!\widehat{\ \ }}}

\newcommand\C{\mathbb{C}}
\newcommand\N{\mathbb{N}}
\newcommand\Q{\mathbb{Q}}
\newcommand\R{\mathbb{R}}
\newcommand\T{\mathbb{T}}
\newcommand\Z{\mathbb{Z}}

\newcommand{\BB}{\mathcal{B}}
\newcommand{\CC}{\mathcal{C}}
\newcommand{\G}{\mathcal{G}}

\newcommand{\LL}{\mathcal{L}}
\newcommand{\M}{\mathcal{M}}
\newcommand{\OO}{\mathcal{O}}
\newcommand{\SSS}{\mathcal{S}}

\newcommand{\UU}{\mathcal{U}}

\newcommand\Char{\operatorname{Char}}

\newcommand{\Cl}{\operatorname{Cl}}
\newcommand\Dom{\operatorname{dom}}
\newcommand{\Gu}{{\mathcal{G}^{(0)}}}
\newcommand{\Gxx}{\mathcal{G}^x_x}

\newcommand\Ind{\operatorname{Ind}}

\newcommand\MT{\operatorname{M}}
\newcommand\Stab{\operatorname{Stab}}
\newcommand\Sub{\operatorname{Sub}}

\newcommand\Per{\operatorname{Per}}
\newcommand\Prim{\operatorname{Prim}}
\newcommand\Ran{\operatorname{ran}}

\newcommand\prim{\mathrm{prim}}
\newcommand\sing{\mathrm{sing}}
\newcommand\SL{\operatorname{SL}}
\newcommand\eps{\varepsilon}
\newcommand\Iso{\operatorname{Iso}(\mathcal G)^{\circ}}
\newcommand\IsoG{\operatorname{Iso}(\mathcal G)}
\newcommand\IsoGx[1]{\operatorname{Iso}(\mathcal G_{\overline{[#1]}})}

\newcommand\ee{\nopagebreak\mbox{\ }\hfill$\diamond$}

\begin{document}

\title{The primitive spectrum of C$^*$-algebras of \'etale groupoids with abelian isotropy}

\date{May 3, 2024; revised March 3, 2025}

\author{Johannes Christensen}
\address{Department of Mathematics, KU Leuven, Belgium}
\email{johannes.christensen@kuleuven.be}

\author{Sergey Neshveyev}
\address{Department of Mathematics, University of Oslo, Norway}
\email{sergeyn@math.uio.no}

\thanks{J.C. is supported by the postdoctoral fellowship 1291823N of the Research Foundation Flanders. S.N. is partially supported by the NFR funded project 300837 ``Quantum Symmetry''.}

\begin{abstract}
Given a Hausdorff locally compact \'etale groupoid $\G$, we describe as a topological space the part of the primitive spectrum of $C^*(\G)$ obtained by inducing one-dimensional representations of amenable isotropy groups of $\G$. When $\G$ is amenable, second countable, with abelian isotropy groups, our result gives the description of $\Prim C^*(\G)$ conjectured by van~Wyk and Williams. This, in principle, completely determines the ideal structure of a large class of separable C$^*$-algebras, including the transformation group C$^*$-algebras defined by amenable actions of discrete groups with abelian stabilizers and the C$^*$-algebras of higher rank graphs. As an illustration we describe the primitive spectrum of the C$^*$-algebra of any row-finite higher rank graph without sources.
\end{abstract}

\maketitle

\section*{Introduction}
The primitive spectrum $\Prim A$ of a C$^{*}$-algebra $A$ consists of the ideals that can be realized as kernels of irreducible representations. When equipped with the Jacobson topology, this space contains crucial information about the C$^{*}$-algebra, as it completely determines the ideal structure of~$A$. 
A complete description of the Jacobson topology is often a difficult task even for C$^*$-algebras with a relatively simple representation theory, see, e.g.,~\citelist{\cite{BP}\cite{NT}}. Our goal in this paper is to describe the topological space $\Prim A$ for a fairly large class of groupoid C$^*$-algebras.

The groupoid C$^*$-algebras belong to what can be loosely called algebras of crossed product type. The ``Mackey machine'',
since its inception in the works of Clifford~\cite{MR1503352} and Mackey~\citelist{\cite{MR0031489}\cite{MR0098328}}, has been the main tool to study representations of such algebras. One of the biggest achievements of the theory is the proof of the Effros--Hahn conjecture, which states that every primitive ideal of a separable transformation group C$^*$-algebra $C_0(X)\rtimes G$ defined by an action $G\curvearrowright X$ of an amenable group is induced by an irreducible representation of one of the stabilizers~\cite{EH}. This conjecture was proved, in a generalized form, by Sauvageot~\cite{Sau} and Gootman--Rosenberg~\cite{GR}. Their techniques were then extended to groupoid crossed products by Renault~\cite{R} and Ionescu--Williams~\cite{IW}.

Therefore we know by now that if $\G$ is an amenable second countable Hausdorff locally compact \'etale groupoid, then as a set the primitive ideal space $\Prim C^*(\G)$ is a quotient of the set $\Stab(\G)^\prim$ of pairs $(x,J)$, where $x\in\Gu$ and $J\in\Prim C^*(\Gxx)$. Although this is a very powerful result, in order to completely understand the ideal structure of $C^*(\G)$ one still needs to solve two related problems: determine when two points in $\Stab(\G)^\prim$ have identical images under the induction map $\Ind\colon\Stab(\G)^\prim\to\Prim C^*(\G)$ and describe the Jacobson topology on $\Prim C^*(\G)$. Note that the problems are related, because two points have identical images in $\Prim C^*(\G)$ if and only if the closures of these images coincide.

Effros and Hahn themselves proved in~\cite{EH} that if an action $G\curvearrowright X$ of a discrete amenable group is free, so that $\Stab(G\ltimes X)^\prim=X$, then as a topological space $\Prim(C_0(X)\rtimes G)$ is homeomorphic to the space $(G\backslash\! X)^\sim$ of quasi-orbits of the action, that is, $(G\backslash\! X)^\sim$ is the quotient of $X$ such that two points $x,y\in X$ have identical images if and only if $\overline{Gx}=\overline{Gy}$. More generally, Williams proved in~\cite{MR0617538}  that if all stabilizers are contained in one abelian subgroup $H\subset G$, then $\Stab(G\ltimes X)^\prim$ can be given the topology of a quotient of $X\times\widehat H$ and $\Prim(C_0(X)\rtimes G)$ is homeomorphic to the quasi-orbit space $(G\backslash\!\Stab(G\ltimes X)^\prim)^\sim$.

It is then tempting to say that there should exist a natural topology on $\Stab(\G)^\prim$ such that in the amenable case the induction map $\Ind\colon\Stab(\G)^\prim\to\Prim C^*(\G)$ induces a homeomorphism~$\Ind^\sim$ of $(\G\backslash\!\Stab(\G)^\prim)^\sim$ onto $\Prim C^*(\G)$. Making sense of this for a large class of groupoids, beyond the already mentioned cases, has proved to be difficult. All available general results of this sort involve significant restrictions on the local structure of the isotropy bundle. These results cover, for example, transformation groupoids defined by proper actions~\cite{EE}, groupoids with isotropy groups $\Gxx$ that vary continuously in $x\in\Gu$~\citelist{\cite{MR0146297}\cite{MR2966476}} and groupoids with abelian isotropy groups that vary continuously except for ``jump discontinuities''~\cite{MR4395600}.

In the last paper (\cite{MR4395600}) van Wyk and Williams tried to formalize this problem for groupoids with abelian isotropy. They introduced a topology on the space $\Stab(\G)\dach$ of pairs $(x,\chi)$ ($x\in\Gu$, $\chi\in\widehat{\Gxx}$), which can be identified with $\Stab(\G)^\prim$, and cautiously wrote that, under the assumptions of second countability and  amenability, they expect the map $\Ind^\sim\colon(\G\backslash\!\Stab(\G)\dach)^\sim\to\Prim C^*(\G)$ to be a homeomorphism in most circumstances.

Almost at the same time Katsura~\cite{Kat} succeeded in describing the primitive spectrum for C$^*$-algebras of singly generated dynamical systems. These C$^*$-algebras can be defined as groupoid C$^*$-algebras associated with a partially defined local homeomorphism $\sigma\colon\Dom(\sigma)\subset X\to X$. On a superficial level the corresponding groupoids $\G_\sigma$ may seem similar to transformation groupoids $\Z\ltimes X$, but they are known to have a considerably more complicated isotropy structure. The class of C$^*$-algebras~$C^*(\G_\sigma)$ includes graph C$^*$-algebras, and so the results of Katsura subsume in particular earlier results on the ideal structure of Cuntz--Krieger algebras and their generalizations, see~\cite{HS} and the references there for the history of the problem.

From the groupoid point of view the main result of~\cite{Kat} can be interpreted as a description of the pre-images of closed subsets of $\Prim C^*(\G_\sigma)$ in $\Stab(\G_\sigma)\dach$. Katsura, however, does not use groupoids, working instead with C$^*$-correspondences, and it is not obvious (but is true, see Section~\ref{ssec:Katsura}) that this description agrees with the conjecture of van Wyk and Williams, according to which this should give the $\G_\sigma$-invariant closed subsets of $\Stab(\G_\sigma)\dach$.

Very recently, Brix, Carlsen and Sims~\cite{BCS} studied the topology on the primitive spectrum for the Deaconu--Renault groupoids $\G_T$ defined by $k$-tuples $T=(T_1,\dots,T_k)$ of commuting local homeomorphisms $T_i\colon X\to X$. For $k=1$ this gives the subclass of the groupoids $\G_\sigma$ discussed above such that~$\sigma$ is globally defined. The main result of~\cite{BCS} describes the topology on $\Prim C^*(\G_T)$ under the assumption of existence of ``harmonious families of bisections''. It is shown in \cite{BCS} that this assumption is weak enough to cover many examples of one or two commuting local homeomorphisms, but it remains unclear how often it is satisfied for $k\ge3$.  It should be said that it is not at all obvious (but is again true, see Section~\ref{ssec:harmonious}) that the results of~\cite{BCS} agree with the conjecture of van Wyk and Williams.

\smallskip

In this paper we prove that $\Ind^\sim\colon(\G\backslash\!\Stab(\G)\dach)^\sim\to \Prim C^*(\G)$ is a homeomorphism for all amenable second countable \'etale groupoids $\G$ with abelian isotropy groups. Our approach draws on the insight from~\cite{NS} and~\cite{CN2}, but is in itself elementary and relies only on basic properties of the Jacobson topology. In fact, it allows us to prove a more general result. Given an \'etale groupoid $\G$, with no additional assumptions of amenability or second countability, we introduce a topological space $\Char(\G)$ consisting of pairs $(x,\chi)$, where $x\in\Gu$ and $\chi\colon\Gxx\to\T$ is a character. Induction gives us a map $\Ind\colon\Char(\G)\to\Prim C^*(\G)$. Our main result says roughly that every convergent net $\Ind(x_i,\chi_i)\to\Ind(x,\chi)$ in $\Prim C^*(\G)$ comes, up to replacing $(x_i,\chi_i)$ by another point on its $\G$-orbit, from a converging net in $\Char(\G)$. When $\Gxx$ is amenable, the converse is also true. In the follow-up paper~\cite{CN4} we extend our methods to primitive ideals obtained by inducing finite dimensional irreducible representations and even some infinite dimensional ones.

\smallskip

The paper is organized as follows. In Section~\ref{sec:prelim} we fix our notation and quickly review basic facts about \'etale groupoids, Jacobson and Fell topologies, and quasi-orbit spaces.

In Section~\ref{sec:main} we introduce the topological space $\Char(\G)$. We first define the topology similarly to~\cite{MR4395600}, but then reformulate it in a form more amenable to analysis. It is actually this reformulated form that we found when we studied primitive spectra, and then realized that it is equivalent to the construction by van Wyk and Williams. We then discuss the induction map, prove our main result (Theorem~\ref{thm:main}) and draw some consequences.

Once one knows for a groupoid $\G$ with abelian isotropy that the spaces $(\G\backslash\!\Stab(\G)\dach)^\sim$ and $\Prim C^*(\G)$ are homeomorphic, it is in principle possible to completely understand the ideal structure of $C^*(\G)$: the ideals are in a one-to-one correspondence with the $\G$-invariant closed subsets of $\Stab(\G)\dach$. However, the topology on $\Stab(\G)\dach$ is complicated, and in practice getting a good grasp on how such $\G$-invariant sets look like can easily become a herculean task. In Section~\ref{sec:examples} we collect several classes of examples where the space $\Stab(\G)\dach$ has a bit more transparent structure and our results can be formulated in a more explicit form. These are first of all transformation groupoids defined by group actions with abelian stabilizers and groupoids injectively graded by an abelian group, which includes the Deaconu--Renault groupoids $\G_\sigma$, $\G_T$. In particular, the results from~\citelist{\cite{MR0617538}\cite{Kat}\cite{BCS}} we discussed above fall into one of these classes and we show how they can be relatively quickly deduced from our results (at least for the second countable spaces in the case of~\cite{Kat}).

The homeomorphism $(\G\backslash\!\Stab(\G)\dach)^\sim\cong \Prim C^*(\G)$ implies that $\Ind(x,\chi)=\Ind(y,\eta)$ if and only if $(x,\chi)$ and $(y,\eta)$ have the same $\G$-quasi-orbits in $\Stab(\G)\dach$. The groupoids injectively graded by abelian groups provide a large class of examples where this relation takes a relatively simple form: in Section~\ref{ssec:graded} we show that for such groupoids we have $\Ind(x,\chi)=\Ind(y,\eta)$ if and only if $x$ and $y$ have the same orbit closures in $\Gu$ and $\chi$ and $\eta$, viewed as characters of subgroups of the grading group, coincide on the essential isotropy groups at $x$ and $y$. This has actually already been shown in our previous paper \cite{CN2} (and for smaller classes of groupoids earlier in~~\citelist{\cite{MR0617538}\cite{SW}}) using the theory of isotropy fibers of ideals that we developed there. Here we use essentially the same ideas, but thanks to our description of the topology on $\Prim C^*(\G)$ the arguments become significantly shorter and, hopefully, more digestible. Therefore we get a relatively explicit ``catalogue'' of primitive ideals.

In Section~\ref{sec:Graph} we apply results of Section~\ref{sec:examples} to graph algebras. As was already mentioned, a complete description of the primitive spectrum of Cuntz--Krieger algebras of countable directed graphs was obtained by Hong and Szyma\'nski~\cite{HS} as a culmination of a long line of research. In Section~\ref{ssec:1-graphs} we give an equivalent description based on analysis of quasi-orbits of the canonical shift map on the path space of the graph. A key observation, a form of which is valid for all partially defined local homeomorphisms of second countable spaces, is that every quasi-orbit is represented either by an aperiodic path or by a periodic path with discrete orbit. In both cases the orbit closures are then not too difficult to understand in graph-theoretic terms.

In Section~\ref{subsecHRG} we consider higher rank graphs. Although formally our general results completely describe the primitive spectra of the corresponding Cuntz--Krieger algebras, a lot of work is still needed to formulate this description in terms of the underlying graphs. In this paper we carry out this work for higher rank graphs that are row-finite and have no sources. This is the most tractable class of possibly infinite higher rank graphs, and most of the research has been concentrated on~it.

A parameterization of the primitive ideals in the spirit of~\cite{HS} in this case has been already obtained in~\cite{MR3150171}. We show that it can also be quickly deduced from our results and give two complementary descriptions of the topology on the primitive spectrum. Our first result describes the open subsets of the spectrum. It is inspired by a result in~\cite{BCS} obtained under the assumption of existence of harmonious families of bisections, which is always satisfied for $2$-graphs. Our description is more explicit and does not require any extra assumptions. However, it is probably still not optimal, since in general it involves conditions on uncountably many paths.
Our second result describes convergence in the primitive spectrum and involves only finite paths. For usual row-finite graphs without sources this gives again the result of Hong and Szyma\'nski~\cite{HS}.

\bigskip

\section{Preliminaries} \label{sec:prelim}

\subsection{\'Etale groupoids and their \texorpdfstring{C$^*$}{C*}-algebras}\label{ssec:groupoids}
Throughout the paper we work with Hausdorff locally compact \'etale groupoids~$\G$. In the following we will briefly introduce our notation for such groupoids, and we refer the reader to \cite{SSW} for more background information. For a number of results we have to assume that $\G$ is in addition second countable or amenable, but we will make these assumptions explicitly every time they are needed.

\smallskip

As usual we denote by $s\colon\G\to\Gu$ and $r\colon\G\to\Gu$ the source and range maps. The assumption of \'etaleness means that these maps are local homeomorphisms. Recall that a subset of $\G$ on which $s$ and $r$ are injective is called a bisection of $\G$. We let $\G_x:=s^{-1}(x)$, $\G^x:=r^{-1}(x)$ and we define the isotropy group at $x\in \Gu$ to be $\Gxx:=\G_x\cap\G^x$. We denote by $[x]:=r(\G_x)$ the $\G$-orbit of $x$ in $\Gu$. The isotropy bundle is defined by
$$
\IsoG:=\{g\in\G:s(g)=r(g)\}.
$$
This is a closed subgroupoid of $\G$.

\smallskip

The space $C_c(\G)$ of continuous compactly supported functions on $\G$ is a $*$-algebra with convolution product
\begin{equation*} \label{eprod}
(f_{1}*f_{2})(g) := \sum_{h \in \G^{r(g)}} f_{1}(h) f_{2}(h^{-1}g)
\end{equation*}
and involution by $f^{*}(g):=\overline{f(g^{-1})}$. If $f\in C_c(W)$ for an open bisection $W\subset\G$, then for every representation $\pi\colon C_c(\G)\to B(H)$ we have
$$
\|\pi(f)\|\le\|f\|_\infty.
$$
This implies that we can define a norm on the $*$-algebra $C_c(\G)$ by $\lVert f \rVert =\sup_{\pi}\|\pi(f)\|$, where the supremum is taken over all representations of $ C_c(\G)$. We denote by $C^*(\G)$ the C$^{*}$-algebra obtained by completing $C_c(\G)$ in this norm.

\smallskip

Take a point $x\in\Gu$ and a subgroup $S\subset\Gxx$. Then every unitary representation $\pi\colon S\to U(H)$ on a Hilbert space $H$ can be induced to a representation $\Ind\pi=\Ind^\G_S\pi$ of $C^*(\G)$ as follows. The underlying space $\Ind H$ of $\Ind\pi$ consists of the functions $\xi \colon  \G_{x} \to H$ such that
\begin{equation*}
\xi(gh)=\pi(h)^{*}\xi(g),\quad g \in \G_{x},\ h\in S,
\end{equation*}
and
\begin{equation*}
\sum_{g\in \G_{x}/S}\lVert \xi(g)\rVert^{2}<\infty .
\end{equation*}
The space $\Ind H$ is then a Hilbert space with the inner product
\begin{equation*}
(\xi_{1}, \xi_{2} ):=\sum_{g\in \G_{x}/S}( \xi_{1}(g), \xi_{2}(g)) ,\quad \xi_{1}, \xi_{2} \in \Ind H.
\end{equation*}
For $f\in C_c(\G)$ we have
\begin{equation}\label{eq:ind-rep}
\big((\Ind\pi)(f)\xi\big)(g) :=
\sum_{h\in \G^{r(g)}}f(h) \xi(h^{-1}g),\quad g\in\G_x,\ \xi\in\Ind H.
\end{equation}

We write out the following standard observation for future reference.

\begin{lemma} \label{lem:UE}
For every $g\in\G_x$, the right translation by $g$ defines a unitary equivalence between the representations $\Ind^\G_S\pi$ and $\Ind^\G_{gSg^{-1}}\pi(g^{-1}\cdot g)$.
\end{lemma}
\begin{proof}
Let $H_{x}$ be the Hilbert space underlying $\Ind \pi$ and let $H_{r(g)}$ be the Hilbert space underlying $\Ind \pi(g^{-1} \cdot g)$. The map $V\colon H_{r(g)} \to H_{x}$ defined by
\begin{equation*}
(V\xi)(g'):=\xi(g' g^{-1}) , \quad g'\in \G_{x}, \; \xi \in H_{r(g)},
\end{equation*}
is a unitary operator satisfying $V^{*} (\Ind^\G_S\pi)(\cdot) V = \Ind^\G_{gSg^{-1}}\pi(g^{-1}\cdot g)$.
\end{proof}

\subsection{Weak containment and the Jacobson topology}
By an ideal in a C$^{*}$-algebra we always mean a closed two-sided ideal. We refer the reader to \cite{Ped} for an in-depth treatment of primitive ideals. Given a C$^*$-algebra $A$, recall that the \emph{Jacobson topology} on its set of primitive ideals is the topology in which the closed sets have the form
$$
\operatorname{hull}(J):=\{I\in\Prim A: J\subset I\},
$$
where $J\subset A$ is an ideal. Therefore the closure of a set $\CC\subset\Prim A$ is $\text{hull}(\bigcap_{I\in\CC}I)$.

Given a nondegenerate representation $\pi\colon A\to B(H)$, denote by $\SSS_\pi(A)$ the collection of states on $A$ of the form $(\pi(\cdot)\xi,\xi)$, where $\xi\in H$ is a unit vector. Following \cite{MR0146681}, we say that a representation~$\pi$ is \emph{weakly contained} in a representation $\rho$, if $\SSS_\pi(A)$ is contained in the weak$^*$ closure of the convex hull of $\SSS_\rho(A)$. 
If $\pi$ has a cyclic unit vector $\xi$, it suffices to check that $(\pi(\cdot)\xi,\xi)$ lies in that closed convex hull.

\begin{lemma}[{\cite{MR0146681}*{Theorem~1.4}}]\label{lem:Fell}
Assume $A$ is a C$^*$-algebra, $\pi$ and $\pi_i$ ($i\in I$) are irreducible representations of $A$. Then the following conditions are equivalent:
\begin{enumerate}
\item $\ker\pi$ lies in the closure of $(\ker\pi_i)_{i\in I}$ in $\Prim A$;
\item $\pi$ is weakly contained in $\bigoplus_{i\in I}\pi_i$;
\item $\SSS_\pi(A)$ is contained in the weak$^*$ closure of $\bigcup_{i\in I}\SSS_{\pi_i}(A)$.
\end{enumerate}
\end{lemma}

From the equivalence of (1) and (2) we get the following well-known observation.

\begin{lemma}\label{lem:Jacobson}
Assume $A$ is a C$^*$-algebra, $\pi$ is an irreducible representation of $A$ and $(\pi_i)_i$ is a net of irreducible representations of $A$. Then $\ker\pi_i\to\ker\pi$ in $\Prim A$ if and only if for every subnet $(\pi_{i_j})_j$ the representation $\pi$ is weakly contained in $\bigoplus_j\pi_{i_j}$.
\end{lemma}

\begin{proof}
The ``only if'' implication follows by Lemma \ref{lem:Fell}, since a subnet of a convergent net converges towards the same limit. For the ``if'' implication, observe that if $U$ is a neighbourhood of $\ker\pi$ such that there does not exist an index $i_{0}$ with $\ker\pi_i \in U$ for all $i\geq i_{0}$, then $J:=\{i \; | \; \ker\pi_i \notin U \}$ defines a subnet such that $\pi$ is not weakly contained in $\bigoplus_j\pi_{i_j}$ by Lemma~\ref{lem:Fell}.
\end{proof}

\subsection{Fell topology}
Assume $X$ is a topological space. Denote by $\Cl(X)$ the set of closed subsets of $X$. The \emph{Fell topology} on $\Cl(X)$ is defined using as a basis the sets
$$
\UU(K;(U_i)^n_{i=1}):=\{A\in\Cl(X):A\cap K=\emptyset,\ A\cap U_i\ne\emptyset\ \ \text{for}\ \ i=1,\dots,n\},
$$
where $K\subset X$ is compact and $U_i\subset X$ are open. As is shown in~\cite{MR0139135}, the space $\Cl(X)$ is always compact. It is Hausdorff when~$X$ is locally compact.

If $G$ is a locally compact group, then the subset $\operatorname{Sub}(G)\subset\Cl(G)$ of closed subgroups of $G$ is closed, hence it is compact in the relative topology, which is called the \emph{Chabauty topology}.

We will mainly use the Fell topology for discrete spaces $X$. In this case a net $(C_{i})_{i}$ converges to $C\in\Cl(X)$ if and only if the indicator functions $\un_{C_{i}}$ converge to $\un_{C}$ pointwise, and all the above statements become straightforward to verify.

\subsection{\texorpdfstring{$T_0$}{T0}-ization}\label{ssec:T0} Recall that a topological space $X$ is $T_{0}$, when for any pair of distinct points in~$X$ there exists an open set containing only one of them. When $X$ is a topological space and~$R$ is an equivalence relation on $X$, we remind that the quotient space $X / R$ is equipped with the topology in which a subset of $X / R$ is open exactly when its pre-image under the quotient map $X\to X / R$ is open.

Following~\cite{EH}, for a topological space $X$, we denote by $X^\sim$ its \emph{$T_0$-ization}, also known as the \emph{Kolmogorov quotient} of $X$, the topological space obtained by identifying points of $X$ that have identical closures. If $p\colon X\to X^\sim$ is the quotient map and $F\subset X$ is a closed subset, then it follows by definition that $p^{-1}(p(F))=F$. Hence $p$ is both closed and open, and $p$ establishes a bijection between the closed subsets of $X$ and the closed subsets of $X^\sim$, equivalently, a bijection between the open subsets of $X$ and the open subsets of $X^\sim$. The space $X^\sim$ is a $T_0$-space, and every continuous map from $X$ into a $T_0$-space factors through a continuous map from~$X^\sim$.

The following lemma will be useful to recognize the spaces $X^\sim$.

\begin{lemma}\label{lem:T0}
Assume $X$ and $Y$ are topological spaces, with $Y$ a $T_0$-space, and $p\colon X\to Y$ is a surjective continuous map. Assume also that $R$ is an equivalence relation on $X$ satisfying the following properties:
\begin{enumerate}
  \item[(i)] if $x_1\sim_R x_2$, then $p(x_1)=p(x_2)$;
  \item[(ii)] if $U\subset X$ is an open set, then its $R$-saturation
  $$
  R(U):=\{x\in X\mid x\sim_R u \ \text{for some}\ u\in U\}
  $$
  is again open in $X$;
  \item[(iii)] if $x\in X$ and $A\subset X$ are such that $p(x)\in\overline{p(A)}$, then $x\in\overline{R(A)}$.
\end{enumerate}
Then $p$ defines a homeomorphism of $(X/R)^\sim$ onto $Y$. Moreover, the map $p\colon X\to Y$ is open, and we have $p(x_1)=p(x_2)$ if and only if $\overline{R(x_1)}=\overline{R(x_2)}$.
\end{lemma}


\bp
We start by proving that $p(x_1)=p(x_2)$ if and only if $\overline{R(x_1)}=\overline{R(x_2)}$. If $p(x_1)=p(x_2)$, then, for every $x\sim_R x_1$, we have $p(x)=p(x_2)$ by (i), hence $x\in\overline{R(x_2)}$ by (iii). Therefore $\overline{R(x_1)}\subset\overline{R(x_2)}$. For the same reason the opposite inclusion holds, so $\overline{R(x_1)}=\overline{R(x_2)}$. Conversely, if $\overline{R(x_1)}=\overline{R(x_2)}$, then
$$
p(x_1)\in p\big(\,\overline{R(x_2)}\,\big)\subset \overline{p(R(x_2))}=\overline{\{p(x_2)\}}.
$$
For the same reason $p(x_2)\in \overline{\{p(x_1)\}}$. As $Y$ is a $T_0$-space, it follows that $p(x_1)=p(x_2)$.

Since $p$ factors through $X/R$ and $Y$ is a $T_0$-space, we get a surjective continuous map $p^\sim\colon (X/R)^\sim\to Y$. Now we observe that if $\overline{R(x_1)}=\overline{R(x_2)}$, then the images of $x_1$ and $x_2$ in $X/R$ have the same closures, hence their images in $(X/R)^\sim$ are equal. It follows that $p^\sim$ is a bijection. It remains to show that $p$ is open, since then $p^\sim$ is open as well.

Take an open set $U\subset X$ and consider the set $F:=X\setminus R(U)$, which is closed by (ii). If $p(x)\in\overline{p(F)}$, then by (iii) we have $x\in\overline{R(F)}$. But the set $F$ is already $R$-saturated and closed, hence $x\in F$. This shows that the set $p(F)$ is closed and $p(U)\cap p(F)=\emptyset$. Hence $p(U)=Y\setminus p(F)$ is open.
\ep

We remark that once it is proved that $(X/R)^\sim\to Y$ is a homeomorphism, the fact that $p\colon X\to Y$ is open follows from openness of the quotient maps $X\to X/R$ and $X/R\to(X/R)^\sim$. Even more, we have the following property.

\begin{cor}\label{cor:open}
In the setting of Proposition~\ref{lem:T0}, the map $p$ establishes a bijection between the $R$-saturated open subsets of $X$ and the open subsets of~$Y$.
\end{cor}

\bp
By definition, the quotient map $X\to X/R$ establishes a bijection between the $R$-saturated open subsets of $X$ and the open subsets of~$X/R$. By properties of $T_0$-ization we also know that the map $X/R\to(X/R)^\sim$ defines a bijection between the open subsets of~$X/R$ and the open subsets of~$(X/R)^\sim$. By combining these two facts we get the result.
\ep

\begin{cor}
Assume $X$ is a topological space and $R$ is an equivalence relation on $X$ such that the $R$-saturation of every open set is open. Then $(X/R)^\sim$ is the quotient of $X$ obtained by identifying points that have identical closures of their $R$-equivalence classes.
\end{cor}

\bp
Consider the quotient maps $q\colon X\to X/R$ and $p\colon X\to(X/R)^\sim$. Assume $p(x)\in\overline{p(A)}$ for some $x\in X$ and $A\subset X$. Since the map $X/R\to(X/R)^\sim$ is closed, the set $\overline{p(A)}$ is the image of the closed set $\overline{q(A)}$. By the definition of $T_0$-ization it follows that $q(x)\in \overline{q(A)}$. Since~$q$ is open, $q(X\setminus\overline{R(A)})$ is an open set that does intersect $q(A)$. It follows that $x\in \overline{R(A)}$. Therefore we can apply Lemma~\ref{lem:T0} to $p\colon X\to(X/R)^\sim$ and conclude that $p(x_1)=p(x_2)$ if and only if $\overline{R(x_1)}=\overline{R(x_2)}$.
\ep

\bigskip

\section{Primitive ideals induced by characters}\label{sec:main}

\subsection{The spaces \texorpdfstring{$\Char(\G)$}{Char(G)}, \texorpdfstring{$\Stab(\G)\dach$}{Stab(G)} and the induction map}
Assume $\G$ is a Hausdorff locally compact \'etale groupoid. Denote by $\Char(\G)$ the set of pairs $(x,\chi)$, where $x\in\Gu$ and $\chi\colon\Gxx\to\T$ is a character. For every open set $U\subset\Gu$, compact set $K\subset\IsoG$ and open set $V\subset\T$, consider the subset $\OO(U,K,V)\subset\Char(\G)$ defined by
$$
\OO(U,K,V):=\{(x,\chi): x\in U,\ \chi(K\cap \Gxx)\subset V\}.
$$
Consider the topology on $\Char(\G)$ with a basis consisting of finite intersections of the sets $\OO(U,K,V)$. When $\G$ has abelian isotropy groups, we denote this space by $\Stab(\G)\dach$. The space $\Char(\G)$ is easily seen to be $T_1$, but in general it is non-Hausdorff, as will be discussed later.

The topological space $\Stab(\G)\dach$ for transformation groupoids defined by proper actions (with not necessarily abelian stabilizers) was introduced in~\cite{EE}. For (not necessarily \'etale) groupoids with abelian isotropy it was introduced in~\cite{MR4395600}, whose definition we follow. We remark that in \cite{MR4395600} this space is denoted by $\Stab(\G)$. Although this notation is lighter than $\Stab(\G)\dach$, it seems misleading to us, so we will use the latter one, which is also more in line with~\cite{EE}.

\smallskip

The topology on $\Char(\G)$ can also be described as follows.

\begin{lemma}\label{lem:Char-topology}
Fix a point $(x,\chi)\in\Char(\G)$. For every $g\in\Gxx$, choose an open bisection $W_g$ containing $g$. Then a base of open neighbourhoods of $(x,\chi)$ in $\Char(\G)$ is given by the sets $\UU_x^\chi(U,\eps,(W_g)_{g\in F})$ defined as follows, where $\eps>0$, $F\subset\Gxx$ is a finite set and $U$ is an open neighbourhood of $x$ in $\Gu$ such that $U\subset\bigcap_{g\in F}r(W_g)$: the set $\UU_x^\chi(U,\eps,(W_g)_{g\in F})$ consists of the points $(y,\eta)$ such that $y\in U$ and for every $g\in F$ we have either  $W_g\cap\G^y_y=\emptyset$, or $W_g\cap\G^y_y=\{h\}$ for some $h$ and $|\chi(g)-\eta(h)|<\eps$.
\end{lemma}

\bp
Let us show first that the sets $\UU_x^\chi(U,\eps,(W_g)_{g\in F})$  are open. Take a point $(y,\eta)\in \UU_x^\chi(U,\eps,(W_g)_{g\in F})$, so $y\in U$ and for every $g\in F$ we have either $W_g\cap\G^y_y=\emptyset$, or $W_g\cap\G^y_y=\{h_g\}$ for some $h_g$ and $|\chi(g)-\eta(h_g)|<\eps$. Consider the set $S\subset\G^y_y$ formed by the elements $h_g$. For every $h\in S$, let $V_h:=\bigcap_{g\in F:h_g=h}W_{g}$. This is an open bisection containing $h$. Choose a number $\delta>0$ such that $|\chi(g)-\eta(h_g)|+\delta<\eps$ for all $g\in F$ such that $W_g\cap\G^y_y\ne\emptyset$. Choose an open neighbourhood $V$ of $y$ satisfying the following properties: $\bar V$ is compact, $\bar V\subset U\cap\bigcap_{h\in S}r(V_h)$ and $VW_gV=\emptyset$ (in other words, $r^{-1}(V)\cap W_g\cap s^{-1}(V)=\emptyset$) for all $g\in F$ such that $W_g\cap\G^y_y=\emptyset$. For every $h\in S$, consider the compact subset
$$
K_h:=\IsoG\cap V_h\cap r^{-1}(\bar V)
$$
of $\IsoG\cap V_h$. We claim that then $(y,\eta)\in O \subset \UU_x^\chi(U,\eps,(W_g)_{g\in F})$, where
$$
O:=\bigcap_{h\in S}\OO(V,K_h,\{w\in\T:|w-\eta(h)|<\delta\}).
$$

Since $K_h\cap\G^y_y=\{h\}$ for all $h\in S$, it is clear that $(y,\eta)\in O$. Next, assume $(z,\omega)\in O$. Take $g\in F$ and assume $W_g\cap\G^z_z=\{h'\}$ for some $h'$. We must have $W_g\cap\G^y_y\ne\emptyset$, since otherwise $VW_gV=\emptyset$, contradicting the existence of~$h'$. Let $h:=h_g$. Then $h'\in K_h$ by the definition of~$K_h$. Hence $|\omega(h')-\eta(h)|<\delta$. By our choice of $\delta$ we also have $|\chi(g)-\eta(h)|+\delta<\eps$, hence $|\omega(h')-\chi(g)|<\eps$. This proves our claim. Hence $\UU_x^\chi(U,\eps,(W_g)_{g\in F})$ is an open neighbourhood of $(x,\chi)$.

\smallskip

Assume now that $(x,\chi)\in\bigcap^n_{i=1}\OO(U_i,K_i,V_i)$. By compactness of the sets $K_i$ we can find an open neighbourhood $U$ of $x$ and a finite subset $F\subset \Gxx$ such that $U\subset\bigcap_{i=1}^nU_i$ and $r^{-1}(U)\cap\big(\bigcup_{i=1}^nK_i\big)\subset\bigcup_{g\in F}W_g$. Indeed, otherwise we would be able to find a net $(g_j)_j$ in $\bigcup_{i=1}^nK_i$ such that it eventually lies outside of every bisection $W_g$ and $r(g_j)\to x$. But then by compactness of $\bigcup_{i=1}^nK_i\subset\IsoG$ we would get a cluster point $h$ of this net with the property $h\in\Gxx\setminus\bigcup_{g\in\Gxx}W_g$, which is impossible.

By replacing $U$ by a smaller set if necessary, we may assume that $U\subset\bigcap_{g\in F}r(W_g)$. By replacing $U$ by an even smaller set we may also assume that for every index $i$ and every $g\in F$ we have either $r^{-1}(U)\cap K_i\cap W_g=\emptyset$ or $g\in K_i$. Finally, let us choose $\eps>0$ such that
\begin{equation}\label{eq:Vi}
\{w\in\T:|w-\chi(g)|<\eps\}\subset V_i
\end{equation}
for all $i$ and $g\in F$ such that $g\in K_i$. We claim that then
$$
(x,\chi)\in\UU_x^\chi(U,\eps,(W_g)_{g\in F})\subset \bigcap^n_{i=1}\OO(U_i,K_i,V_i).
$$

In order to show this, assume $(y,\eta)\in \UU_x^\chi(U,\eps,(W_g)_{g\in F})$ and fix an index $i$. Assuming $K_i\cap\G^y_y\ne\emptyset$, take an element $h$ in this set. Then $h\in W_g$ for some $g\in F$, hence $|\chi(g)-\eta(h)|<\eps$. By our choice of $U$ we must have $g\in K_i$. Then $\eta(h)\in V_i$ by~\eqref{eq:Vi}. It follows that $(y,\eta)\in \OO(U_i,K_i,V_i)$, proving our claim. This completes the proof of the lemma.
\ep

As an immediate consequence we get the following description of convergence in $\Char(\G)$.

\begin{cor}\label{cor:char-convergence}
Fix a point $(x,\chi)\in\Char(\G)$. For every $g\in\Gxx$, choose an open bisection~$W_g$ containing $g$. Then a net $((x_i,\chi_i))_{i\in I}$ converges to $(x,\chi)$ in $\Char(\G)$ if and only if $x_i\to x$ in~$\Gu$ and, for every $g\in \Gxx$ and $\eps>0$, there is $i_0\in I$ such that for each $i\ge i_0$ we have either  $W_g\cap\G^{x_i}_{x_i}=\emptyset$, or $W_g\cap\G^{x_i}_{x_i}=\{h\}$ for some $h$ and $|\chi(g)-\chi_i(h)|<\eps$.
\end{cor}

For $(x,\chi)\in\Char(\G)$, denote by $\pi^\chi_x$ the induced representation $\Ind^\G_{\Gxx}\chi$ of $C^*(\G)$. It is known and not difficult to check that the representations $\pi^\chi_x$ are irreducible, see~\cite{IW0} for a more general statement in the second countable case. We therefore get a map
$$
\Ind\colon\Char(\G)\to\Prim C^*(\G),\quad \Ind(x,\chi):=\ker\pi^\chi_x.
$$

\begin{lemma}\label{lem:continuity}
Assume $(x,\chi)\in\Char(\G)$ is a point such that the group $\Gxx$ is amenable. Then the map $\Ind \colon\Char(\G)\to\Prim C^*(\G)$ is continuous at $(x,\chi)$.
\end{lemma}

\bp
This can be proved along the lines of \cite{MR4395600}*{Corollary~4.6}. Since our \'etale case does not really need any sophisticated machinery, we will give an essentially self-contained proof for the reader's convenience.

Assume $((x_i,\chi_i))_{i\in I}$ is a net in $\Char(\G)$ converging to $(x,\chi)$. We need to show that the net $(\Ind(x_i,\chi_i))_i$ converges to $\Ind(x,\chi)$ in $\Prim C^*(\G)$. By Lemma~\ref{lem:Jacobson} for this it suffices to show that $\pi^\chi_x$ is weakly contained in $\bigoplus_{i\in I}\pi^{\chi_i}_{x_i}$.

For every $g\in\Gxx$, fix an open bisection $W_g$ containing $g$. By Corollary~\ref{cor:char-convergence}, by passing to a subnet we may assume that for every $i\in I$ we are given a number $\eps_i>0$ and a finite subset $F_i\subset\Gxx$ such that the following properties are satisfied: $\eps_i\to0$, $F_i\nearrow\Gxx$ and for every $g\in F_i$ we have either  $W_g\cap\G^{x_i}_{x_i}=\emptyset$, or $W_g\cap\G^{x_i}_{x_i}=\{h\}$ for some~$h$ and
$|\chi(g)-\chi_{i}(h)|<\eps_i$. Let $S_i\subset F_i$ be the subset of points $g$ such that $W_g\cap\G^{x_i}_{x_i}\ne\emptyset$. By passing to a subnet we may assume that $S_i\to S$ in the Fell topology for some subset $S\subset\Gxx$.

We claim that $S$ is a subgroup of $\Gxx$. In order to see this, assume $g,h\in S$. Then for all $i$ large enough we have $g,h\in S_i$ and $gh\in F_i$. Since $W_gW_h$ is an open bisection containing $gh$, we also have $W_gW_h\cap r^{-1}(V)=W_{gh}\cap r^{-1}(V)$ for a neighbourhood $V$ of $x$, and hence if $x_i\in V$ and  $g,h\in S_i$, then $W_{gh}\cap\G^{x_i}_{x_i}\ne\emptyset$. Therefore $gh\in S_i$ for all $i$ sufficiently large, hence $gh\in S$. Similar arguments show that $S$ contains the unit and is closed under taking inverses.

By construction the subgroup $S$ has the following properties. If $g\in S$, then $W_g\cap\G^{x_i}_{x_i}=\{h_{g,i}\}$ for some $h_{g,i}$ for all $i$ sufficiently large, and $\chi_i(h_{g,i})\to\chi(g)$. While if $g\in\Gxx\setminus S$, then $W_g\cap\G^{x_i}_{x_i}=\emptyset$ for all $i$ sufficiently large.

We claim that $\pi:=\Ind^\G_S(\chi|_S)$ is weakly contained in $\bigoplus_{i\in I}\pi^{\chi_i}_{x_i}$. In order to show this, define unit vectors $\zeta$ and $\zeta_i$ in the underlying spaces of the representations $\pi$ and $\pi^{\chi_i}_{x_i}$ by
$$
\zeta(g):=\begin{cases}
            \overline{\chi(g)}, & \mbox{if } g\in S, \\
            0, & \mbox{if } g\in\G_x\setminus S,
          \end{cases}\qquad
\zeta_i(g):=\begin{cases}
            \overline{\chi_i(g)}, & \mbox{if } g\in\G^{x_i}_{x_i}, \\
            0, & \mbox{if } g\in\G_{x_i}\setminus \G^{x_i}_{x_i}.
          \end{cases}
$$
Since $\zeta$ is a cyclic vector for the representation $\pi$, in order to prove the claim it suffices to show that the states $\varphi_i:=(\pi^{\chi_i}_{x_i}(\cdot)\zeta_i,\zeta_i)$ converge weakly$^*$ to $\varphi:=(\pi(\cdot)\zeta,\zeta)$. For this it suffices to check that $\varphi_i(f)\to\varphi(f)$ for all $f\in C_c(W)$, where $W$ runs through a collection of open bisections covering $\G$. It is enough to consider the bisections $W$ satisfying one of the following properties:
\begin{enumerate}
\item $W=W_g$ for some $g\in S$;
\item $W=W_g$ for some $g\in\Gxx\setminus S$;
\item $x\notin\overline{s(W)}\cap\overline{r(W)}$.
\end{enumerate}

In the first case we have $\varphi(f)=\chi(g)f(g)$ and, for $i$ large enough, $\varphi_i(f)=\chi_i(h_{g,i})f(h_{g,i})$. Hence $\varphi_i(f)\to\varphi(f)$ by the definition of $S$ and continuity of $f$, since $h_{g,i}\to g$. In the second case we have $\varphi(f)=0$ and $\varphi_i(f)=0$ as long as $i$ is large enough so that $W_g\cap\G^{x_i}_{x_i}=\emptyset$. Thus, $\varphi_i(f)\to\varphi(f)$. In the third case we have $\varphi(f)=0$ and $\varphi_i(f)=0$ as long as $i$ is large enough so that $x_i\notin\overline{s(W)}\cap\overline{r(W)}$. Therefore we again have $\varphi_i(f)\to\varphi(f)$. This completes the proof of weak containment of $\Ind^\G_S(\chi|_S)$ in $\bigoplus_{i\in I}\pi^{\chi_i}_{x_i}$.

Now, since $\Gxx$ is amenable, the representation $\chi$ of $\Gxx$ is weakly contained in $\Ind^{\Gxx}_S(\chi|_S)$ by \cite{MR0246999}*{Theorem~5.1}. Hence $\pi^\chi_x=\Ind^\G_{\Gxx}\chi$ is weakly contained in $\Ind^\G_{\Gxx}\Ind^{\Gxx}_S(\chi|_S)\sim\Ind^\G_S(\chi|_S)$.
\ep

Denote by $\Char_a(\G)\subset\Char(\G)$ the subset of pairs $(x,\chi)$ such that $\Gxx$ is amenable, and endow $\Char_a(\G)$ with the relative topology.
The groupoid $\G$ acts on $\Char(\G)$ and $\Char_a(\G)$ by $g(x,\chi):=(r(g),\chi(g^{-1}\cdot g))$ for $g\in\G_x$. Using either Lemma~\ref{lem:Char-topology} or Corollary~\ref{cor:char-convergence}, it is easy to verify the following.

\begin{lemma}[{cf.~\cite{MR4395600}*{Corollary~3.8}}]\label{lem:contconj}
The action of $\G$ on $\Char(\G)$ is continuous.
\end{lemma}

By Lemma \ref{lem:UE}, the map $\Ind$ factors through $\G\backslash\!\Char(\G)$. Using the fact that $\Prim C^*(\G)$ is a $T_0$-space, we then get the following corollary to Lemma~\ref{lem:continuity}.

\begin{cor}\label{cor:Ind1}
The map $\Ind$ induces a continuous map $\G\backslash\!\Char_a(\G)\to\Prim C^*(\G)$, hence a continuous map
$$
\Ind^\sim\colon (\G\backslash\!\Char_a(\G))^\sim\to \Prim C^*(\G).
$$
\end{cor}

\subsection{Main results}
The following is the key result of the paper.

\begin{thm}\label{thm:main}
Let $\G$ be a Hausdorff locally compact \'etale groupoid. Assume $(x,\chi)\in\Char(\G)$ and $A\subset\Char(\G)$ are such that $\Ind(x,\chi)\in\overline{\Ind A}$ in $\Prim C^*(\G)$. Then $(x,\chi)\in\overline{\G A}$ in $\Char(\G)$.
\end{thm}

\bp For $(y,\eta)\in\Char(\G)$ and $z\in[y]:=r(\G_y)$, we denote by $\eta_z\colon\G^z_z\to\T$ the character $\eta(g^{-1}\cdot g)$, where $g$ is any element of $\G^z_y$. In other words, $(z,\eta_z)=g(y,\eta)$. Note that $\eta_z$ is independent of the choice of $g\in\G^z_y$. 

For every $g\in \Gxx$, fix an open bisection $W_g$ containing $g$. Fix a neighbourhood $U$ of $x$ in~$\Gu$, $\eps>0$ and a finite subset $F\subset\Gxx$. By the description of the topology on $\Char(\G)$ given in Lemma~\ref{lem:Char-topology} we need to show that there exist $(y,\eta)\in A$ and $z\in [y]\cap U$ satisfying the following property: for every $g\in F$, we have either  $W_g\cap\G^z_z=\emptyset$, or $W_g\cap\G^z_z=\{h\}$ for some $h$ and
$|\chi(g)-\eta_z(h)|<\eps$.

Denote by $H$ and $H_{y,\eta}$ the underlying spaces of the representations~$\pi^\chi_x$ and $\pi^\eta_y$. Consider the unit vector $\zeta\in H$ defined by
$$
\zeta(g):=\begin{cases}
            \overline{\chi(g)}, & \mbox{if } g\in\Gxx, \\
            0, & \mbox{if } g\in\G_x\setminus\Gxx,
          \end{cases}
$$
and the corresponding state $\varphi:=(\pi^\chi_x(\cdot)\zeta,\zeta)$ on $C^*(\G)$. By replacing $U$ by a smaller neighbourhood of $x$ if necessary, we may assume that there are functions $f_g\in C_c(W_g)$  ($g\in F$) such that $0\le f_g(h)\le 1$ for all $h\in W_g$, $f_g(h)=1$ for all $h\in r^{-1}(U)\cap W_g$. Let us also choose a function $f\in C_c(U)$ such that $0\le f \le 1$ and $f(x)=1$. Fix a number $\alpha\in(0,|F|^{-1})$. Let $\delta>0$ be such that
\begin{equation}\label{eq:circle-inequality}
\text{if}\ \operatorname{Re}w>1-\frac{\delta}{\alpha}\ \text{for some}\ w\in\T,\ \text{then}\ |1-w|<\eps.
\end{equation}

Since $\pi^\chi_x$ is weakly contained in $\bigoplus_{(y,\eta)\in A}\pi^\eta_y$, by Lemma~\ref{lem:Fell}(3) we can find $(y,\eta)\in A$ and a unit vector $\xi\in H_{y,\eta}$ such that
\begin{equation}\label{eq:weak}
|\varphi(f)-(\pi^\eta_y(f)\xi,\xi)|<1-|F|\alpha,\quad |\varphi(f_g)-(\pi^\eta_y(f_g)\xi,\xi)|<\delta\quad\text{for all}\quad g\in F.
\end{equation}
For every $z\in[y]$, fix an element $r_z\in\G^z_y$. We have $\varphi(f)=1$ and
$$
(\pi^\eta_y(f)\xi,\xi)=\sum_{z\in[y]}f(z)|\xi(r_z)|^2.
$$
Therefore, by the choice of $f$, the first inequality in~\eqref{eq:weak} implies that
\begin{equation}\label{eq:nec1}
\sum_{z\in [y]\cap U}|\xi(r_z)|^2>|F|\alpha.
\end{equation}
We also have $\varphi(f_g)=\chi(g)$, hence the second inequality gives
\begin{equation}\label{eq:nec2}
|\chi(g)-(\pi^\eta_y(f_g)\xi,\xi)|<\delta\quad\text{for all}\quad g\in F.
\end{equation}

For every $g\in F$, let $Z_g$ be the set of points $z\in[y]\cap U$ such that $W_g\cap \G^z_z=\{h_{g,z}\}$ for some~$h_{g,z}$~and
$
|\chi(g)-\eta_z(h_{g,z})|\ge\eps.
$
We want to show that the set $([y]\cap U)\setminus\bigcup_{g\in F}Z_g$ is nonempty. Assume this is not the case. By~\eqref{eq:nec1} we then get that
$$
\sum_{g\in F}\sum_{z\in Z_g}|\xi(r_z)|^2\ge \sum_{z\in [y]\cap U}|\xi(r_z)|^2>|F|\alpha.
$$
It follows that there is $g\in F$ such that
\begin{equation}\label{eq:>alpha}
\sum_{z\in Z_g}|\xi(r_z)|^2>\alpha.
\end{equation}

Now, for every $z\in Z_g$ consider the unit vector $\zeta_z\in H_{y,\eta}$ defined by
$$
\zeta_z(g'):=\begin{cases}
            \overline{\eta(r_z^{-1}g')}, & \mbox{if } g'\in \G^z_y, \\
            0, & \mbox{if } g'\in\G_y\setminus\G^z_y.
          \end{cases}
$$
From the definition~\eqref{eq:ind-rep} of an induced representation we get that, for each $z'\in[y]$,
$$
(\pi^\eta_y(f_g)\zeta_z)(r_{z'})
=\sum_{h\in \G^{z'}} f_{g}(h) \zeta_z(h^{-1} r_{z'})
=\sum_{h\in \G^{z'}_{z}} f_{g}(h) \zeta_z(h^{-1}r_{z'}).
$$
The last expression is zero for $z'\ne z$, since $f_g$ is supported on $W_g$ and $W_g\cap \G_z=W_g\cap \G^z_z=\{h_{g,z}\}$, while for $z'=z$ we get
$$
f_g(h_{g,z})\zeta_z(h^{-1}_{g,z}r_z)=\eta_z(h_{g,z}).
$$
Thus, $\pi^\eta_y(f_g)\zeta_z= \eta_z(h_{g,z})\zeta_z$, and a similar computation reveals that $\pi^\eta_y(f_g^*)\zeta_z=\overline{\eta_z(h_{g,z})}\zeta_z$.

Therefore in the representation $\pi^\eta_y$ the spaces $\C\zeta_z\subset H_{y,\eta}$ for $z\in Z_g$ are invariant subspaces for the C$^*$-algebra $A_g\subset C^*(\G)$ generated by $f_g$. Define $\tilde\xi\in H_{y,\eta}$ by $\tilde\xi(g'):=0$ if $r(g')\in Z_g$ and $\tilde\xi(g'):=\xi(g')$ otherwise. Then the vectors $\tilde\xi$ and $\zeta_z$ for all $z\in Z_g$ are mutually orthogonal and
$$
\xi = \tilde\xi+\sum_{z\in Z_{g}} \xi(r_{z}) \zeta_z \; .
$$
It follows that on $A_g$ we have
$$
(\pi^\eta_y(\cdot)\xi,\xi)=\sum_{z\in Z_g}|\xi(r_z)|^2(\pi^\eta_y(\cdot)\zeta_z,\zeta_z)+ (\pi^\eta_y(\cdot)\tilde\xi,\tilde\xi) \; .
$$
Applying this to $f_g$ we get
$$
\left|(\pi^\eta_y(f_g)\xi,\xi)-\sum_{z\in Z_g}|\xi(r_z)|^2\eta_z(h_{g,z})\right|=|(\pi^\eta_y(f_g)\tilde\xi,\tilde\xi)|\le \|\tilde\xi\|^2=1-\sum_{z\in Z_g}|\xi(r_z)|^2.
$$

Together with \eqref{eq:nec2} this gives
$$
\left|\chi(g)-\sum_{z\in Z_g}|\xi(r_z)|^2\eta_z(h_{g,z})\right|<1-\sum_{z\in Z_g}|\xi(r_z)|^2+\delta.
$$
It follows that
$$
\sum_{z\in Z_g}|\xi(r_z)|^2\operatorname{Re}\big(\overline{\chi(g)}\eta_z(h_{g,z})\big)>\sum_{z\in Z_g}|\xi(r_z)|^2-\delta.
$$
This implies that there is $z\in Z_g$ such that
$$
\operatorname{Re}\big(\overline{\chi(g)}\eta_z(h_{g,z})\big)>1-\frac{\delta}{\sum_{z'\in Z_g}|\xi(r_{z'})|^2}>1-\frac{\delta}{\alpha},
$$
where in the last inequality we used~\eqref{eq:>alpha}. By our choice~\eqref{eq:circle-inequality} of $\delta$ it follows that
$$
|\chi(g)-\eta_z(h_{g,z})|<\eps.
$$
But this contradicts the definition of $Z_g$.

In conclusion, there must exist an element $ z\in ([y]\cap U)\setminus\bigcup_{g\in F}Z_g$. By the definition of the sets~$Z_{g}$, we have, for every $g\in F$, that either $W_{g} \cap \G_{z}^{z} = \emptyset$, or $W_g\cap\G^z_z=\{h\}$ for some $h$ and $|\chi(g)-\eta_z(h)|<\eps$. This proves the theorem.
\ep

We can now show that the map $\Ind^\sim$ defined in Corollary~\ref{cor:Ind1} is a homeomorphism onto its image. It is convenient to formulate this in the following formally stronger form.

\begin{cor}\label{cor:Ind2}
For any $\G$-invariant subset $A\subset\Char_a(\G)$, the map $\Ind$ defines a homeomorphism of $(\G\backslash\! A)^\sim$ onto $\Ind A\subset\Prim C^*(\G)$. Moreover, the map $\Ind|_A\colon A\to\Ind A$ is open, and $\Ind(x,\chi)=\Ind(y,\eta)$ if and only if the $\G$-orbits of $(x,\chi)$ and $(y,\eta)$ have identical closures in $A$ (equivalently, in $\Char(\G)$).
\end{cor}

\bp
Consider the surjective map $p:=\Ind|_A\colon A\to\Ind A$ and the orbit equivalence relation on $X:=A$. By Lemma~\ref{lem:continuity} the map $p$ is continuous, and by respectively Lemma~\ref{lem:UE}, Lemma~\ref{lem:contconj} and Theorem~\ref{thm:main} it satisfies the three properties of Lemma~\ref{lem:T0}. This finishes the proof.
\ep

If $\Ind A$ coincides with the entire primitive spectrum, Corollary~\ref{cor:Ind2} gives a description of $\Prim C^*(\G)$. For example, if $\G$ is amenable and second countable, then we know that every primitive ideal is induced from an isotropy group~\cite{IW}. Hence we get the following result conjectured in~\cite{MR4395600}.

\begin{cor}\label{cor:kappa}
Assume $\G$ is an amenable second countable Hausdorff locally compact \'etale groupoid with abelian isotropy groups. Then the map
$\Ind\colon\Stab(\G)\dach\to\Prim C^*(\G)$ defines a homeomorphism of $(\G\backslash\!\Stab(\G)\dach)^\sim$ onto $\Prim C^*(\G)$.
\end{cor}

We finish the section with a small technical refinement of Theorem~\ref{thm:main}. It can happen that $\Ind(x,\chi)$ depends only on the values of $\chi$ on a proper subgroup of $\Gxx$. Not surprisingly, in this case we can ignore the values of $\chi$ outside that subgroup when discussing the topology on the primitive ideal space, at least when $\Gxx$ is abelian.

\begin{cor}\label{cor:prim-convergence}
Assume $\G$ is a Hausdorff locally compact \'etale groupoid, $x\in\Gu$ is a point with an abelian isotropy group $\Gxx$, $\chi\in\widehat{\Gxx}$ and $\Gamma_x\subset\Gxx$ is a subgroup such that $\Ind(x,\omega)=\Ind(x,\chi)$ for all $\omega\in\widehat{\Gxx}$ such that $\omega=\chi$ on $\Gamma_x$. For every $g\in \Gamma_x$, fix an open bisection $W_g$ containing~$g$. Then $\Ind(x,\chi)$ belongs to the closure of $\Ind A$ in $\Prim C^*(\G)$ for a subset $A\subset\Char(\G)$ if and only if for every neighbourhood~$U$ of~$x$, every $\eps>0$ and every finite subset $F\subset\Gamma_x$, there exist $(y,\eta)\in A$ and $z\in [y]\cap U$ satisfying the following property: for every $g\in F$, we have either  $W_g\cap\G^z_z=\emptyset$, or $W_g\cap\G^z_z=\{h\}$ for some~$h$ and $|\chi(g)-\eta_z(h)|<\eps$.
\end{cor}

Here we use the same notation as in the proof of Theorem~\ref{thm:main}: the character $\eta_z\colon\G^z_z\to\T$ is defined by $\eta_z=\eta(g^{-1}\cdot g)$, where $g$ is any element of $\G^z_y$.

\bp
The ``only if'' part follows from the theorem. For the ``if'' part, we may assume that $A$ is $\G$-invariant, since $\Ind(\G A)=\Ind A$. Then we can find a net $((x_i,\chi_i))_{i\in I}$ in $A$, numbers $\eps_i>0$ and finite subsets $F_i\subset\Gxx$ such that the following properties are satisfied: $x_i\to x$, $\eps_i\to0$, $F_i\nearrow\Gxx$ and for every $g\in F_i\cap\Gamma_x$ we have either  $W_g\cap\G^{x_i}_{x_i}=\emptyset$, or $W_g\cap\G^{x_i}_{x_i}=\{h_{g,i}\}$ for some~$h_{g,i}$ and
$|\chi(g)-\chi_{i}(h_{g,i})|<\eps_i$.

For every $g\in\Gxx\setminus\Gamma_x$, choose an open bisection $W_g$ containing $g$. For every $i\in I$, let $S_i\subset F_i$ be the subset of points $g$ such that $W_g\cap\G^{x_i}_{x_i}=\{h_{g,i}\}$ for some $h_{g,i}$. By passing to a subnet and arguing as in the proof of Lemma~\ref{lem:continuity}, we may assume that $S_i\to S$ for some subgroup $S\subset\Gxx$. Then $S_i\cap\Gamma_x\to S\cap\Gamma_x$. By passing to a subnet we may also assume that for every $g\in S$ the net $(\chi_i(h_{g,i}))_i$ converges to some $\eta(g)\in\T$. It is easy to see that $\eta\in\widehat{S}$ and $\eta=\chi$ on $S\cap\Gamma_x$. It follows that we can define a character $\omega$ on the subgroup $S\Gamma_x\subset\Gxx$ by $\omega(gh):=\eta(g)\chi(h)$ for $g\in S$ and $h\in\Gamma_x$. Extend $\omega$ to a character on $\Gxx$ and continue to denote this extension by~$\omega$.

By construction, for all $g\in S$ we have $\chi_i(h_{g,i})\to\eta(g)=\omega(g)$, while for $g\notin S$ we have $W_g\cap\G^{x_i}_{x_i}=\emptyset$ for all $i$ large enough. This means that $(x_i,\chi_i)\to(x,\omega)$ in $\Char(\G)$. By Lemma~\ref{lem:continuity} we then have $\Ind(x_i,\chi_i)\to\Ind(x,\omega)$. But since $\omega=\chi$ on $\Gamma_x$, we have $\Ind(x,\omega)=\Ind(x,\chi)$. Thus, $\Ind(x,\chi)\in\overline{\Ind A}$.
\ep

\bigskip

\section{Injectively graded groupoids}\label{sec:examples}

\subsection{Transformation groupoids}
Assume a discrete group $G$ acts by homeomorphisms on a Hausdorff locally compact space $X$ and denote by $G_{x}$ the stabilizer group of a point $x\in X$.  Consider the transformation groupoid $\G:=G\ltimes X$. As a topological space it is the product space $G\times X$, while the multiplication on $\G$ is defined by $(g,hx)(h,x)=(gh,x)$. In this case every element $g\in G$ defines a bisection $\{g\}\times X$ of $\G$ and as a result the topology on the space $\Char(\G)$ has a bit more transparent description: a net $((x_i,\chi_i))_{i\in I}$ converges to $(x,\chi)$ in $\Char(\G)$ if and only if $x_i\to x$ and, for every $g\in G_x$ and $\eps>0$, there is $i_0\in I$ such that for each $i\ge i_0$ we have either $gx_i\ne x_i$ or $|\chi(g)-\chi_i(g)|<\eps$.

The $\G$-orbits in $\Char(\G)$ are simply the $G$-orbits with respect to the action
\begin{equation}\label{eq:G-action}
g(x,\chi):=(gx,\chi^g),
\end{equation}
where $\chi^g:=\chi(g^{-1}\cdot g )\colon G_{gx}\to\T$. Hence Corollary~\ref{cor:Ind2} identifies the subspace of $\Prim(C_0(X)\rtimes G)$ of primitive ideals obtained by inducing one-dimensional representations of amenable stabilizers with $(G\backslash\!\Char_a(G\ltimes X))^\sim$. Corollary~\ref{cor:kappa} in this case gives the following result.

\begin{thm}
Assume we are given an amenable action of a countable group $G$ on a second countable Hausdorff locally compact space $X$ such that the stabilizer of every point is abelian. Then the map $\Ind\colon\Stab(G\ltimes X)\dach\to\Prim(C_0(X)\rtimes G)$ defines a homeomorphism of $(G\backslash\!\Stab(G\ltimes X)\dach)^\sim$ onto $\Prim(C_0(X)\rtimes G)$.
\end{thm}

The space $\Stab(G\ltimes X)\dach$ admits a better description when all stabilizers are contained in one abelian subgroup of $G$, which can then be assumed to be normal. Namely, we have the following result.

\begin{prop}\label{prop:Wil}
Assume a discrete group $G$ acts on a Hausdorff locally compact space $X$ and all stabilizers are contained in a normal abelian subgroup $H\subset G$. Let $\Delta$ be the quotient topological space of $X\times\widehat{H}$ obtained by identifying $(x,\chi)$ with $(x,\eta)$ when $\chi|_{G_x}=\eta|_{G_x}$. Consider the actions of $G$ on $X\times\widehat{H}$ and $\Delta$ defined similarly to~\eqref{eq:G-action}. Then the $G$-equivariant map
$$
X\times\widehat{H}\to\Stab(G\ltimes X)\dach,\quad (x,\chi)\mapsto(x,\chi|_{G_x}),
$$
is continuous and open, hence it induces a homeomorphism of $\Delta$ onto $\Stab(G\ltimes X)\dach$ and a homeomorphism of $(G\backslash\Delta)^\sim$ onto $(G\backslash\!\Stab(G\ltimes X)\dach)^\sim$.
\end{prop}

When $\G=G\ltimes X$ is amenable and second countable, the last statement follows from \cite{MR4395600}*{Propositions~8.3} and \cite{MR0617538}*{Corollary 5.11}, but this relies on the identification of $(G\backslash\Delta)^\sim$ with $\Prim(C_0(X)\rtimes G)$ established in \cite{MR0617538}. We will prove the proposition by a direct argument. Together with our Corollary~\ref{cor:Ind2} this will provide an alternative proof of \cite{MR0617538}*{Corollary 5.11} in the case of discrete group actions.

For the proof we need a couple of auxiliary results, which will also be useful later.

\begin{lemma}
Assume $H$ is a discrete abelian group, $S$ is a subgroup of $H$ and $(S_i)_{i\in I}$ is a net of subgroups of $H$. Then $S_i\to S$ in the Chabauty topology on $\Sub(H)$ if and only if $S_i^\perp\to S^\perp$ in $\Sub(\widehat{H})$.
\end{lemma}

\bp
Assume first that $S_i\to S$. By passing to a subnet we may assume that $S_i^\perp\to T^\perp$ for some subgroup $T\subset H$, and we need to prove that $T=S$.

Take $h\in T$. As $S_i^\perp\to T^\perp$ and $\widehat{H}$ is compact, the groups $S_i^\perp$ eventually lie in every given neighbourhood of $T^\perp$. In particular, for every $\eps>0$ we have  $|1-\chi(h)|<\eps$ for all $\chi\in S_i^\perp$ and all $i$ large enough. But if a group character takes values in the set $\{z:|1-z|<1\}$, then it is trivial. Hence, for all $i$ large enough, we have $h\in(S_i^\perp)^\perp=S_i$. It follows that $h\in S$. Thus, $T\subset S$.

Next, take $\chi\in T^\perp$ and $h\in S$. As $S_i^\perp\to T^\perp$, for every $\eps>0$ and all $i$ large enough we can find $\chi_i\in S_i^\perp$ such that $|\chi(h)-\chi_i(h)|<\eps$. But eventually we have $h\in S_i$, so $\chi_i(h)=1$ and we get $|\chi(h)-1|<\eps$. It follows that $\chi(h)=1$. Therefore $\chi\in S^\perp$. Thus, $T^\perp\subset S^\perp$, and then $S\subset T$, completing the proof of the equality $T=S$.

Conversely, assume $S_i^\perp\to S^\perp$. By passing to a subnet we may assume that $S_i\to T$ for some subgroup $T\subset H$, and we need to prove that $T=S$. But by the first part of the proof we already know that $S_i^\perp\to T^\perp$, hence $T^\perp=S^\perp$ and $T=S$.
\ep


\begin{lemma}\label{lem:fellconv}
Assume $H$ is a discrete abelian group, $(S_i)_{i\in I}$ is a net of subgroups of $H$ converging to a subgroup $T\subset H$. Assume $(\chi_i)_{i\in I}$ and $\chi$ are characters on $H$, $S\subset H$ is a subgroup and $\chi_i(h)\to\chi(h)$ for all $h\in S\cap T$. Then, by possibly passing to a subnet, we can find characters $\eta_i,\eta\in\widehat{H}$ such that $\eta_i|_{S_i}=\chi_i|_{S_i}$, $\eta|_S=\chi|_S$ and $\eta_i(h)\to\eta(h)$ for all $h\in H$.
\end{lemma}

\bp
By passing to a subnet we may assume that $\chi_i\to\omega$ for some $\omega\in\widehat{H}$. Then $\omega|_{S\cap T}=\chi|_{S\cap T}$, so $\chi\omega^{-1}$ lies in the annihilator of $S\cap T$. By the previous lemma, the annihilators of $S_i\cap S$ in~$\widehat S$ converge to the annihilator of $T\cap S$. It follows that, by possibly passing to a subnet, we can find characters $\nu_i\in\widehat S$ such that~$\nu_i$ is trivial on $S_i\cap S$ and $\nu_i(h)\to \chi(h)\omega(h)^{-1}$ for all $h\in S$. Extend~$\nu_i$ to a character on $S_iS$ by letting $\nu_i(gh):=\nu_i(h)$ for $g\in S_i$ and $h\in S$, and then extend~$\nu_i$ to a character on $H$. We continue to denote the extension by $\nu_i$.

By construction we have $\nu_i\in S_i^\perp\subset\widehat{H}$ and $\nu_i(h)\to\chi(h)\omega(h)^{-1}$ for all $h\in S$. By passing to a subnet, we may assume that $(\nu_i)_i$ converges to some character $\nu\in\widehat{H}$. We let $\eta_i:=\chi_i\nu_i$ and $\eta:=\omega\nu$. Then $\eta_i|_{S_i}=\chi_i|_{S_i}$, $\eta_i\to \eta$ and $\eta|_S=\omega\chi\omega^{-1}|_S=\chi|_S$.
\ep

\bp[Proof of Proposition~\ref{prop:Wil}]
It is immediate that the map $p\colon X\times\widehat{H}\to \Stab(G\ltimes X)\dach$, $(x,\chi)\mapsto(x,\chi|_{G_x})$, is continuous. Assume it is not open. Then there exist a point $(x,\chi)\in X\times\widehat{H}$, an open neighbourhood $U\times V$ of this point and a net $((x_i,\chi_i))_i$ such that $p(x_i,\chi_i)\notin p(U\times V)$, but $p(x_i,\chi_i)\to p(x,\chi)$.

By passing to a subnet we may assume that $G_{x_i}\to T$ for some subgroup $T\subset H$. The convergence $p(x_i,\chi_i)\to p(x,\chi)$
means then that $x_i\to x$ and $\chi_i(h)\to\chi(h)$ for all $h\in G_x\cap T$. Note that we actually have $T\subset G_x$, since if $hx_i=x_i$ for some $h\in H$ and all $i$ large enough, then $hx=x$. By the previous lemma applied to $S_i=G_{x_i}$, $T$ and $S=H$, by possibly passing to a subnet, we can then find characters $\eta_i\in\widehat{H}$ such that $\eta_i|_{G_{x_i}}=\chi_i|_{G_{x_i}}$ and $\eta_i(h)\to\chi(h)$ for all $h\in H$. It follows that for all $i$ large enough we have $x_i\in U$ and $\eta_i\in V$, hence $p(x_i,\chi_i)=p(x_i,\eta_i)\in p(U\times V)$, which is a contradiction.
\ep

\subsection{Groupoids graded by abelian groups}\label{ssec:graded}
Next we consider a Hausdorff locally compact \'etale groupoid~$\G$ injectively graded by a discrete abelian group $\Gamma$. By definition this means that we are given a continuous homomorphism (another name is $1$-cocycle) $\Phi\colon\G\to\Gamma$ such that its restriction to every isotropy group $\Gxx$ is injective. In this case every character on $\Gxx$ has the form $\chi\circ\Phi$ for some $\chi\in\widehat{\Gamma}$. Therefore as a set the space $\Stab(\G)\dach$ is a quotient of $\Gu\times\widehat{\Gamma}$. Moreover, the groupoid~$\G$ acts on $\Gu\times\widehat{\Gamma}$ by $g(x,\chi):=(r(g),\chi)$ for $g\in\G_x$, and then the map
\begin{equation}\label{eq:q-on-Stab}
q\colon\Gu\times\widehat{\Gamma}\to \Stab(\G)\dach,\quad q(x,\chi):=(x,\chi\circ\Phi|_{\Gxx}),
\end{equation}
is $\G$-equivariant. This map is continuous, but as we will see later, it is in general not open, in contrast to Proposition~\ref{prop:Wil}.


For every $x\in\Gu$ and $\chi\in\widehat{\Gamma}$, denote by $\pi_{(x,\chi)}$ the representation $\pi^{\chi\circ\Phi|_{\Gxx}}_x$, so
$$
\pi_{(x,\chi)}=\Ind^\G_{\Gxx}(\chi\circ\Phi|_{\Gxx}).
$$
By definition, its underlying space is the Hilbert space of functions $f\colon\G_x\to\C$ such that $f(gh)=\overline{\chi(\Phi(h))}f(g)$ for $g\in\G_x$ and $h\in \Gxx$ and
$$
\sum_{g\in\G_x/\Gxx}|f(g)|^2<\infty.
$$
We can identify this space with $\ell^2([x])$ by associating to every $\xi\in\ell^2([x])$ the function $f_\xi\colon\G_x\to\C$ by
$$
f_\xi(g):=\overline{\chi(\Phi(g))}\xi(r(g)).
$$
Then $\pi_{(x,\chi)}$ becomes a representation $C^*(\G)\to B(\ell^2([x]))$ such that
\begin{equation}\label{eq:pi(x,z)}
\pi_{(x,\chi)}(f)\delta_y=\sum_{g\in\G_y}\chi(\Phi(g))f(g)\delta_{r(g)},\quad y\in[x],\ f\in C_c(\G).
\end{equation}

Using this notation, Corollary~\ref{cor:kappa} (recall also Corollary~\ref{cor:prim-convergence} for a more explicit description of convergence in the Jacobson topology) gives the following result.

\begin{thm}\label{thm:gradedgroupoids}
Assume $\G$ is an amenable second countable Hausdorff locally compact \'etale groupoid injectively graded by a discrete abelian group $\Gamma$, with grading $\Phi\colon\G\to\Gamma$. Then the map $\Gu\times\widehat{\Gamma}\to\Prim C^*(\G)$, $(x,\chi)\mapsto\ker\pi_{(x,\chi)}$, is surjective, and the topology on $\Prim C^*(\G)$ is described as follows.  Fix $(x,\chi)\in\Gu\times\widehat{\Gamma}$ and, for every $g\in\Gxx$, choose an open bisection $W_g$ containing $g$. Then, given a sequence $((x_n,\chi_n))^\infty_{n=1}$ in $\Gu\times\widehat{\Gamma}$, we have $\ker\pi_{(x_n,\chi_n)}\to\ker\pi_{(x,\chi)}$ if and only if for every neighbourhood $U$ of $x$, every $\eps>0$ and every finite subset $F\subset\Gxx$, there exists an index $n_0$ such that for each $n\ge n_0$ we can find $y\in [x_n]\cap U$ satisfying the following property: for every $g\in F$, we have
$$
\text{either}\quad W_g\cap\G^y_y=\emptyset\quad \text{or}\quad |\chi(\Phi(g))-\chi_n(\Phi(g))|<\eps.
$$
\end{thm}

Note that we could have formulated the result in terms of nets, but since under our assumptions of second countability the spaces $\Prim C^*(\G)$ and $\Stab(\G)\dach$ are second countable (see~\cite{MR4395600}*{Lemma~3.3}), it is enough to deal with sequences.

\smallskip

An immediate consequence of Theorem~\ref{thm:gradedgroupoids}, or already of Lemma~\ref{lem:continuity}, is that if $\overline{[x]}=\overline{[y]}$, then $\ker\pi_{(x,\chi)}=\ker\pi_{(y,\chi)}$ for all $\chi\in\widehat\Gamma$. This does not fully describe when different points $(x,\chi)$ define the same ideal $\ker\pi_{(x,\chi)}$. However, such a description is known from~\cite{CN2} and can be deduced again from Theorem~\ref{thm:gradedgroupoids}. We need to recall some notation in order to formulate the result. Denote by $\Iso$ the interior of the isotropy bundle $\IsoG\subset\G$. For $x\in\Gu$, we then consider the interior $\IsoGx{x}^\circ$ of the isotropy bundle of the groupoid $\G_{\overline{[x]}}$ obtained by reducing~$\G$ to the closed invariant subset~$\overline{[x]}$. The subgroup $\IsoGx{x}^\circ_x\subset\Gxx$ is called the \emph{essential isotropy} of $\G$  at~$x$ in~\cite{BCS}.

We start with the following result, which is already implicit in~\cite{CN2} and will be important in Section~\ref{ssec:1-graphs}.

\begin{cor} \label{cor:essential-iso2}
In the setting of Theorem~\ref{thm:gradedgroupoids}, every primitive ideal of $C^*(\G)$ has the form $\ker\pi_{(x,\chi)}$ for some $(x,\chi)\in\Gu\times\widehat{\Gamma}$ such that $\IsoGx{x}^\circ_x=\Gxx$.
\end{cor}

\bp
We know that every  primitive ideal of $C^*(\G)$ has the form $\ker\pi_{(y,\chi)}$ for some $(y,\chi)\in\Gu\times\widehat{\Gamma}$. By \cite{CN2}*{Lemmas~5.2 and~6.2} we can find $x\in\overline{[y]}$ such that $\overline{[x]}=\overline{[y]}$ and $\IsoGx{x}^\circ_x=\Gxx$. Then $\ker\pi_{(y,\chi)}=\ker\pi_{(x,\chi)}$.
\ep

\begin{cor}[{\cite{CN2}*{Theorem~C}}]\label{cor:essential-iso1}
In the setting of Theorem~\ref{thm:gradedgroupoids}, assume we are given points $(x_1,\chi_1)$ and $(x_2,\chi_2)$ in~$\Gu\times\hat\Gamma$. Then $\ker\pi_{(x_1,\chi_1)}=\ker\pi_{(x_2,\chi_2)}$ if and only if $\overline{[x_1]}=\overline{[x_2]}$ and $\chi_1=\chi_2$ on $\Phi(\IsoGx{x_1}^\circ_{x_1})=\Phi(\IsoGx{x_2}^\circ_{x_2})$.
\end{cor}

We remark that the equality $\Phi(\IsoGx{x_1}^\circ_{x_1})=\Phi(\IsoGx{x_2}^\circ_{x_2})$ here is a consequence of $\overline{[x_1]}=\overline{[x_2]}$ rather than an extra requirement, see~\cite{CN2}*{Lemma~6.4}.

\bp
Since the equality $\overline{[x_1]}=\overline{[x_2]}$ is necessary for $\ker\pi_{(x_1,\chi_1)}=\ker\pi_{(x_2,\chi_2)}$, we may assume that it is satisfied, and then by passing to the reduced groupoid $\G_{\overline{[x_1]}}$ we may assume that both~$[x_1]$ and~$[x_2]$ are dense in~$\Gu$.

For every $g\in\G^{x_1}_{x_1}$, choose an open bisection $W_g$ containing $g$ such that $\Phi(W_g)=\{\Phi(g)\}$. By Theorem~\ref{thm:gradedgroupoids}, we have $\ker\pi_{(x_2,\chi_2)}\subset\ker\pi_{(x_1,\chi_1)}$ if and only if for every neighbourhood $U$ of $x_1$, every $\eps>0$ and every finite subset $F\subset\G^{x_1}_{x_1}$, there exists $y\in [x_2]\cap U$ satisfying the following property: for every $g\in F$, we have
$$
\text{either}\quad W_g\cap\G^y_y=\emptyset\quad \text{or}\quad |\chi_1(\Phi(g))-\chi_2(\Phi(g))|<\eps.
$$

If $g\in\Iso_{x_1}$, then $W_g\cap\G^y_y\ne\emptyset$ for all $y$ close $x_1$. It follows that the equality $\chi_1=\chi_2$ on~$\Phi(\Iso_{x_1})=\Phi(\Iso_{x_2})$ is necessary for $\ker\pi_{(x_2,\chi_2)}\subset\ker\pi_{(x_1,\chi_1)}$. Conversely, assume it is satisfied. As in the proof of Corollary~\ref{cor:essential-iso2}, take any point $x\in\Gu$ with dense orbit such that $\Iso_x=\Gxx$. Since $\Phi(\Iso_{x_1})=\Phi(\Iso_{x_2})=\Phi(\Iso_x)$ and $\ker\pi_{(x,\chi)}$ depends only on the restriction of $\chi$ to $\Phi(\Gxx)=\Phi(\Iso_x)$, we then have
$$
\ker\pi_{(x_1,\chi_1)}=\ker\pi_{(x,\chi_1)}=\ker\pi_{(x,\chi_2)}=\ker\pi_{(x_2,\chi_2)},
$$
finishing the proof of the lemma.
\ep

\begin{remark}\label{rem:essential-iso-proof}
A slightly longer but perhaps more illuminating way to prove Corollary~\ref{cor:essential-iso1} is to show that if $x_1$ and $x_2$ have dense orbits, then for every nonempty finite subset $F\subset\G^{x_1}_{x_1}\setminus\Iso_{x_1}$ and every neighbourhood~$U$ of~$x_1$, we can find $y\in[x_2]\cap U$ such that $W_g\cap\G^y_y=\emptyset$ for all $g\in F$.
To see why this is true, take any point $x\in U\cap\bigcap_{g\in F}r(W_g)$ such that $\overline{[x]}=\Gu$ and $\Iso_x=\Gxx$, which is possible by \cite{CN2}*{Lemmas~5.2 and~6.2}. Then $W_g\cap\Gxx=\emptyset$ for all $g\in F$, since otherwise we would get $\Phi(F)\cap\Phi(\Iso_{x_1})=\Phi(F)\cap\Phi(\Iso_{x})\ne\emptyset$, in contradiction with injectivity of $\Phi$ on~$\G^{x_1}_{x_1}$. For every $g\in F$, we can then find a neighbourhood $V_g$ of $x$ such that $W_g\cap\G^y_y=\emptyset$ for all $y\in V_g$. Then any point $y$ in $[x_2]\cap U\cap\bigcap_{g\in F}V_g$ has the required property.
\end{remark}

\begin{remark}
If $\Gxx\ne\IsoGx{x}^\circ_x$ for some $x$, then the map $q\colon \Gu\times\widehat\Gamma\to\Stab(\G)\dach$ is not open and the topology on $\Stab(\G)\dach$ is non-Hausdorff. In order to see this, take characters $\chi_1,\chi_2\in\widehat\Gamma$ such that $\chi_1=\chi_2$ on $\Phi(\IsoGx{x}^\circ_x)$, but $\chi_1\ne\chi_2$ on $\Phi(\Gxx)$. We claim that for any neighbourhoods~$V_1$ of~$q(x,\chi_1)$ and $V_2$ of $q(x,\chi_2)$ we can find $y\in[x]$ such that $q(y,\chi_1)\in V_1\cap V_2$. By Lemma~\ref{lem:Char-topology} we can find a neighbourhood $U_i$ of $x$, $\eps>0$ and a finite subset $F_i\subset\Gxx$ such that $V_i$ contains all points $q(y,\eta)$ such that $y\in U_i$ and $|\chi_i(g)-\eta(g)|<\eps$ for all $g\in F_i$ with $W_g\cap\G^y_y\ne\emptyset$. Then it suffices to show that there is $y\in [x]\cap U_1\cap U_2$ such that $W_g\cap\G^y_y=\emptyset$ for all $g\in F_2\setminus\IsoGx{x}^\circ_x$. Such a point $y$ indeed exists by the previous remark, so $\Stab(\G)\dach$ is non-Hausdorff. Next, if $\gamma\in\Phi(\Gxx)$ is such that $\chi_1(\gamma_1)\ne\chi_2(\gamma_2)$, then the set $q(\Gu\times\{\eta:|\eta(\gamma)-\chi_2(\gamma)|<|\chi_1(\gamma)-\chi_2(\gamma)|\})$ does not contain any point of the form $q(y,\chi_1)$, so this set cannot be a neighbourhood of~$q(x,\chi_2)$. This shows that $q$ is not open.
Examples of such groupoids $\G$ can be obtained from local homeomorphisms containing periodic points with nondiscrete orbits, see Lemma~\ref{lem:essential=Gxx}.
\end{remark}

\begin{remark}
One might wonder whether some form of Corollary~\ref{cor:essential-iso1} is true for more general groupoids. For example, given a groupoid  $\G$, a point $x\in\Gu$ and homomorphisms $\chi_1,\chi_2\colon\Gxx\to\T$, does the equality $\Ind(x,\chi_1)=\Ind(x,\chi_2)$ imply that $\chi_1=\chi_2$ on $\IsoGx{x}^\circ_x$? The answer is ``no'' already for transformation groupoids. In order to see this, assume we are given an action of a discrete group $\Gamma$ on a locally compact space $X$ and an action of $\Gamma$ by automorphisms on a discrete abelian group $H$. We can then view the action $\Gamma\curvearrowright X$ as an action of $\Gamma\ltimes H$ and consider the transformation groupoid $\G:=(\Gamma\ltimes H)\ltimes X$. Assume $x\in X$ is a point with trivial stabilizer in $\Gamma$. Then $\IsoGx{x}^\circ_x=\Gxx=H$. For $\chi_1,\chi_2\in\widehat H$, it follows from Corollary~\ref{cor:Ind2} that $\Ind(x,\chi_1)=\Ind(x,\chi_2)$ in $\Prim C^*(\G)$ if and only if the $\Gamma$-orbits of $(x,\chi_1)$ and $(x,\chi_2)$ have the same closures in $X\times\widehat H$, which is in general a weaker requirement than $\chi_1=\chi_2$. As a concrete example we can take the standard actions $\SL_2(\Z)\curvearrowright \R^2$ and $\SL_2(\Z)\curvearrowright \Z^2$ and any point $x\in\R^2\setminus\R\cdot\Q^2$, see~\cite{CN4}.\ee
\end{remark}

Next, we can improve Corollary~\ref{cor:prim-convergence} as follows.

\begin{cor}\label{cor:graded-groupoids}
In the setting of Theorem~\ref{thm:gradedgroupoids}, with a point $(x,\chi)\in\Gu\times\widehat{\Gamma}$ and open bisections~$W_g$ ($g\in\Gxx$) fixed, assume $((x_n,\chi_n))^\infty_{n=1}$ is a sequence in $\Gu\times\widehat{\Gamma}$ such that for every neighbourhood~$U$ of~$x$, every $\eps>0$ and every finite subset $F\subset\IsoGx{x}^\circ_x$, there exists an index~$n_0$ such that for each $n\ge n_0$ we can find $y\in [x_n]\cap U$ satisfying the following property: for every $g\in F$, we have
$$
\text{either}\quad W_g\cap\IsoGx{y}^\circ_y=\emptyset\quad \text{or}\quad |\chi(\Phi(g))-\chi_n(\Phi(g))|<\eps.
$$
Then $\ker\pi_{(x_n,\chi_n)}\to\ker\pi_{(x,\chi)}$ in $\Prim C^*(\G)$.
\end{cor}

\bp
The proof can actually be reduced to Corollary~\ref{cor:prim-convergence} by perturbing $x_n$ to get $\IsoGx{x_n}^\circ_{x_n}=\G^{x_n}_{x_n}$ and arguing similarly to Remark~\ref{rem:essential-iso-proof} that the assumptions of the corollary are still satisfied. But we will instead use a more constructive argument based on Lemma~\ref{lem:fellconv}. Its advantage is that with minor modifications it can be used without the assumption of second countability.

It suffices to show that we can find a subsequence converging to $\ker\pi_{(x,\chi)}$. In order to ease the notation let us write $\Gamma_y$ for $\IsoGx{y}^\circ_y$. Since $\Phi$ is locally constant and~$\G$ is second countable, the set $\Phi(\G)$ is countable. Therefore without loss of generality we may assume that the group $\Gamma$ is countable. By replacing $W_g$ by smaller bisections we may also assume that $\Phi(W_g)=\{\Phi(g)\}$ for all $g\in\Gxx$.

By passing to a subsequence we may assume that $\chi_n\to\omega$ for a character $\omega\in\widehat{\Gamma}$ and $\Phi(\Gamma_{x_n})\to T$ in $\Sub(\Gamma)$ for a subgroup $T\subset\Gamma$. Choose an increasing sequence of finite subsets $F_n\subset\Gxx$ and a sequence of numbers $\eps_n>0$ such that $\cup_nF_n=\Gxx$ and $\eps_n\to0$.
Similarly to the proof of Corollary~\ref{cor:prim-convergence}, by passing to a subsequence and replacing $x_n$ by an element on the same orbit, we may assume that $x_n\to x$ and for every $g\in F_n\cap\Gamma_x$ we have either  $W_g\cap\Gamma_{x_n}=\emptyset$ or $|\chi(\Phi(g))-\chi_n(\Phi(g))|<\eps_n$.

For every $n$, let $S_n\subset F_n$ be the subset of points $g$ such that $W_g\cap \G^{x_n}_{x_n}\ne\emptyset$, and let $R_n\subset S_n$ be the subset of points $g$ such that $W_g\cap\Gamma_{x_n}\ne\emptyset$. By passing to a subsequence and arguing as in the proof of Lemma~\ref{lem:continuity}, we may assume that $S_n\to S$ and $R_n\to R$ for some subgroups $R\subset S\subset\Gxx$. Then $\chi_n(\Phi(g))\to\chi(\Phi(g))$ for all $g\in\Gamma_x\cap R$. At the same time $\chi_n\to\omega$, hence $\omega=\chi$ on $\Phi(\Gamma_x)\cap\Phi(R)$. Let $\eta\in\widehat{\Gamma}$ be any character such that $\eta=\chi$ on $\Phi(\Gamma_x)$ and $\eta=\omega$ on~$\Phi(R)$. We then have $\chi_n(\Phi(g))\to\eta(\Phi(g))$ for all $g\in R$.

Observe next that $\Phi(S)\cap T=\Phi(R)$. Indeed, it is clear that $\Phi(R)\subset \Phi(S)\cap T$. To prove the opposite inclusion, take $g\in S$ such that $\Phi(g)\in T$. Then, for all $n$ large enough, $W_g\cap \G^{x_n}_{x_n}\ne\emptyset$ and $\Phi(g)\in\Phi(\Gamma_{x_n})$. As $\Phi(W_g)=\{\Phi(g)\}$ and $\Phi$ is injective on $\G^{x_n}_{x_n}$, this is possible only if $W_g\cap \Gamma_{x_n}\ne\emptyset$. Hence $g\in R$.

We now apply Lemma~\ref{lem:fellconv} to $\Phi(S)$, $T$ and $\Phi(\Gamma_{x_n})$ in place of $S$, $T$ and $S_n$. By possibly passing to a subsequence, we can then find characters $\eta_n$ and $\tilde\eta$ in $\widehat{\Gamma}$ such that
$$
\eta_n=\chi_n\ \ \text{on}\ \ \Phi(\Gamma_{x_n}),\quad \tilde\eta=\eta\ \ \text{on}\ \ \Phi(S),\quad \eta_n\to\tilde\eta.
$$
Then, on the one hand, $\ker\pi_{(x_n,\chi_n)}=\ker\pi_{(x_n,\eta_n)}$ and $\ker\pi_{(x,\chi)}=\ker\pi_{(x,\eta)}$ by Corollary~\ref{cor:essential-iso1}. On the other hand, by Theorem~\ref{thm:gradedgroupoids} and the definition of $S$ we have $\ker\pi_{(x_n,\eta_n)}\to \ker\pi_{(x,\eta)}$.
\ep

\begin{remark}\label{rem:invariant-open}
By Corollaries~\ref{cor:kappa} and~\ref{cor:open}, the map $\Ind\colon\Stab(\G)\dach\to \Prim C^*(\G)$ establishes a bijection between the open $\G$-invariant subsets of $\Stab(\G)\dach$ and the open subsets of $\Prim C^*(\G)$. This implies that in terms of the topology on $\Stab(\G)\dach$ the above corollary can be formulated as follows. If $V\subset\Stab(\G)\dach$ is a $\G$-invariant open subset, then for every point $(x,\chi)\in q^{-1}(V)\subset\Gu\times\widehat\Gamma$ there exists a neighbourhood $U$ of $x$ and a finite subset $F\subset\IsoGx{x}^\circ_x$ such that $q^{-1}(V)$ contains the set
$\tilde\UU^\chi_x(U,\eps,(W_g)_{g\in F})$ consisting of all pairs $(y,\eta)\in\Gu\times\widehat\Gamma$ such that $y\in U$ and, for every $g\in F$, either $W_g\cap \IsoGx{y}^\circ_y=\emptyset$ or $|\chi(\Phi(g))-\eta(\Phi(g))|<\eps$. Indeed, otherwise we would be able to construct a sequence of elements $(x_n,\chi_n)\notin q^{-1}(V)$ satisfying the assumptions of the corollary. But then $\ker\pi_{(x_n,\chi_n)}\to\ker\pi_{(x,\chi)}$, contradicting the fact that $q^{-1}(V)$ is the pre-image of an open neighbourhood of $\ker\pi_{(x,\chi)}$ in $\Prim C^*(\G)$. \ee
\end{remark}

For a later use we need yet another reformulation of convergence in $\Prim C^*(\G)$. First, let us introduce the following notion.

\begin{defn} \label{def:newalongconv}
Let $\Gamma$ be a discrete abelian group and $(S_{n})^\infty_{n=1}$ be a sequence of subsets of $\Gamma$. We say that a sequence $(\chi_{n})^\infty_{n=1}$ in $\widehat{\Gamma}$ \emph{converges to $\chi\in \widehat{\Gamma}$ along $(S_{n})_n$} if, for all $\gamma \in \Gamma$, we have
$$
\lim_{n\to\infty}\un_{S_n}(\gamma)|\chi_n(\gamma)-\chi(\gamma)|=0,
$$
where $\un_{S_n}$ is the characteristic function of $S_n$.
\end{defn}

\begin{cor}\label{cor:graded-groupoids2}
In the setting of Theorem~\ref{thm:gradedgroupoids}, with a point $(x,\chi)\in\Gu\times\widehat{\Gamma}$ and open bisections~$W_g$ ($g\in\Gxx$) fixed, the following conditions on a sequence $((x_n,\chi_n))_n$ in $\Gu\times\widehat{\Gamma}$ are equivalent:
\begin{enumerate}
  \item $\ker\pi_{(x_n,\chi_n)}\to\ker\pi_{(x,\chi)}$ in $\Prim C^*(\G)$;
  \item there exist points $y_n\in[x_n]$ such that $y_n\to x$ and $\chi_n\to\chi$ along the sets
\begin{equation}\label{eq:Sn}
S_n:=\{\Phi(g): g\in\Gxx,\ W_g\cap\G^{y_n}_{y_n}\ne\emptyset\};
\end{equation}
  \item there exist points $y_n\in[x_n]$ such that $y_n\to x$ and $\chi_n\to\chi$ along the sets
\begin{equation}\label{eq:Rn}
R_n:=\{\Phi(g): g\in\IsoGx{x}^\circ_x,\ W_g\cap\IsoGx{y_n}^\circ_{y_n}\ne\emptyset\}.
\end{equation}
\end{enumerate}
\end{cor}

\bp It is clear that (2)$\Rightarrow$(3). The implication (3)$\Rightarrow$(1) follows from Corollary~\ref{cor:graded-groupoids}. It remains to show the implication (1)$\Rightarrow$(2). This is a routine consequence of Theorem~\ref{thm:gradedgroupoids}. Namely, assume that $\ker\pi_{(x_n,\chi_n)}\to\ker\pi_{(x,\chi)}$. Fix a decreasing sequence $(U_n)_n$ of open sets forming a neighbourhood basis of $x$. Choose also an increasing sequence of finite subsets $F_n\subset\Gxx$ with union $\Gxx$. By Theorem~\ref{thm:gradedgroupoids}, for every $k\ge1$, there is $n_k\ge1$ such that for all $n\geq n_{k}$ we can find a point $y^{k}_{n}\in [x_{n}] \cap U_k$ with the following property: for every $g\in F_k$, either $W_g \cap \G_{y^{k}_{n}}^{y^{k}_{n}}=\emptyset$ or $|\chi(\Phi(g))-\chi_{n}(\Phi(g))| < 1/k$. We can assume without loss of generality that $n_k<n_{k+1}$ for all~$k$. We then take $y_n:=x_n$ for $n<n_1$ and $y_n: = y_{n}^{k}$ for $n_k\le n<n_{k+1}$ and $k\ge1$.
\ep


\begin{remark}
Throughout this entire subsection amenability of $\G$ was needed only for surjectivity of $\Ind\colon\Stab(\G)\dach\to\Prim C^*(\G)$, so as long we do need this surjectivity or replace $\Prim C^*(\G)$ by $\Ind(\Stab(\G)\dach)$, all the results hold for any second countable Hausdorff locally compact \'etale groupoid injectively graded by a discrete abelian group.
\end{remark}

\subsection{Harmonious families of bisections}\label{ssec:harmonious}
Brix, Carlsen and Sims~\cite{BCS} have recently described the topology on the primitive ideal space for the Deaconu--Renault groupoids defined by $k$-tuples of commuting local homeomorphisms under the assumption of existence of so-called \emph{harmonious families of bisections} at each point of the unit space. We will discuss these groupoids in more detail later. The goal of this subsection is to show that the description in~\cite{BCS} follows from our Corollary~\ref{cor:graded-groupoids2} for all amenable injectively graded groupoids for which harmonious families exist.

\smallskip

Assume therefore as in Section~\ref{ssec:graded} that $\G$ is an amenable second countable Hausdorff locally compact \'etale groupoid injectively graded by a discrete abelian group $\Gamma$, with grading $\Phi\colon\G\to\Gamma$. A harmonious family of bisections at $x\in\Gu$ is a collection $\BB=(B_\alpha)_\alpha$ of open bisections~$B_\alpha$ that contain different elements of $\Gxx$ and satisfy a number of axioms. We refer the reader to~\cite{BCS}*{Definition~6.1} for the precise formulation. What is important for us are the following two consequences of the definition:
\begin{enumerate}
  \item[(i)] every $g\in\IsoGx{x}^\circ_x\subset\Gxx$ lies in one of the bisections $B_\alpha$;
  \item[(ii)] for every $y\in\Gu$ sufficiently close to $x$, the set
  $$
  H_\BB(y):=\{\Phi(g): g\in \Gxx\cap B_\alpha\ \ \text{for some} \ \alpha,\ \ \IsoGx{y}^\circ_y\cap B_\alpha\ne\emptyset\}
  $$
  is a subgroup of $\Gamma$.
\end{enumerate}

For every $g\in\Gxx$, let us fix an open bisection $W_g$ containing $g$ such that if $g\in B_\alpha$ for some~$\alpha$, then $W_g=B_\alpha$. Then, if $(y_n)_n$ is a sequence converging to $x$, we see from properties (i) and (ii) above that
$$
R_n\subset H_\BB(y_n)\subset S_n,
$$
where the sets $S_n$ and $R_n$ are defined by~\eqref{eq:Sn} and~\eqref{eq:Rn}. Therefore we can conclude from Corollary~\ref{cor:graded-groupoids2} that for any sequence $((x_n,\chi_n))_n$ in $\Gu\times\widehat{\Gamma}$ the following conditions are equivalent:
\begin{enumerate}
  \item $\ker\pi_{(x_n,\chi_n)}\to\ker\pi_{(x,\chi)}$ in $\Prim C^*(\G)$;
  \item there exist points $y_n\in[x_n]$ such that $y_n\to x$ and $\chi_n\to\chi$ along $(H_\BB(y_n))_n$.
\end{enumerate}

A possible advantage of this formulation is that since the sets $H_\BB(y_n)$ are groups (for $n$ large enough), the convergence along them admits the following more transparent description.

\begin{lemma} \label{lem:equivdef}
Let $\Gamma$ be a countable discrete abelian group and $(S_{n})_{n}$ be a sequence of subgroups of $\Gamma$.  Then a sequence $(\chi_{n})_n$ in $\widehat{\Gamma}$ converges to $\chi\in \widehat{\Gamma}$ along $(S_{n})_n$ in the sense of Definition~\ref{def:newalongconv} if and only if there exist characters $\nu_n\in S_n^\perp$ such that $\chi_n\nu_n\to\chi$.
\end{lemma}

In other words, for sequences of subgroups of countable groups convergence in the sense of our Definition~\ref{def:newalongconv} is equivalent to that in the sense of \cite{BCS}*{Definition~9.1}.

\begin{proof}
It is clear that if $\nu_n\in S_n^\perp$ and $\chi_n\nu_n\to\chi$, then $(\chi_{n})_n$ converges to $\chi$ along $(S_{n})_n$.

To prove the other direction, assume that $\chi_{n} \to \chi$ along $(S_{n})_n$, but there are no characters~$\nu_n\in S_n^\perp$ such that $\chi_n\nu_n\to\chi$. This implies that by possibly passing to a subsequence we can choose an open neighbourhood $V\subset \widehat{\Gamma}$ of $\chi$ such that $\chi_{n} \notin V S_{n}^{\perp}$ for all $n$. By possibly passing to a subsequence again, we can assume that $S_{n} \to T$ in the Chabauty topology for some subgroup $T\subset \Gamma$. Every $\gamma \in T$ lies in $S_{n}$ for all $n$ large enough, so $\chi_{n}(\gamma)\to \chi(\gamma)$. We now use Lemma~\ref{lem:fellconv} with $H=S=\Gamma$ to find, after possibly passing to a subsequence, characters $\eta_n\in\widehat{\Gamma}$ such that $\eta_n\to\chi$ and $\eta_{n}|_{S_{n}}=\chi_{n}|_{S_{n}}$ for all $n$. Then, for all $n$ large enough, we have $\eta_{n}\in V$, and hence $\chi_{n}=\eta_n(\eta_n^{-1} \chi_{n}) \in V S_{n}^{\perp}$, which is a contradiction.
\end{proof}

This lemma and the discussion preceding it show that \cite{BCS}*{Theorem~9.5} follows from Corollary~\ref{cor:graded-groupoids2}.

\begin{remark}
In fact, we get a more precise description of convergence for groupoids admitting harmonious families of bisections compared to \cite{BCS}*{Theorem~9.5}, since in that theorem it is only proved that if $\ker\pi_{(x_n,\chi_n)}\to\ker\pi_{(x,\chi)}$ in $\Prim C^*(\G)$, then there exist~$(y_n,\eta_n)$ such that $\ker\pi_{(x_n,\chi_n)}=\ker\pi_{(y_n,\eta_n)}$ (equivalently, $\overline{[x_n]}=\overline{[y_n]}$ and $\chi_n=\eta_n$ on $\Phi(\IsoGx{x_n}^\circ_{x_n})=\Phi(\IsoGx{y_n}^\circ_{y_n})$), $y_n\to x$ and  $\eta_n\to\chi$ along $(H_\BB(y_n))_n$. In this regard we want to caution the reader that the first paragraph of the proof of \cite{BCS}*{Theorem 9.5} may suggest that if $\ker \pi_{(y_{n}, \eta_{n})}\to \ker\pi_{(x, \chi)}$ and $y_{n}\to x$, then one would always get that  $\eta_{n} \to \chi$ along $(H_\BB(y_{n}))_n$. However, when the authors invoke \cite{BCS}*{Theorem~7.1}, they potentially have to re-pick the sequence.
\end{remark}

\subsection{Deaconu--Renault groupoids}\label{subsecRD}
A rich supply of groupoids injectively graded by an abelian group is provided by the Deaconu--Renault groupoids defined by (partial) actions of group embeddable commutative monoids by local homeomorphisms. We will concentrate on the free abelian monoids~$\Z^k_+$, since this will be the setting of the subsequent sections. The corresponding groupoids for $k=1$ were introduced by Deaconu~\cite{D} and Renault~\cite{MR1770333}. For $k\ge2$, these groupoids  are also on occasion called higher-rank Deaconu--Renault groupoids in the literature. We follow~\cite{RW}*{Section 5} in our presentation.

\smallskip

We denote by $\Z_+$ the additive monoid $\{0,1,2,\dots\}$ of nonnegative integers. Fix $k\ge1$. Define the maximum $m \vee n$ of two elements $m,n\in \Z^k_+$ by taking the coordinate-wise maximum.

Assume that $X$ is a  Hausdorff locally compact space and, for each $n\in \Z_+^k$, we are given open subsets $\Dom(\sigma^{n})\subset X$ and $\Ran(\sigma^{n})\subset X$ and a local homeomorphism $\sigma^{n}$ from $\Dom(\sigma^{n})$ onto $\Ran(\sigma^{n})$ satisfying the following conditions:
\begin{itemize}
\item $\Dom(\sigma^{0}) = \Ran(\sigma^{0})=X$ and $\sigma^{0}=\text{id}_{X}$;
\item for all $n,m\in \Z_+^k$, we have $\Dom(\sigma^{m+n})=\Dom(\sigma^{n})\cap (\sigma^{n})^{-1}(\Dom(\sigma^{m}))$ and
$$
\sigma^{m+n}(x)=\sigma^{m}(\sigma^{n}(x)) \quad \text{for all}\quad x\in \Dom(\sigma^{m+n});
$$
\item for all $m,n\in \Z_+^k$, we have $\Dom(\sigma^{m})\cap \Dom(\sigma^{n})=\Dom(\sigma^{m\vee n})$.
\end{itemize}
We then say that $\Z_+^k\curvearrowright^\sigma X$ is a \emph{partial action} of $\Z^k_+$ on~$X$ by local homeomorphisms. We remark that in~\cite{RW} this is called a directed semigroup action.

Given such an action, we define a groupoid $\G_\sigma\subset X\times\Z^k\times X$ by setting
$$
\G_\sigma:=\{ (x,m-n,y) \in\Dom(\sigma^m)\times \Z^{k} \times\Dom(\sigma^n) : m,n\in \Z_+^k, \; \sigma^{m}(x)=\sigma^{n}(y) \},
$$
$$
r((x,p,y)):=(x,0,x),\quad s((x,p,y)):=(y,0,y)\quad \text{and}\quad (x,p,y)(y,q,w):=(x, p+q, w).
$$
We identify the unit space $\G_\sigma^{(0)}$ with $X$. The topology on $\G_\sigma$ is defined by using as a basis the sets of the form
$$
Z(U,m,n,V) : = \{ (x,m-n,y) \in \G_\sigma : x\in U\cap\Dom(\sigma^m),\ y\in V\cap\Dom(\sigma^n),\ \sigma^{m}(x) = \sigma^{n}(y) \},
$$
where $U,V \subset X$ are open subsets and $m,n \in \Z_+^k$. Equipped with this topology, $\G_\sigma$ becomes a locally compact Hausdorff \'etale groupoid, and the map
$$
\Phi\colon \G_\sigma \to \Z^{k},\quad (x,l,y)\mapsto l,
$$
defines an injective grading on $\G_\sigma$. If $X$ is second countable, then $\G_\sigma$ is also second countable and amenable, see~\cite{RW}*{Theorem~5.13}.

\smallskip

By construction the groupoid $\G_\sigma$ has a canonical system of open bisections. For an element $(x,m-n,x)\in (\G_\sigma)^x_x$, with $\sigma^m(x)=\sigma^n(x)$, we can in particular consider an open bisection $Z(U,m,n,U)$ containing it, where $U\subset\Dom(\sigma^m)\cap\Dom(\sigma^n)$ is an open neighbourhood of $x$ and the maps $\sigma^{m}$ and $\sigma^n$ are injective on $U$. 
As a consequence, in the formulations of results from Section~\ref{ssec:graded} we can use finite subsets of $\Z_+^k$ instead of finite collections of bisections. Then, for example, Theorem~\ref{thm:gradedgroupoids} takes the following form.

\begin{thm}\label{thm:DR}
Assume $\Z_+^k\curvearrowright^\sigma X$ is a partial action of $\Z^k_+$ by local homeomorphisms on a second countable Hausdorff locally compact space~$X$. Consider the corresponding Deaconu--Renault groupoid $\G_\sigma$. Then the map $X\times\T^k\to\Prim C^*(\G_\sigma)$, $(x,z)\mapsto\ker\pi_{(x,z)}$, is surjective, and the topology on $\Prim C^*(\G)$ is described as follows. For a sequence $\big((x(l),z(l))\big)^\infty_{l=1}$ in $X\times\T^k$, we have $\ker\pi_{(x(l),z(l))}\to\ker\pi_{(x,z)}$ if and only if for every neighbourhood $U$ of $x$, every $\eps>0$ and every finite set $\{(m(i),n(i))\}_{i=1}^{N} \subset \Z_+^{k} \times \Z_+^{k}$, with $\sigma^{m(i)}(x)=\sigma^{n(i)}(x)$  for all $i$, there exists an index $l_0$ such that for each $l\ge l_0$ we can find $y\in [x(l)]\cap U$ satisfying the following property: for every $i=1,\dots, N$, we have
$$
\text{either}\quad \sigma^{m(i)}(y)\ne\sigma^{n(i)}(y)\quad \text{or}\quad |z^{m(i)-n(i)}-z(l)^{m(i)-n(i)}|<\eps.
$$
\end{thm}

Here, for $z\in\T^k$ and $n\in\Z_+^k$, we let $z^n:=\prod^k_{j=1}z_j^{n_j}$. We also use the convention that the inequality $\sigma^{m}(y)\ne\sigma^{n}(y)$ is true if $y\notin \Dom(\sigma^{m})\cap\Dom(\sigma^{n})$. Note also that we may require in addition that the elements $m(i)-n(i)$ ($1\le i\le N$) are all different.

\smallskip

Ultimately, one might be more interested in understanding the closed subsets of $\Prim C^*(\G_\sigma)$ than convergence in this space, since such sets are in a bijective correspondence with the ideals of $C^*(\G_\sigma)$. Recall that, as we already observed in Remark~\ref{rem:invariant-open}, the map $\Ind\colon\Stab(\G_\sigma)\dach\to \Prim C^*(\G_\sigma)$ establishes a bijection between the open $\G_\sigma$-invariant subsets of $\Stab(\G_\sigma)\dach$ and the open subsets of $\Prim C^*(\G_\sigma)$. As a consequence, we get the following result.


\begin{thm} \label{thm:RDopen}
Assume $\Z_+^k\curvearrowright^\sigma X$ is a partial action of $\Z^k_+$ by local homeomorphisms on a second countable Hausdorff locally compact space~$X$. Consider the corresponding Deaconu--Renault groupoid $\G_\sigma$. For a subset $Y\subset X\times\T^k$, put $Y_{x}:=\{z\in \T^k: (x,z) \in Y\}$. Then~$Y$ is the pre-image of a $\G_\sigma$-invariant open subset of $\Stab(\G_\sigma)\dach$ under the map $q\colon X\times\T^k\to\Stab(\G_\sigma)\dach$ given by~\eqref{eq:q-on-Stab} if and only if $Y$ satisfies the following conditions:
\begin{enumerate}
\item[(i)] if $x\in\Dom(\sigma^n)$ for some $n\in\Z_+^k$, then $Y_{\sigma^n(x)}=Y_x$;
\item[(ii)] if $x\in X$ and $z\in Y_x$, then there exist $\varepsilon>0$, a neighbourhood $U\subset X$ of $x$ and a finite set $\{(m(i),n(i))\}_{i=1}^{N} \subset \Z_+^{k} \times \Z_+^{k}$, with $\sigma^{m(i)}(x)=\sigma^{n(i)}(x)$  for all $i$, such that the following property holds: for each $y\in U$, we have
$$
T_y:=\{ w \in \T^{k} : |z^{m(i)-n(i)}-w^{m(i)-n(i)}|<\varepsilon \text{ for all } i \text{ with } \sigma^{m(i)}(y)=\sigma^{n(i)}(y)\}\subset Y_y.
$$
\end{enumerate}

Therefore we get a bijection $Y\mapsto \bigcap_{(x,z)\in (X\times\T^k)\setminus Y} \ker\pi_{(x,z)}$ between the subsets $Y\subset X\times \T^{k}$ as above and the ideals of $C^{*}(\G_\sigma)$.
\end{thm}

Here we use the convention that if $N=0$ or $\sigma^{m(i)}(y)\ne\sigma^{n(i)}(y)$ for all $i$, then condition (ii) says that $Y_y=\T^k$. Note again that we may require in addition that the elements $m(i)-n(i)$ ($1\le i\le N$) are all different.

\bp
Condition (i) says simply that $Y$ is $\G_\sigma$-invariant. Therefore we need to show only that for every $\G_\sigma$-invariant subset $Y\subset X\times\T^k$ condition (ii) is satisfied if and only if $Y=q^{-1}(q(Y))$ and~$q(Y)$ is open in $\Stab(\G_\sigma)\dach$. Assuming that $Y=q^{-1}(q(Y))$, recall from Lemma~\ref{lem:Char-topology} that a basis of neighbourhoods of $q(x,z)$ is given by the sets $\UU_x^{\chi_z}(U,\eps,(W_g)_{g\in F})$, where we denote by~$\chi_z$ the character of $(\G_\sigma)^x_x$ defined by~$z$. Taking as our bisections $W_g$ sets $Z(U,m,n,U)$ and untangling the definitions one sees that condition (ii) is equivalent to saying that $\UU_x^{\chi_z}(U,\eps,(W_g)_{g\in F})\subset q(Y)$.

Finally, observe that once condition (ii) is satisfied, applying it to $y=x$ we see that with every $z\in Y_x$ the set $Y_x$ contains all elements $w\in\T^k$ that define the same character as~$z$ on~$(\G_\sigma)^x_x$. Therefore condition (ii) does imply that $Y=q^{-1}(q(Y))$.
\ep

\begin{remark}\label{rem:Yy}
By Remark~\ref{rem:invariant-open}, given a $\G_\sigma$-invariant open set $Y=q^{-1}(q(Y))$, we could have used in the above proof the larger (and not necessarily open) sets $q(\tilde\UU_x^{z}(U,\eps,(W_g)_{g\in F}))$ from that remark instead of $\UU_x^{\chi_z}(U,\eps,(W_g)_{g\in F})$ as a basis of neighbourhoods of $q(x,z)$. The conclusion would be that in condition (ii) of Theorem~\ref{thm:RDopen} we may in addition require that $m(i)-n(i)\in\Phi(\IsoGx{x}^\circ_x)$ for all~$i$ and in the definition of $T_y$ we may consider only~$i$ such that $m(i)-n(i)\in\Phi(\IsoGx{y}^\circ_y)$, where $\G:=\G_\sigma$.
\end{remark}

\subsection{Singly generated dynamical systems}\label{ssec:Katsura}
In this subsection we consider the Deaconu--Renault groupoids defined by partial actions of $\Z_+$, that is, when the action is defined by one partially defined local homeomorphism. In this case the topology on the primitive spectrum has been recently described by Katsura~\cite{Kat}. Since he does not use groupoids, in order to compare his results with ours let us first say a few words about the connection between the two settings.

\smallskip

Let $X$ be a Hausdorff locally compact space and $\sigma$ be a local homeomorphism of $\Dom(\sigma)\subset X$ onto $\Ran(\sigma)\subset X$. Katsura associates a C$^{*}$-algebra $C^{*}(X, \sigma)$ to such a local homeomorphism by considering a universal C$^{*}$-algebra generated by the images of a $*$-homomorphism $t^{0}: C_{0}(X) \to C^{*}(X, \sigma)$ and a linear map $t^{1}:C_{c}(\Dom(\sigma)) \to C^{*}(X, \sigma)$ satisfying certain conditions \cite{Kat}*{Definition 1.4}. For each pair $(x, z) \in X\times \T$ a representation $\pi_{(x, z)}: C^{*}(X, \sigma) \to B(\ell^{2}([x]))$ is then introduced by constructing a representation $C_{0}(X) \to B(\ell^{2}([x]))$ and a linear map $C_{c}(\Dom(\sigma)) \to B(\ell^{2}([x]))$ \cite{Kat}*{Definitions~2.3, 2.6}.

The C$^*$-algebra $C^{*}(X, \sigma)$ is known to be isomorphic to the C$^*$-algebra of the Deaconu--Renault groupoid~$\G_\sigma$ associated to~$\sigma$. Namely, the isomorphism arises from the canonical isomorphism $C_0(X)\cong C_0(\G^{(0)}_\sigma)$ and the linear map $t^{1}\colon C_{c}(\Dom(\sigma)) \to C^{*}(\G_\sigma)$, given by
$$
t^{1}(f)(x,n,y) =
\begin{cases}
f(x),  &  \text{if } n=1, \ y=\sigma(x), \\
0,  &  \text{otherwise}.
\end{cases}
$$
It is then easy to check that under this isomorphism Katsura's representations $\pi_{(x, z)}$ become exactly the representations of $C^*(\G_\sigma)$ defined by~\eqref{eq:pi(x,z)}.

\smallskip

The main result of~\cite{Kat} is a description of the closed subsets of $\Prim C^*(X,\sigma)$ in terms of the representations $\pi_{(x,z)}$. As we discussed in Section~\ref{subsecRD}, assuming that $X$ is second countable, the map $\Ind\colon\Stab(\G_\sigma)\dach\to \Prim C^*(\G_\sigma)$ establishes a bijection between the closed $\G_\sigma$-invariant subsets of $\Stab(\G_\sigma)\dach$ and the closed subsets of $\Prim C^*(\G_\sigma)$. Therefore from this perspective \cite{Kat} describes the pre-images of closed $\G_\sigma$-invariant subsets of $\Stab(\G_\sigma)\dach$ in $X\times\T$ under the map $q\colon X\times\T\to\Stab(\G_\sigma)\dach$ given by~\eqref{eq:q-on-Stab}. We are going to show how to obtain this description directly from the definition of $\Stab(\G_\sigma)\dach$. We will need the following notation to formulate the result.

\begin{defn}[\cite{Kat}*{Definition 2.13}]
Let $x\in X$. If there exists $p\ge1$ such that $\sigma^{n+p}(x)=\sigma^{n}(x)$ for some $n\ge0$, then we say that $x$ is periodic, and we define its period $p(x)$ to be the smallest possible such $p$. If $x$ is periodic, we denote by $l(x)$ the smallest number $l\ge0$ satisfying $\sigma^{p(x)+l}(x)=\sigma^{l}(x)$. If $x$ is not periodic, we set $p(x)=l(x)=\infty$.
\end{defn}

Note that if $x$ is periodic, then $x$ lies in the domain of definition of $\sigma^n$ for all $n\ge0$ and $l(x)$ is the number of elements that appear only finitely many times in the sequence $x, \sigma(x), \sigma^{2}(x),\dots $. It follows that $l(x)$ is the smallest number $l\ge0$ such that $\sigma^{p+l}(x)=\sigma^{l}(x)$ for some $p\ge1$.

\begin{thm}[{cf.~\cite{Kat}*{Theorem~7.8}}]\label{thm:Kat}
Assume $X$ is a Hausdorff  locally compact space and $\sigma\colon \Dom(\sigma)\subset X\to\Ran(\sigma)\subset X$ is a partially defined local homeomorphism of $X$. Consider the corresponding Deaconu--Renault groupoid~$\G_\sigma$. For a subset $Y\subset X\times\T$, put $Y_{x}:=\{z\in \T:(x,z) \in Y\}$. Then~$Y$ is the pre-image of a $\G_\sigma$-invariant closed subset of $\Stab(\G_\sigma)\dach$ under the map $q\colon X\times\T\to\Stab(\G_\sigma)\dach$ given by~\eqref{eq:q-on-Stab} if and only if $Y$ satisfies the following conditions:
\begin{enumerate}
\item[(i)] $Y$ is a closed subset of $X\times \T$ with respect to the product topology;
\item[(ii)] $Y_{x} = Y_{\sigma(x)}$ for all $x\in \Dom(\sigma)$;
\item[(iii)] if $Y_{x_{0}}\neq \emptyset, \T$, then $x_{0}$ is periodic, $e^{2 \pi i/ p(x_{0})} Y_{x_{0}} = Y_{x_{0}}$ and there exists a neighbourhood~$V$ of~$x_{0}$ such that for all $x\in V$ with $l(x)\neq l(x_{0})$ we have $Y_{x}=\emptyset$.
\end{enumerate}
\end{thm}

Therefore if $X$ is in addition second countable, then this theorem together with Corollary~\ref{cor:kappa} give a classification of ideals of $C^*(\G_\sigma)\cong C^*(X,\sigma)$ in terms of subsets of $X\times\T$, recovering in this case the result of~\cite{Kat}.

\smallskip

Before we turn to the proof, let us make a few observations about the topology on $\Stab(\G_\sigma)\dach$.

\begin{lemma} \label{lem:convSGDS}
Let $(x,z)\in X\times\T$ and $((x_i,z_i))_i$ be a net in $X\times\T$ such that $x_i\to x$. Then
\begin{enumerate}
\item if $p(x) = \infty$ or $p(x_{i}) \to \infty$, then $q(x_i,z_i)\to q(x,z)$ in $\Stab(\G_\sigma)\dach$;
\item if $l(x_{i})> l(x)$ for all $i$, then $q(x_i,z_i)\to q(x,z)$;
\item if $q(x_i,z_i)\to q(x,z)$, $x$ is periodic, $p(x_{i})=p$ and $l(x_{i})=l(x)$ for all $i$ and some $p\ge1$, then $p(x)$ divides $p$ and $z_{i}^p\to z^p$.
\end{enumerate}
\end{lemma}

\begin{proof}
Let us write $\G$ for $\G_\sigma$. If $p(x)=\infty$, then $\Gxx=\{x\}$ and we have $q(x_i,z_i)\to q(x,z)$ by Corollary~\ref{cor:char-convergence}. Assume now that $p(x)<\infty$. Then $\Phi(\Gxx)=p(x)\Z$. Consider the open bisections
$$
W_{mp(x)}:=
\begin{cases}
Z(U_m, l(x)+mp(x), l(x), U_m),& \  \text{if $m\geq 0$,} \\
Z(U_m, l(x), l(x)-mp(x), U_m), & \  \text{if $m< 0$,}
\end{cases}
$$
containing the elements of $\Gxx$, where $U_m\subset \Dom(\sigma^{l(x)+|m|p(x)})$ is a fixed open neighbourhood of~$x$ such that $\sigma^{l(x)+|m|p(x)}$ is injective on $U_m$. If $m\ne0$ and $p(x_{i}) \to \infty$, then $W_{mp(x)} \cap \G_{x_{i}}^{x_{i}}=\emptyset$ for all sufficiently large $i$, so by Corollary~\ref{cor:char-convergence} we have $q(x_i,z_i)\to q(x,z)$. This proves (1).

If $W_{mp(x)} \cap \G_{x_{i}}^{x_{i}}\neq \emptyset$ for some $m\ne0$ and $i$, then $\sigma^{l(x)+|m|p(x)}(x_{i})=\sigma^{l(x)}(x_{i})$. By the observation before Theorem~\ref{thm:Kat} it follows that $l(x_{i}) \leq l(x)$. Therefore if $l(x_{i})> l(x)$ for all $i$, then $W_{mp(x)} \cap \G_{x_{i}}^{x_{i}}= \emptyset$ for all $m\ne0$ and we again get $q(x_i,z_i)\to q(x,z)$. This proves (2).

In order to prove (3), notice that since by assumption we have $\sigma^{p+l(x)}(x_{i}) = \sigma^{l(x)}(x_{i})$ for all~$i$, we get $\sigma^{p+l(x)}(x) = \sigma^{l(x)}(x)$. Hence $p(x)$ divides $p$ and $\G_{x_{i}}^{x_{i}} \cap W_{p} \neq \emptyset$ for all $i$ sufficiently large. Then $z_i^p\to z^p$ by Corollary~\ref{cor:char-convergence}.
\end{proof}

\bp[Proof of Theorem~\ref{thm:Kat}]
We again write $\G$ for $\G_\sigma$. Condition (ii) in the statement of the theorem is equivalent to $\G$-invariance of $Y$. Since the map $q\colon X\times\T\to\Stab(\G)\dach$ is $\G$-equivariant, it follows that in order to prove the theorem it suffices to show that the pre-images of closed sets are characterized by conditions (i) and (iii).

\smallskip

Assume first that $Y \subset X\times \T$ is the pre-image of a closed set. Since $q\colon X\times\T\to\Stab(\G_\sigma)\dach$ is continuous, condition (i) is obviously satisfied.
Assume $x_0\in X$ is such that $Y_{x_{0}} \neq \emptyset, \T$. If the group $\G_{x_{0}}^{x_{0}}$ was trivial, all numbers $z\in \T$ would induce the same character on $\G_{x_{0}}^{x_{0}}$ and we would have $Y_{x_{0}} =\T$. It follows that $\G_{x_{0}}^{x_{0}}$ is nontrivial and hence $x_{0}$ is periodic. Since $\Phi(\G_{x_{0}}^{x_{0}}) = p(x_{0}) \Z$, the numbers $e^{2 \pi i/ p(x_{0})} z$ and $z$ define the same character on $\Phi(\G_{x_{0}}^{x_{0}})$ for all $z\in\T$, which implies that $e^{2 \pi i/ p(x_{0})} Y_{x_{0}} = Y_{x_{0}}$.

To verify the last condition in (iii), assume for a contradiction that one can find a net $(x_{i})_i$ with $x_{i} \to x_{0}$, $l(x_{i})\neq l(x_{0})$ and $Y_{x_{i}} \neq \emptyset$ for all $i$. The net $(p(x_i))_i$ must be eventually bounded, since otherwise we would get $Y_{x_{0}}=\T$ by Lemma \ref{lem:convSGDS}(1). Therefore by passing to a subnet we may assume that $p(x_{i})=p$ for all $i$ and some $p\ge1$. By passing to a subnet we can then also assume that either $l(x_{i})=l$ for all $i$ and some $l<l(x_0)$ or $l(x_{i})>l(x_{0})$ for all $i$. If $l(x_{i})>l(x_{0})$ for all~$i$, then Lemma \ref{lem:convSGDS}(2) implies that $Y_{x_{0}}=\T$, giving a contradiction. Therefore $l(x_{i})=l<l(x_0)$ for all~$i$. Then $\sigma^{l+p}(x_{i})=\sigma^{l}(x_{i})$, and hence by continuity $\sigma^{l+p}(x_{0})=\sigma^{l}(x_{0})$, contradicting that $l<l(x_{0})$. Thus we reach a contradiction in both cases, which proves that condition (iii) holds true for $Y$.

\smallskip

Conversely, assume that $Y\subset X \times \T$ satisfies conditions (i) and (iii). 
Assume $((x_j, z_j))_j$ is a net in~$Y$ such that $q(x_j,z_j)\to q(x,z)$ for some $(x,z)\in X\times\T$. We need to show that $(x,z)\in Y$, as then we can conclude that $q(Y)$ is closed and $Y=q^{-1}(q(Y))$. By definition we have $x_j\to x$. By passing to a subnet we may assume that $z_j \to w$ for some $w\in \T$. Then $w\in Y_{x}$ by condition~(i). If $Y_{x}=\T$, then $(x,z) \in Y$ and we are done, so assume $Y_{x}\neq\T$. By condition (iii) we may then assume that $l(x_j)=l(x)$ for all $j$. The net $(p(x_j))_j$ must be eventually bounded, since otherwise using the property $e^{2 \pi i/ p(x_j)} Y_{x_j} = Y_{x_j}$, which follows from condition (iii), we would get $Y_x=\T$ by condition~(i). Therefore by passing to a subnet we may assume that $p(x_j)=p$ for all $j$ and some $p\ge1$.

By Lemma~\ref{lem:convSGDS}(3) we can conclude now that $z_j^p\to z^p$. Hence ${w}^p=z^p$, so $z=we^{2\pi i l/p}$ for some $l\ge0$. By condition (iii) we have $z_je^{2\pi i l/p} \in Y_{x_j}$ for all $j$, hence by condition (i) we get that $z= w e^{2\pi i l/p} \in Y_{x}$, proving that $(x,z)\in Y$.
\ep

\begin{remark}
If $X$ is not second countable, then the map $\Ind\colon \Stab(\G_\sigma)\dach\to\Prim C^*(\G_\sigma)$ might be nonsurjective, but Corollary~\ref{cor:Ind2} implies that we still have a one-to-one correspondence between the closed subsets of $\Ind(\Stab(\G_\sigma)\dach)\subset\Prim C^*(\G_\sigma)$ (in the relative topology) and the $\G_\sigma$-invariant closed subsets of $\Stab(\G_\sigma)\dach$. Therefore we get a one-to-one correspondence between the closed subsets of $\Ind(\Stab(\G_\sigma)\dach)$ and the subsets of $X\times\T$ satisfying conditions (i)--(iii) of Theorem~\ref{thm:Kat}. In order to obtain a full classification of closed subsets of $\Prim C^*(\G_\sigma)$ from this, as in \cite{Kat}, it remains to show that for every closed subset $A\subset \Prim C^*(\G_\sigma)$ we have
$$
A=\overline{A\cap \Ind(\Stab(\G_\sigma)\dach)}.
$$
Equivalently, every primitive ideal in $C^*(\G_\sigma)$ is an intersection of ideals $\ker\pi_{(x,z)}$. This is a property established in~\cite{Kat}*{Corollary~4.19}.
\end{remark}

\bigskip

\section{Graph algebras}\label{sec:Graph}

\subsection{\texorpdfstring{$1$}{1}-graphs}\label{ssec:1-graphs}
The primitive ideal space for Cuntz-Krieger C$^{*}$-algebras of directed graphs has been completely described by Hong and Szyma\'nski~\cite{HS}, see also~\cite{Gabe} for a correction. In this subsection we propose an equivalent description obtained entirely using the groupoid picture for these C$^*$-algebras. For the case of row-finite graphs without sources, see also~\citelist{\cite{CS}\cite{BCS}}.

\smallskip

Our starting point is an observation about singly generated dynamical systems. In order to formulate the result we introduce the following notation.

\begin{defn}\label{def:nice_points}
Given a partially defined local homeomorphism $\sigma\colon\Dom(\sigma)\to\Ran(\sigma)$ of a Hausdorff locally compact space $X$, denote by $A(\sigma)\subset X$ the set of aperiodic points and by $P_0(\sigma)\subset X$ the set of periodic points $x$ that are isolated in~$[x]$.
\end{defn}

In other words, $x\in A(\sigma)$ if and only if there are no $l\ge0$ and $p\ge1$ such that $\sigma^{l+p}(x)=\sigma^l(x)$, and $x\in P_0(\sigma)$ if and only if $\sigma^{l+p}(x)=\sigma^l(x)$ for some $l\ge0$ and $p\ge1$ and there exists a neighbourhood $U$ of $x$ such that if $\sigma^m(y)=\sigma^n(x)$ for some $y\in U$ and $m,n\ge0$, then $y=x$. Note that if $\sigma$ is injective, then $P_0(\sigma)$ is simply the set of periodic points.

Consider the associated Deaconu--Renault groupoid $\G_\sigma$.

\begin{lemma} \label{lem:essential=Gxx}
The sets $A(\sigma)$ and $P_0(\sigma)$ are $\G_\sigma$-invariant subsets of $\G^{(0)}_\sigma=X$, and their union is the set of points $x$ such that $\operatorname{Iso}((\G_\sigma)_{\overline{[x]}})^\circ_x=(\G_\sigma)^x_x$.
\end{lemma}

\begin{proof}
Write $\G$ for $\G_\sigma$. Since
$$
A(\sigma)=\{x:\Gxx=\{x\}\}\quad\text{and}\quad P_0(\sigma)=\{x:\Gxx\ne\{x\}\ \text{and}\ x\ \text{is isolated in}\ [x]\},
$$
it is clear that the sets $A(\sigma)$ and $P_0(\sigma)$ are invariant.

It is also clear by definition that $A(\sigma)\cup P_0(\sigma)$ is contained in the set of points $x$ such that $\IsoGx{x}^\circ_x=\Gxx$. In order to prove the equality we need to show that if $x$ is periodic and $\IsoGx{x}^\circ_x=\Gxx$, then $x\in P_0(\sigma)$. Since the set $P_0(\sigma)$ is invariant, we can further assume that $\sigma^{p}(x)=x$, where $p:=p(x)$ is the period of $x$. Then, by definition, we can find open neighbourhoods $U$  and $V$ of~$x$ and a number $n\ge 0$ such that
$$
Z(U, p+n, n,V) \cap \G_{\overline{[x]}}\subset\IsoG.
$$
We may assume that $U\subset\Dom(\sigma^{p+n})$ and $V\subset\Dom(\sigma^n)$. Then $Z(U, p, 0, V) \subset Z(U, p+n, n, V)$, so $Z(U, p, 0, V) \cap \G_{\overline{[x]}}$ consists entirely of isotropy. Since $x,\sigma(x),\dots, \sigma^{p-1}(x)$ are different elements, we may, by possibly choosing a smaller $U$, assume that $U\cap \sigma^{j}(U)=\emptyset$ for $1\leq  j <p$ and $\sigma^{p}(U) \subset V$.

Suppose now that $y\in U\cap [x]$. Then $(y, p, \sigma^{p}(y))\in Z(U, p, 0, V) \cap \G_{\overline{[x]}}$. Since the last set consists of isotropy, we get $\sigma^{p}(y)=y$. Since $\sigma^{p}(x)=x$ and $y\in  [x]$, it follows that $y=\sigma^{l}(x)$ for some $0\le l< p$. But then $y\in U\cap \sigma^{l}(U)$, implying that $l=0$ and $y=x$. In conclusion, $U\cap [x]=\{x\}$, so $x\in P_0(\sigma)$.
\end{proof}

We will apply this lemma in the special case of directed graphs. We refer the reader to \cite{BCW}*{Section 2} for more background and proofs regarding the groupoid model for Cuntz--Krieger algebras of directed graphs, but note that in order to be consistent with the next subsection we follow the ``Australian convention'' that swaps the roles of sources and ranges.

Let $E=(E^{0},E^{1}, r,s)$ denote a countable directed graph, i.e., $E^{0}$ is a countable set of vertices, $E^{1}$ is a countable set of edges and $s,r\colon E^{1} \to E^{0}$ denote respectively the source and range maps. Define
$$
E^\sing:=\{v\in E^{0} : |r^{-1}(v)| \in \{0, \infty\}\}.
$$

A finite path $e_{1} \cdots e_{n}$ of length $n\ge1$ is a concatenation of edges with $s(e_{i})=r(e_{i+1})$ for all $i<n$, and an infinite path $e_{1}e_{2} \cdots $ is an infinite concatenation of edges with $s(e_{i})=r(e_{i+1})$ for all $i$. We view the vertices of $E$ as paths of length $0$. We denote the set of finite paths by~$E^{*}$ and the set of infinite paths by~$E^{\infty}$, and we extend the range map to both sets by letting $r(e_{1}e_{2} \cdots):=r(e_{1})$ for paths of length $\ge1$ and $r(v):=v$ for $v\in E^0\subset E^*$. We also extend the source map~$s$ to~$E^{*}$ by $s(e_{1} \cdots e_{n}):=s(e_{n})$ for $n\ge1$ and $s(v):=v$ for $v\in E^0\subset E^*$.

Consider the set
$$
\partial E:=E^{\infty} \cup \{ \alpha \in E^{*} : s(\alpha)\in E^{\sing}\}
$$
of so-called boundary paths. For $\alpha\in E^*$, we denote by $Z(\alpha)$ the sets of paths in $\partial E$ of the form~$\alpha x$, with $x=\emptyset$ or $x\in\partial E$ satisfying $r(x)=s(\alpha)$. The set $\partial E$ is a second countable Hausdorff locally compact space with a basis of topology given by the sets
$$
Z(\alpha)\setminus\bigcup_{e\in F}Z(\alpha e),
$$
where $F\subset E^1$ is a finite (possibly empty) subset of edges in $r^{-1}(s(\alpha))$. Each of the open sets~$Z(\alpha)$ is compact in this topology.

The shift map $\sigma_E\colon \partial E\setminus E^{\sing}\to\partial E  $ is defined on paths of length $\ge 2$ by $\sigma_E(e_1e_2\cdots):=e_2\cdots$, and on paths of length $1$ by $\sigma_E(e):=s(e)$. It is a local homeomorphism, and we denote the corresponding Deaconu--Renault groupoid by $\G_{E}$. The C$^{*}$-algebra $C^{*}(\G_{E})$ is the Cuntz--Krieger algebra $C^{*}(E)$ of $E$.

\smallskip

Define a preorder on $E^{0}$ by declaring that $v\leq w$ iff there exists $\alpha \in E^{*}$ with $s(\alpha) = v$ and $r(\alpha)=w$. We then get an equivalence relation on $E^{0}$ by declaring that $v\sim w$ iff $v\leq w$ and $w\leq v$. An equivalence class in $E^{0}$ is called a \emph{component}.

\begin{defn}
A \emph{primitive loop} in $E$ is a finite component $L\subset E^{0}$ such that every vertex $v\in L$ is the range of exactly one edge originating in $L$, that is,
$$
|\{ e\in E^1 : r(e)=v\ \text{and}\ s(e)\in L\}|=1.
$$
Denote by $\LL(E)$ the set of primitive loops in $E$.
\end{defn}

In other words, the vertices of a simple cycle $e_1\cdots e_p$ in~$E$ (that is, $e_1\cdots e_p$ is a path such that $r(e_1)=s(e_p)$ and $r(e_{i}) \neq r(e_{j})$ for $i\neq j$) form a primitive loop if and only if the only paths of positive length from $r(e_1)$ to itself are the powers of $e_1\cdots e_p$.

The role of primitive loops is explained by the following lemma.

\begin{lemma} \label{lem:simpleloop}
A periodic path $x=x_{1} x_{2}\cdots\in \partial E$ with period $p:=p(x)$ and $\sigma_E^p(x)=x$ has the property that $x$ is isolated in $[x]$ if and only if the vertices $r(x_1),\dots, r(x_{p})$ form a primitive loop.
\end{lemma}

\begin{proof}
If $r(x_1),\dots, r(x_{p})$ form a primitive loop, then it is straightforward to check that $Z(r(x))\cap [x]=\{x\}$. Conversely, assume $x$ is isolated in $[x]$. Consider $\alpha:=x_1x_2\cdots x_{p}$, then  $x=\alpha^{\infty}$. Choose $k\ge1$ big enough that $Z(\alpha^{k})\cap [x]=\{x\}$. Assume $L:=\{r(x_1),\dots,r(x_p)\}$
is not a primitive loop. Then, for some $i,j\in\{1,\dots, p\}$, there exists a path $e_{1} \cdots e_{n}\in E^{*}$ with $r(e_{1})=r(x_{i})$, $s(e_{n})=r(x_{j})$ and $e_{1}\neq x_{i}$. Then the path $y:=\alpha^{k}x_{1} \cdots x_{i-1}e_{1}\cdots e_{n} x_{j}x_{j+1}x_{j+2}\cdots $ satisfies $y\in Z(\alpha^{k})\cap [x]$ and $y\ne x$, which is a contradiction. This proves that $L$ is a primitive loop.
\end{proof}

For every $L\in\LL(E)$, fix an infinite path $x_L=x_1x_2\cdots$ such that $r(x_i)\in L$ for all $i$. Then $\sigma_E^{|L|}(x_L)=x_L$ and any other choice of such a path has the form $\sigma_E^k(x_L)$ for some $0\le k<|L|$.

\smallskip

We are now ready to give a preliminary description of $\Prim C^*(\G_E)$.

\begin{prop} \label{prop:graph}
For every countable directed graph $E$, the map
$$
\Big(A(\sigma_E)\sqcup\bigsqcup_{L\in\LL(E)}[x_L]\Big)\times \T \to \Prim C^{*}(\G_E),\quad (x,z)\mapsto \ker\pi_{(x,z)},
$$
is onto, and $ \ker\pi_{(x,z)}=\ker\pi_{(y,w)}$ if and only if either $x,y\in A(\sigma_E)$ and $\overline{[x]}=\overline{[y]}$, or $x,y\in [x_L]$ for some $L\in\LL(E)$ and $z^{|L|}=w^{|L|}$. The topology on  $\Prim C^{*}(\G_E)$ is described as follows. Consider a sequence $((x_n,z_n))_n$ and an element $(x,z)$ in $\Big(A(\sigma_E)\sqcup\bigsqcup_{L\in\LL(E)}[x_L]\Big)\times \T$. Then we have:
\begin{enumerate}
\item[(i)] if $x\in A(\sigma_E)$, then $\ker \pi_{(x_n,z_n)} \to \ker \pi_{(x,z)}$ if and only if there exist $y_{n}\in [x_{n}]$ with $y_{n} \to x$;
\item[(ii)] if $x\in[x_L]$ for some $L\in\LL(E)$ and $x_{n} \notin [x_L]$ for all $n$, then $\ker \pi_{(x_n,z_n)} \to \ker \pi_{(x,z)}$ if and only if there exist $y_{n}\in [x_{n}]$ with $y_{n} \to x$;
\item[(iii)] if $x\in [x_L]$ for some $L\in\LL(E)$ and $x_{n} \in [x_L]$ for all $n$, then $\ker \pi_{(x_n,z_n)} \to \ker \pi_{(x,z)}$ if and only if  $z_n^{|L|}\to z^{|L|}$.
\end{enumerate}
\end{prop}

Note that, given $L$, for any sequence $((x_n,z_n))_n$, we can discard finitely many elements and divide the rest into two subsequences such that one satisfies $x_{n} \notin [x_L]$ for all $n$, while the other satisfies $x_{n} \in [x_L]$ for all $n$.  Hence Proposition~\ref{prop:graph} completely describes the topology on $\Prim C^*(\G_E)$.

\begin{proof}[Proof of Proposition~\ref{prop:graph}]
Let us write $\G$ for $\G_E$ and $\sigma$ for $\sigma_E$. The map in the formulation is surjective by Corollary~\ref{cor:essential-iso2} and the description of the points $x$ such that $\IsoGx{x}^\circ_x=\Gxx$ provided by Lemmas~\ref{lem:essential=Gxx} and~\ref{lem:simpleloop}. The characterization of the equality $\ker\pi_{(x,z)}=\ker\pi_{(y,w)}$ follows from Corollary~\ref{cor:essential-iso1} once one observes that if $\overline{[x_{L}]}=\overline{[x]}$ for some $x$, then $x\in [x_{L}]$, since $x_L$ is isolated in $\overline{[x_L]}$.

Since $\Gxx=\{x\}$ for $x\in A(\sigma)$, part (i) follows immediately from Theorem~\ref{thm:DR}. From the same theorem we also get (iii), since every $x\in[x_L]$ is isolated in $[x_L]$.

To prove (ii), assume that $x\in[x_{L}]$. Then $\Phi(\Gxx)=|L|\Z$. Assume $m\ge1$ and $l\ge0$ are such that $\sigma^{m|L|+l}(x)=\sigma^l(x)$. If $x=e_1e_2\cdots$, let $U:=Z(e_{1} \cdots e_{m|L|+l})$. Then $\sigma^{m|L|+l}(y)\ne\sigma^l(y)$ for all $y\in U\setminus\{x\}$. By applying Theorem~\ref{thm:DR} we get~(ii).
\end{proof}

It remains to describe the closures $\overline{[x]}$ for $x\in A(\sigma_E)$ in graph-theoretic terms.

\begin{defn}[{\cite{HS}}]\label{def:max-tails}
A nonempty subset $M\subset E^0$ is called a \emph{maximal tail} if the following three conditions are satisfied:
\begin{enumerate}
  \item[(i)] if $v\in E^0$, $w\in M$ and $v\ge w$, then $v\in M$;
  \item[(ii)] if $v\in M$ and $0<|\{e\in E^1:r(e)=v\}|<\infty$, then there is $e\in E^1$ such that $r(e) = v$ and $s(e)\in M$;
  \item[(iii)] for every $v, w\in M$, there exists $u\in M$ such that $v\ge u$ and $w\ge u$.
\end{enumerate}
Denote by $\M(E)$ the set of maximal tails. Denote by $\M_\gamma(E)\subset \M(E)$ the subset of all maximal tails $M$ such that for each simple cycle $e_1\cdots e_p$ in $E$ with vertices in $M$ there is an edge $e\in E^1$ such that $e\ne e_i$ for all $i$, $r(e)=r(e_j)$ for some $j$ and $s(e)\in M$. A bit informally we formulate this by saying that every simple cycle in $M$ has an entrance in $M$.
\end{defn}

\begin{lemma}\label{lem:tails}
We have a well-defined map $\MT\colon A(\sigma_E)\to\M(E)$ that associates to $x\in A(\sigma_E)$ the set $\MT(x):=\{r(y):y\in[x]\}$. Then
$$
\MT(A(\sigma_E)\cap E^\infty)\subset\M_\gamma(E)\subset \MT(A(\sigma_E)).
$$
If $\overline{[x]}=\overline{[y]}$, then $\MT(x)=\MT(y)$, and the converse is true if both paths $x$ and $y$ are infinite.
\end{lemma}

\bp
Given $x\in A(\sigma_E)$, it is easy to see that properties (i) and (iii) in Definition~\ref{def:max-tails} are satisfied for $\MT(x)$. Property (ii) is also satisfied, because if $0<|\{e\in E^1|:r(e)=v\}|<\infty$ for some $v$, then $v\notin \partial E$.

Next, we need to show that if $x$ is infinite, then $\MT(x)\in\M_\gamma(E)$. Assume this is not the case, so there is a simple cycle $\alpha=\alpha_1\cdots\alpha_p\in E^*$ with vertices in~$\MT(x)$ that does not have an entrance in~$\MT(x)$. This implies that if $y\in[x]$ and $r(y)=r(\alpha_1)$, then $y=\alpha^\infty$, contradicting the assumption that $x$ is aperiodic.

The claim that $\MT(x)=\MT(y)$ when $\overline{[x]}=\overline{[y]}$ is obvious from the fact that the map $r\colon\partial E\to E^0$ is continuous if $E^0$ is considered as a discrete space. Conversely, assume that $\MT(x)=\MT(y)$. If both $x$ and $y$ are infinite and $x=\alpha x'$ for some $\alpha\in E^*$ and $x'\in E^\infty$, then $r(x')=r(y')$ for some $y'\in[y]$, hence $\alpha y'\in[y]$. This implies that $x\in\overline{[y]}$. For the same reason $y\in\overline{[x]}$.

It remains to show that if $M\in\M_\gamma(E)$, then there is $x\in A(\sigma_E)$ with $\MT(x)=M$. Consider three cases.

Assume first that $M$ contains a unique least element $v$, that is, $v\leq w$ for all $w\in M$ and $v$ is the only element with this property in $M$. If $v\in E^\sing$, we simply take $x=v$. If $v\notin E^\sing$, then by minimality of $v$ and property (ii) of maximal tails there is a self-loop at $v$, that is, an edge from~$v$ to~$v$. As $M\in M_\gamma(E)$, there in fact must be at least two such self-loops $e_1, e_2$. Then for~$x$ we take any aperiodic infinite path obtained by concatenating $e_1$ and $e_2$.

Next, assume that $M$ has two or more least elements, say, $v$ and $u$. Then there is a simple cycle~$\alpha\in E^*$ passing through $u$ and starting and ending at $v$. As $M\in M_\gamma(E)$, we can then find another (not necessarily simple) cycle $\beta$ such that $r(\beta)=s(\beta)=v$ and $\beta$ is not a power of $\alpha$. Then for $x$ we take any aperiodic infinite path obtained by concatenating the cycles~$\alpha$ and~$\beta$.

Finally, assume that $M$ does not have a least element. Using property (iii) of maximal tails we can find vertices $v_n\in M$ ($n\ge1$) such that $v_n\ge v_{n+1}$ for all $n$ and for every $v\in M$ we have $v\ge v_n$ for $n$ large enough. Then as $x$ we take any infinite path obtained by concatenating paths from $v_{n+1}$ to $v_n$ for all $n$. Such a path is aperiodic, since it passes through infinitely many different vertices.
\ep


In order to fully describe the orbit closures of vertices in $E^\sing$ we will need the following notion.

\begin{defn}[{\cite{MR1988256}}]
A vertex $v\in E^\sing$ is called a \emph{breaking vertex} if
$$
0<|\{e\in E^1: r(e)=v,\ s(e)\ge v\}|<\infty.
$$
Denote by $BV(E)$ the set of breaking vertices.
\end{defn}

\begin{lemma}\label{lem:sing-orbits}
For every vertex $v\in E^\sing$ we have one of the following possibilities.
\begin{enumerate}
  \item The set $\{e\in E^1: r(e)=v,\ s(e)\ge v\}$ is empty. Then $\MT(v)\in\M_\gamma(E)$, and if $\MT(v)=\MT(x)$ for some $x\in E^\sing\cup E^\infty$, then $x=v$.
  \item The set $\{e\in E^1: r(e)=v,\ s(e)\ge v\}$ is infinite. Then $\MT(v)\in\M_\gamma(E)$ and $\overline{[v]}=\overline{[x]}$ for some $x\in E^\infty\cap A(\sigma_E)$.
  \item The vertex $v$ is a breaking vertex and $\MT(v)\notin\M_\gamma(E)$. Then $\MT(v)\ne \MT(x)$ for all $x\in E^\infty\cap A(\sigma_E)$, and if $\overline{[v]}=\overline{[u]}$ for some $u\in E^\sing \cup E^{\infty}$, then $u=v$.
  \item The vertex $v$ is a breaking vertex and $\MT(v)\in\M_\gamma(E)$. Then $\MT(v)=\MT(x)$ for some $x\in E^\infty\cap A(\sigma_E)$. For every such $x$ we have $\overline{[v]}\ne\overline{[x]}$, and if $\overline{[v]}=\overline{[u]}$ for some $u\in E^\sing$, then $u=v$.
\end{enumerate}
\end{lemma}

\bp (1) Assume the set $\{e\in E^1: r(e)=v,\ s(e)\ge v\}$ is empty. Since $v$ is a least element in~$\MT(v)$, the only possibility for a simple cycle in $\MT(v)$ not to have an entrance in $\MT(v)$ is to pass through~$v$. But there is no such cycle by our assumption, so $\MT(v)\in\M_\gamma(E)$. Since $v$ is the only finite path in~$\MT(v)$ with range $v$, if $\MT(v)=\MT(x)$ for some $x\in E^\sing\cup E^\infty$, we must have $x=v$.

\smallskip

(2) Assume the set $\{e\in E^1: r(e)=v,\ s(e)\ge v\}$ is infinite. Let $(e_n)_n$ be a sequence of different elements in this set. We can then find cycles of the form $\alpha_n=e_n\alpha_n'$. Consider the infinite aperiodic path $x:=\alpha_1\alpha_2\cdots$. As $\alpha_n\to v$ in $\partial E$, we see that $v\in\overline{[x]}$. As $\alpha_1\cdots\alpha_n\to x$, we also have $x\in\overline{[v]}$, so $\overline{[v]}=\overline{[x]}$. Since $x$ is aperiodic, we then have $\MT(v)=\MT(x)\in\M_\gamma(E)$ by Lemma~\ref{lem:tails}.

\smallskip

(3) Assume that $v$ is a breaking vertex and $\MT(v)\notin\M_\gamma(E)$. For every $x\in E^\infty\cap A(\sigma_E)$, we have $\MT(x)\in\M_\gamma(E)$ by Lemma~\ref{lem:tails}, hence $\MT(x)\ne \MT(v)$. Since $v$ is a breaking vertex, there is no sequence of paths of length $\ge1$ in $\MT(v)$ converging to $v$. It follows that if $v\in\overline{[u]}$ for some $u\in E^\sing$ with $\MT(u)=\MT(v)$, then we must have $v\in[u]$ and hence $u=v$.

\smallskip

(4) Finally, assume that $v$ is a breaking vertex and $\MT(v)\in\M_\gamma(E)$. Since $v$ is a breaking vertex, there is a simple cycle $\alpha$ starting and ending at $v$. As $\MT(v)\in\M_\gamma(E)$, there must exist another (possibly nonsimple) cycle $\beta$ starting and ending at $v$ that is not a power of $\alpha$. By concatenating these two cycles we can construct an infinite aperiodic path $x$ with $\MT(x)=\MT(v)$. The same argument as in (3) shows that if $v\in\overline{[y]}$ for some $y\in E^\sing\cup E^\infty$ with $\MT(y)=\MT(v)$, then $y=v$, completing the proof of the lemma.
\ep

We are now ready to describe the quasi-orbits of aperiodic paths.

\begin{prop}
For a countable directed graph $E$, consider the quasi-orbit space $(\G_E\backslash A(\sigma_E))^\sim$, that is, two points $x,y\in A(\sigma_E)$ define the same point of this space if and only if $\overline{[x]}=\overline{[y]}$. Then there is a unique bijection
$$
\M_\gamma(E)\sqcup BV(E)\to (\G_E\backslash A(\sigma_E))^\sim
$$
satisfying the following properties:
\begin{enumerate}
  \item[(i)] if $M\in\M_\gamma(E)$ has a unique least element $v$ and this vertex does not have self-loops, then the corresponding quasi-orbit is represented by $x_M:=v$;
  \item[(ii)] if $M\in\M_\gamma(E)$ does not have a unique least element without self-loops, then the corresponding quasi-orbit is represented by any path $x_M\in E^\infty\cap A(\sigma_E)$ such that $\MT(x_M)=M$ (and such a path indeed exists);
  \item[(iii)] the quasi-orbit corresponding to $v\in BV(E)$ is represented by $v$.
\end{enumerate}
\end{prop}

\bp
Denote by $E_0^\sing$ (resp., $E_\infty^\sing$) the set of vertices $v\in E^\sing$ such that the set $\{e\in E^1: r(e)=v,\ s(e)\ge v\}$ is empty (resp., infinite). Therefore $E^\sing$ is the disjoint union of the sets $E_0^\sing$, $E_\infty^\sing$ and $BV(E)$.

Observe that if $v\in E^0$ is a least element of a maximal tail $M\in\M(E)$, then the condition $\{e\in E^1: r(e)=v,\ s(e)\ge v\}=\emptyset$ means exactly that $v$ is a unique least element of~$M$ and~$v$ does not have self-loops. If it is satisfied, then $v\in E^\sing$ by property (ii) of maximal tails. Denote by $\M_0(E)$ the set of maximal tails $M$ such that there is a unique least element $v\in M$ and $v$ does not have self-loops. The observation implies that the map $E_0^\sing\to\M_0(E)$, $v\mapsto\MT(v)$, is a bijection. Note also that,  by Lemma~\ref{lem:sing-orbits}(1), the sets $\M_0(E)$ and $\MT(A(\sigma_E)\cap E^\infty)$ are disjoint and $\M_0(E)\subset\M_\gamma(E)$.

Let $p\colon A(\sigma_E)\to (\G_E\backslash A(\sigma_E))^\sim$ be the quotient map. Lemma~\ref{lem:sing-orbits} implies that the map $p$ is injective on $E_0^\sing$ and $BV(E)$, and the space $(\G_E\backslash A(\sigma_E))^\sim$ decomposes into the disjoint union of the sets $p(E_0^\sing)$, $p(BV(E))$ and $p(A(\sigma_E)\cap E^\infty)$. Therefore to finish the proof it suffices to show that the map $p(x)\mapsto\MT(x)$ is a well-defined bijection between $p(A(\sigma_E)\cap E^\infty)$ and $\M_\gamma(E)\setminus\M_0(E)$.

That this map is a well-defined injection follows from Lemma~\ref{lem:tails}. The same lemma implies that every $M\in \M_\gamma(E)\setminus\M_0(E)$ has the form $\MT(x)$ for some $x\in (A(\sigma_E)\cap E^\infty)\cup E_\infty^\sing\cup BV(E)$. But then
 by Lemma~\ref{lem:sing-orbits}(2)-(4) we can always find $x\in A(\sigma_E)\cap E^\infty$ with $M=\MT(x)$.
\ep

Using the elements $x_M$ ($M\in\M_\gamma(E)$), $v\in BV(E)$ and $x_L$ ($L\in\LL(E)$), we can now formulate Proposition~\ref{prop:graph} as follows.

\begin{thm}[cf.~{\citelist{\cite{HS}*{Theorem~3.4}\cite{Gabe}*{Theorem~1}}}]  \label{thm:graph}
For every countable directed graph $E$, we have a bijection
$$
\M_\gamma(E)\sqcup BV(E)\sqcup(\LL(E)\times\T)\to \Prim C^{*}(\G_E)
$$
such that $\M_\gamma(E)\ni M\mapsto\ker\pi_{(x_M,1)}$, $BV(E)\ni v\mapsto \ker\pi_{(v,1)}$, $\LL(E)\times\T\ni (L,w)\mapsto \ker\pi_{(x_L,z)}$, where $z\in\T$ is any $|L|$-th root of $w$. The topology on  $\Prim C^{*}(\G_E)$ is described as follows. Consider a sequence of elements $((x_n,z_n))_n$ and an element $(x,z)$, each of the form $(x_M,1)$, $(v,1)$ or $(x_L,z')$. Then we have:
\begin{enumerate}
\item[(i)] if $x=x_M$ ($M\in\M_\gamma(E)$) or $x=v\in BV(E)$, then $\ker \pi_{(x_n,z_n)} \to \ker \pi_{(x,z)}$ if and only if there exist $y_{n}\in [x_{n}]$ with $y_{n} \to x$;
\item[(ii)] if $x=x_L$ ($L\in\LL(E)$) and $x_{n}\ne x_L$ for all $n$, then $\ker \pi_{(x_n,z_n)} \to \ker \pi_{(x,z)}$ if and only if there exist $y_{n}\in [x_{n}]$ with $y_{n} \to x$;
\item[(iii)] if $x=x_L$ ($L\in\LL(E)$) and $x_{n}=x_L$ for all $n$, then $\ker \pi_{(x_n,z_n)} \to \ker \pi_{(x,z)}$ if and only if  $z_n^{|L|}\to z^{|L|}$.
\end{enumerate}
\end{thm}

The convergences $y_n\to x$ in this theorem can be easily formulated in terms of the graph using the definition of the topology on $\partial E$, but the whole list of rules is long and hardly more illuminating than the above formulation, cf.~\citelist{\cite{HS}\cite{Gabe}}, so we omit it. For example, given $v\in BV(E)$ and a sequence $(L_n)_n$ in $\LL(E)$, elements $y_n\in[x_{L_n}]$ such that $y_n\to v$ exist if and only if for all $n$ large enough we can find finite paths $\alpha_n=e_n\alpha'_n$, with $e_n\in E^1$, such that $r(e_n)=v$, $s(\alpha_n)\in L_n$ and every edge appears in the sequence $(e_n)_n$ at most finitely many times. See also the next subsection for the case of row-finite graphs without sources.

\subsection{Higher rank graphs}\label{subsecHRG}
We next turn to higher rank graphs. Proofs of the claims in the following introductory discussion can be found in \cite{KP}.

A countable higher rank graph is a pair $(\Lambda, d)$, where $\Lambda$ is a countable category, thought of as a countable set of morphisms, and $d\colon\Lambda\mapsto \Z^{k}_+$ is a functor, called the degree map, such that whenever $d(\lambda) = m + n$ for some morphism $\lambda$ and $m,n\in\Z^k_+$, there is a unique factorization $\lambda=\mu\nu$ such that $d(\mu) = m$ and $d(\nu) = n$. The number $k\in\N$ is called the rank of $(\Lambda,d)$, and $(\Lambda,d)$ is also called a $k$-graph. We will follow the standard notation for higher rank graphs and set $\Lambda^{n}:=d^{-1}(n)$ for $n\in \Z^{k}_+$. The unique factorization property implies that $\Lambda^0$ is exactly the set of identity morphisms of objects in our category. Its elements are called vertices, and the elements of $\Lambda^n$ are called paths of degree $n$.

Denote by $r,s\colon\Lambda\to\Lambda^0$ the codomain and domain maps, respectively. For $v\in \Lambda^{0}$, consider the set $v\Lambda^n$ of compositions of $v$ with elements of $\Lambda^n$, in other words, $v\Lambda^{n}=\{\lambda \in \Lambda^n : r(\lambda)=v\}$. We will assume throughout this subsection that our higher rank graphs are row-finite and have no sources, which means that $0<|v\Lambda^{n}|<\infty$ for all $v\in \Lambda^{0}$ and~$n\in \Z_+^{k}$.

The pair
$$
\Omega_{k}:=\{ (p,q)\in \Z_+^{k} \times \Z_+^{k} : p\leq q \}, \quad d: \Omega_{k} \to \Z_+^{k},\quad (p,q)\mapsto q-p,
$$
is a higher rank graph with composition $(p,q)(q,t):=(p,t)$. The space of infinite paths in~$\Lambda$ is defined~by
$$
\Lambda^{\infty}:=\{ x: \Omega_{k} \to \Lambda \; | \; x \text{ is a $k$-graph morphism} \},
$$
where by a morphism one means a functor respecting the degree maps. We define $x(n):=x(n,n)\in \Lambda^{0}$ for $n\in \Z_+^{k}$.

Note that for every finite path $\lambda$ we have well-defined paths $\lambda(p,q)$ for $0\le p\le q\le d(\lambda)$ uniquely determined by the properties that $d(\lambda(p,q))=q-p$ and $\lambda=\lambda(0,p)\lambda(p,q)\lambda(q,d(\lambda))$. In the same way as for infinite paths we write $\lambda(n)$ for $\lambda(n,n)$.

For $\lambda \in \Lambda$ and $x\in \Lambda^{\infty}$ with $s(\lambda) = x(0)$, there is a unique way to define a concatenation $\lambda x \in \Lambda^{\infty}$ such that $(\lambda x)(0,d(\lambda)+n)=\lambda x(0,n)$ for all $n\in\Z^k_+$. For each $\lambda \in \Lambda$, define
$$
Z(\lambda):=\{ x\in \Lambda^{\infty} :  x(0, d(\lambda))=\lambda \}=\{\lambda x : x\in \Lambda^{\infty},\ s(\lambda)=x(0)\}.
$$
The sets $Z(\lambda)$ constitute a basis of compact open sets for a second countable Hausdorff locally compact topology on $\Lambda^{\infty}$.

For every $x\in \Lambda^{\infty}$ and $n\in \Z_+^{k}$, there exists a unique path $\sigma^{n}(x) \in \Lambda^{\infty}$ such that
$$
\sigma^{n}(x)(p,q)=x(n+p,n+q)
$$
for all $(p,q)\in\Omega_k$.  This way we get an action $\Z_+^k\curvearrowright^\sigma \Lambda^\infty$ by local homeomorphisms as in Section~\ref{subsecRD}. We can therefore consider the corresponding Deaconu--Renault groupoid $\G_\Lambda:=\G_{\sigma}$. The groupoid C$^{*}$-algebra $C^{*}(\G_\Lambda)$ is then the Cuntz--Krieger C$^{*}$-algebra~$C^{*}(\Lambda)$ of the higher rank graph $\Lambda$. We let $\Phi\colon\G_\Lambda\to \Z^{k}$ denote the natural grading on~$\G_\Lambda$.

\smallskip

Our first goal is to describe the open subsets of $\Prim C^*(\G_\Lambda)$. Recall from Section~\ref{subsecRD} that they are in a bijective correspondence with the $\G_\Lambda$-invariant open subsets of $\Stab(\G_\Lambda)\dach$.

Let us first consider the space $\G_\Lambda^{(0)}=\Lambda^\infty$. For this space the invariant open sets are known. In order to formulate the result recall the following notions.

\begin{defn}[{\cite{RSY}}]
A subset $H\subset\Lambda^0$ is called \emph{hereditary}, if for every $\lambda \in \Lambda$ with $r(\lambda)\in H$ we have $s(\lambda)\in H$. It is called \emph{saturated}, if whenever $s(v\Lambda^{n})\subset H$ for some $n\in \Z_+^{k}$ and $v\in \Lambda^{0}$, we must have $v\in H$.
\end{defn}

The following observation goes back to \cite{RSY}*{Theorem 5.2}, although it is formulated without using groupoids there, see also \cite{BCS}*{Lemmas~11.3, 11.5}.

\begin{lemma}\label{lem:RSY-sets}
There is a one-to-one correspondence between the $\G_\Lambda$-invariant open subsets of~$\Lambda^\infty$ and the hereditary and saturated subsets of $\Lambda^0$. Namely, given a $\G_\Lambda$-invariant open subset $\Omega\subset\Lambda^\infty$, we define
$$
H_\Omega:=\{v\in\Lambda^0: Z(v)\subset \Omega\},
$$
and given a hereditary and saturated subset $H$ of $\Lambda^0$, we define
$$
\Omega_H:=\{x\in\Lambda^\infty: x(n)\in H\ \text{for some}\ n\in\Z_+^k\}.
$$
Then the maps $\Omega\mapsto H_\Omega$ and $H\mapsto \Omega_H$ are inverse to each other.
\end{lemma}

The following result is inspired by \cite{BCS}*{Corollary~11.7} that gives a description of ideals of $C^*(\Lambda)$ for row-finite higher rank graphs without sources under the assumption of existence of harmonious families of bisections (see Section \ref{ssec:harmonious}).

\begin{prop} \label{prop:HR}
Let $\Lambda$ be a countable row-finite $k$-graph without sources. Then there is a bijective correspondence between the $\G_\Lambda$-invariant  open subsets of $\Stab(\G_\Lambda)\dach$ and the subsets $D\subset \Lambda^{0}\times \T^{k}$ satisfying the following conditions:
\begin{enumerate}
\item[(i)] for every $z\in\T^k$, the set $\{v\in\Lambda^0:(v,z)\in D\}$ is hereditary and saturated;
\item[(ii)] for every $(v, z)\in D$ and every $x\in Z(v)$, there exist $\varepsilon >0$, $n\in \Z_+^{k}$ and a finite set $\{(m(i),n(i))\}_{i=1}^{N} \subset \Z_+^{k} \times \Z_+^{k}$, with $\sigma^{m(i)}(x)=\sigma^{n(i)}(x)$ for all $i$, such that the following property holds: for each $y\in Z(x(0,n))$, there exists $m\in\Z^k_+$ such that $\{y(m)\}\times T_y\subset D$, where
\begin{equation}\label{eq:Ty}
T_y:=\{ w \in \T^{k} :|z^{m(i)-n(i)}-w^{m(i)-n(i)}|<\varepsilon\text{ for all } i \text{ with } \sigma^{m(i)}(y)=\sigma^{n(i)}(y)\}.
\end{equation}
\end{enumerate}
Namely, the pre-image $q^{-1}(U_D)\subset\Lambda^\infty\times\T^k$ of the $\G_\Lambda$-invariant open subset $U_D\subset \Stab(\G_\Lambda)\dach$ corresponding to~$D$, where $q\colon \Lambda^\infty\times\T^k\to\Stab(\G_\Lambda)\dach$ is given by~\eqref{eq:q-on-Stab}, is
\begin{equation}\label{eq:UD}
q^{-1}(U_D)=\{(x,z)\in\Lambda^\infty\times\T^k: (x(n),z)\in D\ \text{for some}\ n\in\Z_+^k\}.
\end{equation}
\end{prop}

As in Theorem~\ref{thm:RDopen}, the convention here is that if $N=0$ or $\sigma^{m(i)}(y)\ne\sigma^{n(i)}(y)$ for all $i$, then $T_y=\T^k$.

\begin{proof}
Denote by $V_D$ the set on the right hand side of~\eqref{eq:UD}. By Lemma~\ref{lem:RSY-sets}, the map $D\mapsto V_D$ establishes a bijective correspondence between the subsets $D\subset\Lambda^0\times\T^k$ satisfying (i) and the subsets $V\subset\Lambda^\infty\times\T^k$ such that $V$ is $\G_\Lambda$-invariant and $V\cap(\Lambda^\infty\times\{z\})$ is open in $\Lambda^\infty\times\{z\}$ for all~$z$. Since the last condition is satisfied for the pre-image of every open subset of $\Stab(\G_\Lambda)\dach$, in order to prove the proposition we only need to show that given a set $D\subset\Lambda^0\times\T^k$ satisfying~(i), condition (ii) is satisfied if and only if $q^{-1}(q(V_D))=V_D$ and the set $q(V_D)\subset\Stab(\G_\Lambda)\dach$ is open.

\smallskip

Assume first that $D\subset\Lambda^0\times\T^k$ satisfies (i), $q^{-1}(q(V_D))=V_D$ and the set $q(V_D)\subset\Stab(\G_\Lambda)\dach$ is open. Fix $(v,z)\in D$ and $x\in Z(v)$.
By Theorem~\ref{thm:RDopen}, there exist $\varepsilon>0$, an open neighbourhood $U\subset \Lambda^{\infty}$ of $x$ and a finite set $\{(m(i),n(i))\}_{i=1}^{N} \subset \Z_+^{k} \times \Z_+^{k}$ with $\sigma^{m(i)}(x) = \sigma^{n(i)}(x)$ for all $i$, such that $(y,w)\in V_D$ for all $y\in U$ and $w\in T_y$, where $T_y$ is defined by~\eqref{eq:Ty}. By replacing~$\eps$ by a smaller number we may assume that we actually have $(y,w)\in V_D$ for all $y\in U$ and $w\in\bar T_y$.

Choose $n\in \Z_+^{k}$ such that $Z(x(0,n))\subset U$. We then claim that condition (ii) is satisfied with this choice of $\eps$, $n$ and $\{(n(i),m(i))\}_{i=1}^{N}$. To prove this, take $y\in Z(x(0,n))$. Since $V_D=q^{-1}(q(V_D))$ is open in the product topology, for each $w \in \bar T_{y}$ there is $m_{w}\in\Z_+^k$ and an open neighbourhood  $U_{w} \subset \T^{k}$ of $w$ such that $Z(y(0,m_{w})) \times U_{w} \subset V_D$. Since  $\bar T_{y}$ is compact, there is a finite set $w_{1}, \dots , w_{p}\in \bar T_y$ such that $\bar T_{y}$ is contained in $U_{w_{1}}\cup \dots \cup U_{w_{p}}$. Take $m\in\Z_+^k$ such that $m\geq m_{w_{j}}$ for all $j=1,\dots,p$. Then $Z(y(0,m))\times U_{w_{j}} \subset V_D$ for all $j$, and hence $Z(y(0,m))\times\{w\}\subset V_D$ for all $w\in T_y$. By $\G_\Lambda$-invariance of $V_D$ this implies that $Z(y(m))\times\{w\}\subset V_D$, hence $(y(m),w)\in D$ by Lemma~\ref{lem:RSY-sets}. Therefore condition (ii) is satisfied.

\smallskip

Next, assume $D\subset\Lambda^0\times\T^k$ satisfies conditions (i) and (ii). We want to apply Theorem~\ref{thm:RDopen} to conclude that $q^{-1}(q(V_D))=V_D$ and the set $q(V_D)\subset\Stab(\G_\Lambda)\dach$ is open. Since we already know that the set $V_D$ is invariant, we only need to check condition (ii) in that theorem.
For this, take $(x,z) \in V_D$ and pick $p\in \Z_+^{k}$ such that $(x(p),z) \in D$.

Apply condition (ii) on $D$ to $v:=x(p)$ and $\sigma^p(x)$ in place of $x$ to get $\varepsilon>0$, $n\in \Z_+^{k}$ and a finite set $\{(m(i),n(i))\}_{i=1}^{N} \subset \Z_+^{k} \times \Z_+^{k}$ with the properties as stated there. We claim that condition (ii) in Theorem~\ref{thm:RDopen} is satisfied for our $\eps$, $U:=Z(x(0,p+n))$ and the finite set $\{(p+m(i),p+n(i))\}_{i=1}^{N} \subset \Z_+^{k} \times \Z_+^{k}$. To see this, assume $y\in Z(x(0,p+n))$ and $w\in \T^k$ satisfy
$$
|z^{m(i)-n(i)}-w^{m(i)-n(i)}|<\varepsilon\text{ for all } i \text{ with } \sigma^{p+m(i)}(y)=\sigma^{p+n(i)}(y).
$$
Since $\sigma^{p}(y)\in Z(x(p,p+n))$, by condition (ii) on $D$ there is $m\in\Z^k_+$ such that $(\sigma^p(y)(m),w)\in D$. As $\sigma^p(y)(m)=y(p+m)$, we conclude that $(y,w)\in V_D$. This proves that condition (ii) in Theorem~\ref{thm:RDopen} is satisfied for~$V_D$.
\end{proof}

By the last part of Theorem~\ref{thm:RDopen} we thus get a bijective correspondence between the ideals of $C^*(\G_\Lambda)\cong C^*(\Lambda)$ and the subsets $D\subset\Lambda^0\times\T^k$ satisfying conditions (i) and (ii) of the above proposition. Namely, the ideal corresponding to $D$ is
$$
\bigcap_{(x,z)\in(\Lambda^\infty\times\T^k)\setminus q^{-1}(U_D)}\ker\pi_{(x,z)}.
$$

\begin{remark}
Instead of Theorem~\ref{thm:RDopen} in the above proof we could have used its variant indicated in Remark~\ref{rem:Yy}. Then we would get that in condition (ii) of Proposition~\ref{prop:HR} we may in addition require that $m(i)-n(i)\in \Phi(\IsoGx{x}^\circ_x)$ for all $i$ and in the definition of $T_y$ we may consider only~$i$ such that $m(i)-n(i)\in\Phi(\IsoGx{y}^\circ_y)$, where $\G:=\G_\Lambda$. Similarly to our discussion in Section~\ref{ssec:harmonious}, a combination of these two forms of Proposition~\ref{prop:HR} implies \cite{BCS}*{Corollary~11.7} when $\G_\Lambda$ admits harmonious families of bisections. \ee
\end{remark}

Our next goal is to give an alternative description of $\Prim C^*(\Lambda)$ similar to Theorem~\ref{thm:graph}. We need some preparation in order to formulate and prove the result.

\smallskip

Denote by $e_1,\dots,e_k$ the standard generators of $\Z^k_+$. The notion of a maximal tail from Definition~\ref{def:max-tails} has the following analogue for row-finite $k$-graphs without sources.

\begin{defn}[\cite{MR3189779}]
A nonempty subset $M\subset\Lambda^0$ is called a \emph{maximal tail} if the following three conditions are satisfied:
\begin{enumerate}
  \item[(i)] if $v\in \Lambda^0$, $w\in M$ and $v\Lambda w\ne\emptyset$, then $v\in M$;
  \item[(ii)] if $v\in M$ and $1\le i\le k$, then there is $\lambda\in v\Lambda^{e_i}$ such that $s(\lambda)\in M$;
  \item[(iii)] for every $v, w\in M$, there exists $u\in M$ such that $v\Lambda u\ne\emptyset$ and $w\Lambda u\ne\emptyset$.
\end{enumerate}
Denote by $\M(\Lambda)$ the set of maximal tails.
\end{defn}

The following result is similar to Lemma~\ref{lem:tails}, but is easier, because now we do not have singular vertices and do not care about aperiodicity of paths, so we omit the proof.

\begin{lemma}\label{lem:k-tails}
We have a well-defined map $\MT\colon \Lambda^\infty\to \M(\Lambda)$ that associates to $x\in\Lambda^\infty$ the set $\MT(x):=\{y(0):y\in[x]\}$. This map is surjective, and we have $\MT(x)=\MT(y)$ if and only if $\overline{[x]}=\overline{[y]}$.
\end{lemma}

We thus get a bijection between $(\G_\Lambda\backslash\Lambda^\infty)^\sim$ and $\M(\Lambda)$. Next, we need to describe the essential isotropy of $\G_\Lambda$.

Given a maximal tail $M$, the set $\Lambda M=M\Lambda M$ is itself a row-finite $k$-graph without sources. We can consider the corresponding set $(\Lambda M)^\infty\subset\Lambda^\infty$ of infinite paths. Thus, $x\in\Lambda^\infty$ lies in $(\Lambda M)^\infty$ if and only if $x(n)\in M$ for all (equivalently, for arbitrary large) $n\in\Z^k_+$.

\begin{defn}[\cite{MR3150171}]
Given a maximal tail $M\subset\Lambda^0$, define an equivalence relation on finite paths in $\Lambda M$ by
$$
\mu\sim_M\nu\quad\text{iff}\quad s(\mu)=s(\nu)\ \ \text{and}\ \ \mu x=\nu x\ \ \text{for all}\ \ x\in (\Lambda M)^\infty\ \ \text{with}\ \ x(0)=s(\mu).
$$
Then define
$$
\Per(\Lambda M):=\{d(\mu)-d(\nu)\mid \mu,\nu\in\Lambda M,\ \mu\sim_M\nu\}\subset\Z^k.
$$
\end{defn}

Note that the relation $\sim_M$ can be defined using only finite paths. Namely, $\mu\sim_M\nu$ if and only if $s(\mu)=s(\nu)$ and $(\mu\lambda)(0,n)=(\nu\lambda)(0,n)$ for all $\lambda\in s(\mu)\Lambda M$ and $n\in\Z^k_+$ such that $n\le d(\mu)+d(\lambda)$ and $n\le d(\nu)+d(\lambda)$.

\smallskip

The following lemma is suggested by results in \citelist{\cite{MR3150171}\cite{SW}}, but it seems it hasn't been formulated explicitly.

\begin{lemma}\label{lem:per}
Assume $M\subset\Lambda^0$ is a maximal tail and $x\in\Lambda^\infty$ is a path such that $\MT(x)=M$. Consider $\G:=\G_\Lambda$ and the standard grading $\Phi\colon\G\to\Z^k$.
Then $\Phi(\IsoGx{x}^\circ_x)=\Per(\Lambda M)$.
\end{lemma}

\bp
We may assume without loss of generality that $\Lambda^0=M$.

Assume first that $\mu\sim_{\Lambda^0}\nu$. We can find $y\in[x]$ such that $y(0,d(\mu))=\mu$. Take any $y'=\mu x'\in Z(\mu)$. By the definition of the relation $\sim_{\Lambda^0}$ we get $y'=\nu x'$. Thus, $\sigma^{d(\mu)}(y')=x'=\sigma^{d(\nu)}(y')$. This shows that $(y,d(\mu)-d(\nu),y)\in\IsoG^\circ$, hence
$$
d(\mu)-d(\nu)\in\Phi(\IsoG^\circ_y)=\Phi(\IsoG^\circ_x),
$$
proving that $\Per(\Lambda)\subset \Phi(\IsoG^\circ_x)$.

For the opposite inclusion, take any $n\in\Phi(\IsoG^\circ_x)$. Then we can find $m,p,q\in\Z^k_+$ such that $n=p-q$ and for all $y\in Z(x(0,m))$ we have $\sigma^p(y)=\sigma^q(y)$. Let $N:=m\vee p\vee q$, $\mu:=x(q,N)$, $\nu:=x(p,N)$ and $v:=x(N)=s(\mu)=s(\nu)$. Take any $x'\in \Lambda^\infty$ with $x'(0)=v$. Then $y:=x(0,N)x'\in Z(x(0,m))$, hence
$$
\mu x'=x(q,N)x'=\sigma^q(y)=\sigma^p(y)=x(p,N)x'=\nu x'.
$$
This shows that $\mu\sim_{\Lambda^0}\nu$. As $d(\mu)=N-q$ and $d(\nu)=N-p$, we conclude that $n=d(\mu)-d(\nu)\in\Per(\Lambda)$.
\ep

We can now give a preliminary version of an analogue of Theorem~\ref{thm:graph}.

\begin{prop} \label{prop:k-graph}
Let $\Lambda$ be a countable row-finite $k$-graph without sources. Then $\Prim C^*(\Lambda)$ can be identified with the set of pairs $(M,\chi)$, where $M\in\M(\Lambda)$ and $\chi\in\Per(\Lambda M)\dach$. Namely, the primitive ideal corresponding to $(M,\chi)$ is $\ker\pi_{(x,z)}$, where $x\in\Lambda^\infty$ is any path with $\MT(x)=M$ and $z\in\T^k$ is any character such that $z|_{\Per(\Lambda M)}=\chi$.

Under this identification the topology on $\Prim C^*(\Lambda)$ is described as follows. Assume we are given $(M,\chi)$ and $(M_n,\chi_n)$ ($n\in\N$) as above. Fix $x,x_n\in\Lambda^\infty$ and $z,z(n)\in\T^k$ such that $\MT(x)=M$, $\MT(x_n)=M_n$, $z|_{\Per(\Lambda M)}=\chi$ and $z(n)|_{\Per(\Lambda M_n)}=\chi_n$. For every $l\in\Per(\Lambda M)$, choose $(p,q)\in\Z^k_+\times\Z^k_+$ such that $p-q=l$ and $\sigma^p(x)=\sigma^q(x)$, and denote by $\Sigma$ the set of pairs $(p,q)$ we thus get.
Then $(M_n,\chi_n)\to(M,\chi)$ if and only if there exist $y_n\in[x_n]$ such that $y_n\to x$ in~$\Lambda^\infty$ and $z(n)\to z$ along the sets
$$
R_n:=\{p-q: (p,q)\in\Sigma,\ p-q\in\Per(\Lambda M_n),\ \sigma^p(y_n)=\sigma^q(y_n)\}.
$$
\end{prop}

A parameterization of $\Prim C^*(\Lambda)$ by the pairs $(M,\chi)$ is already known from~\cite{MR3150171}, but we will obtain it again using our results.

\bp
The identification of $\Prim C^*(\G)$ with the set of pairs $(M,\chi)$ as in the statement of the proposition follows from Theorem~\ref{thm:gradedgroupoids} and Corollary~\ref{cor:essential-iso1} once we take into account the descriptions of quasi-orbits and essential isotropy provided by Lemmas~\ref{lem:k-tails} and~\ref{lem:per}. The description of convergence follows then from Corollary~\ref{cor:graded-groupoids2}.
\ep

A drawback of this formulation is that it involves infinite paths in a seemingly essential way. What we mean is that it is not difficult to formulate existence of paths $y_n\in[x_n]$ such that $y_n\to x$ in terms of existence of a sequence of finite paths with certain properties, but to define the sets $R_n$ we do need to know the entire infinite paths $y_n$. Our goal is to get a formulation involving only finite paths. For this we will need a result from~\cite{MR3150171}.

\begin{defn}[\cite{MR3150171}]
Given a maximal tail $M\subset\Lambda^0$, denote by $M_{\mathrm{Per}}$ the set of vertices $v\in M$ such that for every $\lambda\in v\Lambda M$ and $m\in\Z^k_+$ with $d(\lambda)-m\in\Per(\Lambda M)$, there exists a finite path~$\mu$ in~$\Lambda M$ such that $d(\mu)=m$ and $\lambda\sim_M\mu$.
\end{defn}

The set $M_{\mathrm{Per}}$ is denoted by $H_{\mathrm{Per}}(\Lambda M)$ in~\cite{MR3150171}. It is proved in \cite{MR3150171}*{Theorem~4.2} that $M_{\mathrm{Per}}$ is a nonempty hereditary (with respect to the $k$-graph $\Lambda M$) subset of $M$, and if $x\in (\Lambda M)^\infty$ is such that $x(0)\in M_{\mathrm{Per}}$ and $p,q\in\Z^k_+$ are such that $p-q\in\Per(\Lambda M)$, then $\sigma^p(x)=\sigma^q(x)$. If $\Per(\Lambda M)=0$, then obviously $M_{\mathrm{Per}}=M$. When $\Per(\Lambda M)\ne0$, the set $M_{\mathrm{Per}}$ is a higher rank analogue of primitive loops we used in the previous subsection.

Since $M_{\mathrm{Per}}\subset M$ is hereditary, for any infinite path $x\in (\Lambda M)^\infty$ with $\MT(x)=M$ we can find $m\in \Z^k_+$ such that $x(n)\in M_{\mathrm{Per}}$ for all $n\ge m$. It follows that if $p,q\in\Z^k_+$ are such that $p-q\in\Per(\Lambda M)$, then to check whether $\sigma^p(x)=\sigma^q(x)$ we need to verify a condition involving only finite paths. Namely, we have the following.

\begin{lemma}\label{lem:period}
If $p-q\in\Per(\Lambda M)$, $x\in (\Lambda M)^\infty$ and $m\in\Z^k_+$ is such that $x(m)\in M_{\mathrm{Per}}$, then $\sigma^p(x)=\sigma^q(x)$ if and only if $x(p,m+p)=x(q,m+q)$.
\end{lemma}

\bp
We have $\sigma^p(x)=\sigma^q(x)$ if and only if $x(p,N+p)=x(q,N+q)$ for arbitrarily large $N\in\Z^k_+$. Since $x(m)\in M_{\mathrm{Per}}$, by \cite{MR3150171}*{Theorem~4.2} we have $x(m+p,N+p)=x(m+q,N+q)$ for all $N\ge m$. On the other hand, for all such $N$ we have $x(p,N+p)=x(p,m+p)x(m+p,N+p)$ and $x(q,N+q)=x(q,m+q)x(m+q,N+q)$. Therefore $x(p,N+p)=x(q,N+q)$ for all $N\ge m$ if and only if $x(p,m+p)=x(q,m+q)$.
\ep

We are now ready to prove our main result on the higher rank graphs.

\begin{thm}\label{thm:k-graph}
Let $\Lambda$ be a countable row-finite $k$-graph without sources. In terms of the identification of $\Prim C^*(\Lambda)$ with the set of pairs $(M,\chi)$ ($M\in\M(\Lambda)$, $\chi\in\Per(\Lambda M)\dach$) given by Proposition~\ref{prop:k-graph}, the topology on $\Prim C^*(\Lambda)$ is described as follows.

Assume we are given $(M,\chi)$ and $(M_n,\chi_n)$ ($n\in\N$). For every $l\in\Per(\Lambda M)$, fix $p(l),q(l)\in\Z^k_+$ such that $p(l)-q(l)=l$. Then $(M_n,\chi_n)\to(M,\chi)$ if and only if for every finite path~$\lambda$ in~$\Lambda M$ with $r(\lambda)\in M_{\mathrm{Per}}$, every $\eps>0$ and every finite set $F\subset\Per(\Lambda M)$ we can find $n_0\in\N$ such that for each $n\ge n_0$ the following property is satisfied: there is a finite path $\mu$ in $\Lambda M_n$ of degree $\ge d(\lambda)$ such that $\mu(0,d(\lambda))=\lambda$, $\mu(m)\in(M_n)_{\mathrm{Per}}$ for some $m\le d(\mu)$ and, for every $l\in F\cap\Per(\Lambda M_n)$, we have $p(l)\vee q(l)\le d(\mu)-m$ and
$$
\text{either}\quad \mu(p(l),m+p(l))\ne\mu(q(l),m+q(l))\quad\text{or}\quad |\chi(l)-\chi_n(l)|<\eps.
$$
\end{thm}

\bp
Let us fix $z, z(n)\in\T^k$ such that $z|_{\Per(\Lambda M)}=\chi$ and $z(n)|_{\mathrm{Per}(\Lambda M_n)}=\chi_n$.

Assume first that $(M_n,\chi_n)\to(M,\chi)$ and fix a finite path~$\lambda$ in~$\Lambda M$ with $r(\lambda)\in M_{\mathrm{Per}}$, $\eps>0$ and a finite set $F\subset\Per(\Lambda M)$. We can find $x\in (\Lambda M)^\infty$ such that $x(0,d(\lambda))=\lambda$ and $\MT(x)=M$. Since $x(0)\in M_{\mathrm{Per}}$, by~\cite{MR3150171}*{Theorem~4.2} we have, for every $l\in\Per(\Lambda M)$, that $\sigma^{p(l)}(x)=\sigma^{q(l)}(x)$. Take any $x_n\in (\Lambda M_n)^\infty$ such that $\MT(x_n)=M_n$. Then by Proposition~\ref{prop:k-graph} applied to $\Sigma:=\{(p(l),q(l))\colon l\in\Per(\Lambda M)\}$ we can find $y_n\in[x_n]$ such that $y_n\to x$ and $z(n)\to z$ along the sets
\begin{equation}\label{eq:k-Rn}
R_n:=\{l\in\Per(\Lambda M)\cap\Per(\Lambda M_n):\sigma^{p(l)}(y_n)=\sigma^{q(l)}(y_n)\}.
\end{equation}
Hence, for each $n\ge n_0$ with $n_0$ large enough, we have $y_n(0,d(\lambda))=\lambda$ and $|\chi(l)-\chi_n(l)|<\eps$ for all $l\in F\cap\Per(\Lambda M_n)$ such that $\sigma^{p(l)}(y_n)=\sigma^{q(l)}(y_n)$. As we observed before Lemma~\ref{lem:period}, since $\MT(y_n)=M_n$, there is $m$ (depending on $n$) such that $y_n(m)\in (M_n)_{\mathrm{Per}}$. By that lemma, for every $l\in\Per(\Lambda M)\cap\Per(\Lambda M_n)$, we then have $\sigma^{p(l)}(y_n)=\sigma^{q(l)}(y_n)$ if and only if $y_n(p(l),m+p(l))=y_n(q(l),m+q(l))$. Hence the property in the formulation of the theorem holds for $\mu:=y_n(0,N)$, where $N\in\Z^k_+$ is any element such that $N\ge d(\lambda)$ and $N\ge p(l)\vee q(l)+m$ for all $l\in F\cap\Per(\Lambda M_n)$.

\smallskip

Conversely, assume we can find paths $\mu$ as in the statement of the theorem. Fix $x\in (\Lambda M)^\infty$ such that $x(0)\in M_{\mathrm{Per}}$ and $\MT(x)=M$. Taking $N\in\Z^k_+$ and applying our assumption to $\lambda:=x(0,N)$, $\eps>0$ and a finite $F\subset\Per(\Lambda M)$, we can find, for each $n\ge n_0$, a finite path~$\mu$ in~$M_n$ and an element $m\le d(\mu)$ with properties as in the statement of the theorem. Let $y\in (\Lambda M_n)^\infty$ be any path such that $y(0,d(\mu))=\mu$ and $\MT(y)=M_n$. By Lemma~\ref{lem:period} we get, for every $l\in\Per(\Lambda M)\cap\Per(\Lambda M_n)$, that $\sigma^{p(l)}(y)=\sigma^{q(l)}(y)$ if and only if $y(p(l),m+p(l))=y(q(l),m+q(l))$. Therefore $y$ has the properties that $y(0,N)=x(0,N)$ and, for every $l\in F\cap\Per(\Lambda M_n)$, either $\sigma^{p(l)}(y)\ne\sigma^{q(l)}(y)$ or $|z^l-z(n)^l|<\eps$. From such paths~$y$ we can then construct, using a standard diagonal argument as in the proof of Corollary~\ref{cor:graded-groupoids2}, a sequence of paths $y_n\in (\Lambda M_n)^\infty$ such that $\MT(y_n)=M_n$, $y_n\to x$ and $z(n)\to z$ along the sets~$R_n$ defined by~\eqref{eq:k-Rn}. Therefore the criterion of convergence in Proposition~\ref{prop:k-graph} is satisfied for $x_n:=y_n$ and $\Sigma:=\{(p(l),q(l))\colon l\in\Per(\Lambda M)\}$. Hence $(M_n,\chi_n)\to (M,\chi)$.
\ep

If we consider usual (rank one) row-finite graphs without sources, then we see that the formulation of this theorem is different from that of Theorem~\ref{thm:graph}, but it is not difficult to deduce one theorem from the other. Namely, let $E$ be such a graph. Then $E$ has no breaking vertices. One can check that if $M\in\M_\gamma(E)$, then $\Per(E M)=0$, while if $M\in\M(E)\setminus\M_\gamma(E)$, then $M$ contains a unique primitive loop $L$, $M_{\mathrm{Per}}=L$ and $\Per(E M)=|L|\Z$.

Assume $(M_n,\chi_n)\to(M,\chi)$ for some $M\in\M(E)\setminus\M_\gamma(E)$ with a primitive loop~$L$. Then using either theorem we get that eventually $L\subset M_n$ and hence $M\subset M_n$. If $M$ is a proper subset  of~$M_n$ for some~$n$, then $M_n\in\M_\gamma(E)$ and therefore $\Per(E M_n)=0$, since otherwise $L$ would be the unique primitive loop in $M_n$ and hence $M=M_n$.

Using Theorem~\ref{thm:k-graph} we can thus conclude that if $M\in \M(E)\setminus\M_\gamma(E)$, then $(M_n,\chi_n)\to(M,\chi)$ if and only if for each $n$ large enough either $M\subsetneq M_n$, or $M_n=M$ and $\chi_n$ is arbitrarily close to $\chi$. It is not difficult to see that this is the same as what Theorem~\ref{thm:graph} gives. On the other hand, if $M\in\M_\gamma(E)$, then both theorems imply that $(M_n,\chi_n)\to(M,\chi)$ if and only if every vertex in~$M$ is eventually contained in $M_n$.

\smallskip

Let us finally remark that in certain cases the Jacobson topology on $\Prim C^*(\Lambda)$ has also been described by Li and Yang in~\cite{MR4283280}. In fact, they worked in a more general setting of higher rank analogues of Exel--Pardo algebras. It is possible to extend Theorem~\ref{thm:k-graph} to this setting and give a description of the Jacobson topology for all self-similar $k$-graph C$^*$-algebras satisfying assumptions~(FV) and~(Cyc) in~\cite{MR4283280}. This will be discussed in detail elsewhere.

\bigskip


\bigskip

\begin{bibdiv}
\begin{biblist}

\bib{BP}{article}{
   author={Baggett, Lawrence},
   author={Packer, Judith},
   title={The primitive ideal space of two-step nilpotent group
   $C^*$-algebras},
   journal={J. Funct. Anal.},
   volume={124},
   date={1994},
   number={2},
   pages={389--426},
   issn={0022-1236},
   review={\MR{1289356}},
   doi={10.1006/jfan.1994.1112},
}

\bib{MR1988256}{article}{
   author={Bates, Teresa},
   author={Hong, Jeong Hee},
   author={Raeburn, Iain},
   author={Szyma\'{n}ski, Wojciech},
   title={The ideal structure of the $C^*$-algebras of infinite graphs},
   journal={Illinois J. Math.},
   volume={46},
   date={2002},
   number={4},
   pages={1159--1176},
   issn={0019-2082},
   review={\MR{1988256}},
}

\bib{BCS}{article}{
   author={Brix, K. A.},
   author={Carlsen, T. M.},
   author={Sims, A.},
   title={Ideal structure of $C^*$-algebras of commuting local homeomorphisms},
   how={preprint},
   date={2023},
   eprint={\href{https://arxiv.org/abs/2303.02313v2}{\texttt{arXiv:2303.02313v2 [math.OA]}}},
}

\bib{BCW}{article}{
   author={Brownlowe, Nathan},
   author={Carlsen, Toke Meier},
   author={Whittaker, Michael F.},
   title={Graph algebras and orbit equivalence},
  journal={Ergod. Theory Dyn. Syst.},
   volume={37},
   number={2},
   date={2017},
   pages={389--417},
  doi={10.1017/etds.2015.52},
}

\bib{MR3150171}{article}{
   author={Carlsen, Toke Meier},
   author={Kang, Sooran},
   author={Shotwell, Jacob},
   author={Sims, Aidan},
   title={The primitive ideals of the Cuntz-Krieger algebra of a row-finite
   higher-rank graph with no sources},
   journal={J. Funct. Anal.},
   volume={266},
   date={2014},
   number={4},
   pages={2570--2589},
   issn={0022-1236},
   review={\MR{3150171}},
   doi={10.1016/j.jfa.2013.08.029},
}

\bib{CS}{article}{
   author={Carlsen, Toke M.},
   author={Sims, Aidan},
   title={On Hong and Szyma\'{n}ski's description of the primitive-ideal
   space of a graph algebra},
   conference={
      title={Operator algebras and applications---the Abel Symposium 2015},
   },
   book={
      series={Abel Symp.},
      volume={12},
      publisher={Springer, [Cham]},
   },
   isbn={978-3-319-39284-4},
   isbn={978-3-319-39286-8},
   date={2017},
   pages={115--132},
   review={\MR{3837593}},
}

\bib{CN2}{article}{
   author={Christensen, Johannes},
   author={Neshveyev, Sergey},
   title={Isotropy fibers of ideals in groupoid $\rm C^*$-algebras},
   journal={Adv. Math.},
   volume={447},
   date={2024},
   pages={Paper No. 109696, 32},
   issn={0001-8708},
   review={\MR{4742724}},
   doi={10.1016/j.aim.2024.109696},
}

\bib{CN4}{article}{
   author={Christensen, Johannes},
   author={Neshveyev, Sergey},
   title={The ideal structure of C$^*$-algebras of \'etale groupoids with isotropy groups of local polynomial growth},
   how={preprint},
   date={2024},
   eprint={\href{https://arxiv.org/abs/2412.11805}{\texttt{2412.11805 [math.OA]}}},
}

\bib{MR1503352}{article}{
   author={Clifford, A. H.},
   title={Representations induced in an invariant subgroup},
   journal={Ann. of Math. (2)},
   volume={38},
   date={1937},
   number={3},
   pages={533--550},
   issn={0003-486X},
   review={\MR{1503352}},
   doi={10.2307/1968599},
}

\bib{D}{article}{
   author={Deaconu, Valentin},
   title={Groupoids associated with endomorphisms},
   journal={Trans. Amer. Math. Soc},
   volume={347},
   date={1995},
  number={5},
   pages={1779--1786},
  review={\MR{1233967}},
   doi={10.2307/2154972},
}

\bib{EE}{article}{
   author={Echterhoff, Siegfried},
   author={Emerson, Heath},
   title={Structure and $K$-theory of crossed products by proper actions},
   journal={Expo. Math.},
   volume={29},
   date={2011},
   number={3},
   pages={300--344},
   issn={0723-0869},
   review={\MR{2820377}},
   doi={10.1016/j.exmath.2011.05.001},
}

\bib{EH}{book}{
   author={Effros, Edward G.},
   author={Hahn, Frank},
   title={Locally compact transformation groups and $C\sp{\ast}$-algebras},
   series={},
   volume={No. 75},
   publisher={American Mathematical Society, Providence, R.I.},
   date={1967},
   pages={92},
   review={\MR{0227310}},
}

\bib{MR0146681}{article}{
   author={Fell, J. M. G.},
   title={The dual spaces of $C\sp{\ast} $-algebras},
   journal={Trans. Amer. Math. Soc.},
   volume={94},
   date={1960},
   pages={365--403},
   issn={0002-9947},
   review={\MR{0146681}},
   doi={10.2307/1993431},
}

\bib{MR0139135}{article}{
   author={Fell, J. M. G.},
   title={A Hausdorff topology for the closed subsets of a locally compact
   non-Hausdorff space},
   journal={Proc. Amer. Math. Soc.},
   volume={13},
   date={1962},
   pages={472--476},
   issn={0002-9939},
   review={\MR{0139135}},
   doi={10.2307/2034964},
}

\bib{Gabe}{article}{
   author={Gabe, James},
   title={Graph $C^*$-algebras with a $T_1$ primitive ideal space},
   conference={
      title={Operator algebra and dynamics},
   },
   book={
      series={Springer Proc. Math. Stat.},
      volume={58},
      publisher={Springer, Heidelberg},
   },
   isbn={978-3-642-39459-1},
   isbn={978-3-642-39458-4},
   date={2013},
   pages={141--156},
   review={\MR{3142035}},
   doi={10.1007/978-3-642-39459-1\_7},
}

\bib{MR0146297}{article}{
   author={Glimm, James},
   title={Families of induced representations},
   journal={Pacific J. Math.},
   volume={12},
   date={1962},
   pages={885--911},
   issn={0030-8730},
   review={\MR{0146297}},
}

\bib{MR2966476}{article}{
   author={Goehle, Geoff},
   title={The Mackey machine for crossed products by regular groupoids. II},
   journal={Rocky Mountain J. Math.},
   volume={42},
   date={2012},
   number={3},
   pages={873--900},
   issn={0035-7596},
   review={\MR{2966476}},
   doi={10.1216/RMJ-2012-42-3-873},
}

\bib{GR}{article}{
   author={Gootman, Elliot C.},
   author={Rosenberg, Jonathan},
   title={The structure of crossed product $C\sp{\ast} $-algebras: a proof
   of the generalized Effros--Hahn conjecture},
   journal={Invent. Math.},
   volume={52},
   date={1979},
   number={3},
   pages={283--298},
   issn={0020-9910},
   review={\MR{0537063}},
   doi={10.1007/BF01389885},
}

\bib{MR0246999}{article}{
   author={Greenleaf, F. P.},
   title={Amenable actions of locally compact groups},
   journal={J. Functional Analysis},
   volume={4},
   date={1969},
   pages={295--315},
   issn={0022-1236},
   review={\MR{0246999}},
   doi={10.1016/0022-1236(69)90016-0},
}

\bib{HS}{article}{
   author={Hong, J. H.},
   author={Szyma\'nski, W.},
   title={The primitive ideal space of the C$^{*}$-algebras of infinite graphs},
   journal={J. Math. Soc. Japan},
   volume={56},
   date={2004},
   number={2},
   pages={45--64},
   issn={1016-443X},
   review={\MR{1911663}},
   doi={10.1007/s00039-002-8249-5},
}

\bib{IW0}{article}{
   author={Ionescu, Marius},
   author={Williams, Dana P.},
   title={Irreducible representations of groupoid $C^*$-algebras},
   journal={Proc. Amer. Math. Soc.},
   volume={137},
   date={2009},
   number={4},
   pages={1323--1332},
   issn={0002-9939},
   review={\MR{2465655}},
   doi={10.1090/S0002-9939-08-09782-7},
}

\bib{IW}{article}{
   author={Ionescu, Marius},
   author={Williams, Dana P.},
   title={The generalized Effros--Hahn conjecture for groupoids},
   journal={Indiana Univ. Math. J.},
   volume={58},
   date={2009},
   number={6},
   pages={2489--2508},
   issn={0022-2518},
   review={\MR{2603756}},
   doi={10.1512/iumj.2009.58.3746},
}

\bib{MR3189779}{article}{
   author={Kang, Sooran},
   author={Pask, David},
   title={Aperiodicity and primitive ideals of row-finite $k$-graphs},
   journal={Internat. J. Math.},
   volume={25},
   date={2014},
   number={3},
   pages={1450022, 25},
   issn={0129-167X},
   review={\MR{3189779}},
   doi={10.1142/S0129167X14500220},
}

\bib{Kat}{article}{
   author={Katsura, Takeshi},
   title={Ideal structure of C$^{*}$-algebras of singly generated dynamical systems},
        how={preprint},
        date={2021},
      eprint={\href{https://arxiv.org/abs/2107.10422}{\texttt{2107.10422 [math.OA]}}},
}

\bib{KP}{article}{
   author={Kumjian, Alex},
   author={Pask, David},
   title={Higher rank graph $C^\ast$-algebras},
   journal={New York J. Math.},
   volume={6},
   date={2000},
   pages={1--20},
   review={\MR{1745529}},
}

\bib{MR4283280}{article}{
   author={Li, Hui},
   author={Yang, Dilian},
   title={The ideal structures of self-similar $k$-graph C*-algebras},
   journal={Ergodic Theory Dynam. Systems},
   volume={41},
   date={2021},
   number={8},
   pages={2480--2507},
   issn={0143-3857},
   review={\MR{4283280}},
   doi={10.1017/etds.2020.52},
}

\bib{MR0031489}{article}{
   author={Mackey, George W.},
   title={Imprimitivity for representations of locally compact groups. I},
   journal={Proc. Nat. Acad. Sci. U.S.A.},
   volume={35},
   date={1949},
   pages={537--545},
   issn={0027-8424},
   review={\MR{0031489}},
   doi={10.1073/pnas.35.9.537},
}

\bib{MR0098328}{article}{
   author={Mackey, George W.},
   title={Unitary representations of group extensions. I},
   journal={Acta Math.},
   volume={99},
   date={1958},
   pages={265--311},
   issn={0001-5962},
   review={\MR{0098328}},
   doi={10.1007/BF02392428},
}

\bib{NS}{article}{
   author={Neshveyev, Sergey},
   author={Schwartz, Gaute},
   title={Non-Hausdorff \'{e}tale groupoids and $C^*$-algebras of left
   cancellative monoids},
   journal={M\"{u}nster J. Math.},
   volume={16},
   date={2023},
   number={1},
   pages={147--175},
   issn={1867-5778},
   review={\MR{4563262}},
}

\bib{NT}{article}{
   author={Neshveyev, Sergey},
   author={Tuset, Lars},
   title={Quantized algebras of functions on homogeneous spaces with Poisson
   stabilizers},
   journal={Comm. Math. Phys.},
   volume={312},
   date={2012},
   number={1},
   pages={223--250},
   issn={0010-3616},
   review={\MR{2914062}},
   doi={10.1007/s00220-012-1455-6},
}

\bib{Ped}{book}{
   author={Pedersen, Gert K.},
   title={$C^*$-algebras and their automorphism groups},
   series={Pure and Applied Mathematics (Amsterdam)},
   edition={2},
   publisher={Academic Press, London},
   date={2018},
   pages={xviii+520},
   isbn={978-0-12-814122-9},
   review={\MR{3839621}},
}

\bib{RSY}{article}{
   author={Raeburn, Iain},
   author={Sims, Aidan},
   author={Yeend, Treent},
   title={Higher-rank graphs and their C$^{*}$-algebras},
   journal={Proc. Edinburgh Math. Soc.},
   volume={46},
   date={2003},
   number={1},
   pages={99--115},
  review={\MR{1961175}},
}

\bib{R}{article}{
   author={Renault, Jean},
   title={The ideal structure of groupoid crossed product $C^\ast$-algebras},
   note={With an appendix by Georges Skandalis},
   journal={J. Operator Theory},
   volume={25},
   date={1991},
   number={1},
   pages={3--36},
   issn={0379-4024},
   review={\MR{1191252}},
}

\bib{MR1770333}{article}{
   author={Renault, Jean},
   title={Cuntz-like algebras},
   conference={
      title={Operator theoretical methods},
      address={Timi\c{s}oara},
      date={1998},
   },
   book={
      publisher={Theta Found., Bucharest},
   },
   isbn={973-99097-2-8},
   date={2000},
   pages={371--386},
   review={\MR{1770333}},
}

\bib{RW}{article}{
   author={Renault, Jean},
   author={Williams, Dana P.},
   title={Amenability of groupoids arising from partial semigroup actions and topological higher rank graphs},
   journal={Trans. Amer. Math. Soc.},
   volume={369},
   date={2007},
   number={4},
   pages={2255--2283},
   issn={0002-9947},
  review={\MR{3592511}},
}

\bib{Sau}{article}{
   author={Sauvageot, Jean-Luc},
   title={Id\'{e}aux primitifs induits dans les produits crois\'{e}s},
   journal={J. Functional Analysis},
   volume={32},
   date={1979},
   number={3},
   pages={381--392},
   issn={0022-1236},
   review={\MR{0538862}},
   doi={10.1016/0022-1236(79)90047-8},
}

\bib{SSW}{collection}{
   author={Sims, Aidan},
   author={Szab\'{o}, G\'{a}bor},
   author={Williams, Dana},
   title={Operator algebras and dynamics: groupoids, crossed products, and
   Rokhlin dimension},
   series={Advanced Courses in Mathematics. CRM Barcelona},
   editor={Perera, Francesc},
   publisher={Birkh\"{a}user/Springer, Cham},
   date={2020},
   pages={x+163},
   isbn={978-3-030-39712-8},
   isbn={978-3-030-39713-5},
   review={\MR{4321941}},
   doi={10.1007/978-3-030-39713-5},
}

\bib{SW}{article}{
   author={Sims, Aidan},
   author={Williams, Dana P.},
   title={The primitive ideals of some \'etale groupoid C$^{*}$-algebras},
   journal={Algebr. Represent. Theory},
   volume={19},
   date={2016},
   number={2},
   pages={255--276},
   doi={10.1007/s10468-015-9573-4},
}

\bib{MR4395600}{article}{
   author={Van Wyk, Daniel W.},
   author={Williams, Dana P.},
   title={The primitive ideal space of groupoid $C^*$-algebras for groupoids
   with abelian isotropy},
   journal={Indiana Univ. Math. J.},
   volume={71},
   date={2022},
   number={1},
   pages={359--390},
   issn={0022-2518},
   review={\MR{4395600}},
   doi={10.1512/iumj.2022.71.9523},
}

\bib{MR0617538}{article}{
   author={Williams, Dana P.},
   title={The topology on the primitive ideal space of transformation group
   $C\sp{\ast} $-algebras and C.C.R. transformation group $C\sp{\ast}
   $-algebras},
   journal={Trans. Amer. Math. Soc.},
   volume={266},
   date={1981},
   number={2},
   pages={335--359},
   issn={0002-9947},
   review={\MR{0617538}},
   doi={10.2307/1998427},
}

\end{biblist}
\end{bibdiv}
\end{document}